\documentclass[10pt,a4paper,leqno]{amsart}

\usepackage[dvipsnames]{xcolor}

\usepackage{pifont}
\usepackage{amssymb, amsmath, amscd}
\usepackage[T1]{fontenc}
\usepackage[utf8]{inputenc}
\usepackage{tikz}
\usetikzlibrary{decorations.pathreplacing}

\def\blue{\color{blue}}
\def\black{\color{black}}

\usepackage[hypertexnames=false]{hyperref}
\makeatletter
\@namedef{subjclassname@2020}{%
	\textup{2020} Mathematics Subject Classification}
\makeatother

\DeclareMathOperator{\Id}{\rm Id}

\newtheorem{theorem}{Theorem}[section]

\newtheorem{corollary}[theorem]{Corollary}
\newtheorem{example}[theorem]{Example}

\theoremstyle{definition}

\theoremstyle{remark}
\newtheorem{remark}[theorem]{Remark}

\def\energy{\mathcal{E}}

\def\RTres{\mathbb{R}^3}
\def\HTres{\mathbb{H}^3}
\def\EUnoUno{E(1,1)}
\def\EDos{\widetilde{E}(2)}

\def\Xmenos{-6}
\def\Xmenost{\Xmenos+1.8}
\def\Xmas{6}
\def\Xmast{6-0.4}
\def\XUnCuarto{\Xmast}

\def\XTresOnceavos{\XUnCuarto-0.55}
\def\XCincoMenosDosRaizSiete{\XTresOnceavos-0.55}
\def\XTresDecimos{\XCincoMenosDosRaizSiete-0.55}
\def\XUnTercio{\XTresDecimos-0.55}
\def\XVeintiunCincuentaydosavos{\XUnTercio-0.6}
\def\XSieteDieciseisavos{\XVeintiunCincuentaydosavos-0.55}
\def\XCincoOnceavos{\XSieteDieciseisavos-0.55}
\def\XUnMedio{\XCincoOnceavos-0.55}
\def\XSieteDecimos{\XUnMedio-0.75}
\def\XTresCuartos{\XSieteDecimos-0.55}
\def\XUno{\XTresCuartos-0.8}
\def\XTresMedios{\XUno-1.05}

\def\XTres{\Xmenost+0.7}

\def\XGrupo{\Xmenos+0.5}
\def\XNumeroGrupo{\XGrupo+0.8}
\def\YPrimerGrupo{-0.45}

\title[Homogeneous critical metrics for quadratic curvature functionals]{Four-dimensional homogeneous critical metrics for quadratic curvature functionals}
\author[Brozos-V\'azquez, Caeiro-Oliveira, Garc\'ia-R\'io, Vázquez-Lorenzo]{M. Brozos-V\'azquez, S. Caeiro-Oliveira, E. Garc\'ia-R\'io, R. Vázquez-Lorenzo}
\address{MBV: CITMAga, 15782 Santiago de Compostela, España}	
\address{\phantom{MBV:}
	Universidade da Coru\~na, Campus Industrial de Ferrol,
	 Department of Mathematics,  15403 Ferrol,  Spain}
\email{miguel.brozos.vazquez@udc.gal}
\address{SCO: Department of Mathematics, Universidade de Vigo, Campus Auga, 32004 Ourense, Spain}
\email{sandro.caeiro@uvigo.gal}
\address{EGR: CITMAga, 15782 Santiago de Compostela, España}
\address{\phantom{EGR:}
	Faculty of Mathematics,
	University of Santiago de Compostela,
	15782 Santiago de Compostela, Spain}
\email{eduardo.garcia.rio@usc.es}
\address{RVL: I.E.S de Ribadeo Dionisio Gamallo, 27700 Ribadeo,  Spain}
\email{ravazlor@edu.xunta.gal}
\thanks{Supported by projects PID2019-105138GB-C21(AEI/FEDER, Spain) and ED431C 2019/10, ED431F 2020/04  (Xunta de Galicia, Spain). The second named author S. C.-O. acknowledges support from the Ramón y Cajal grant RYC-2017-22490 (AEI, Spain).}
\date{}

\begin{document}	
\begin{abstract}
	We determine all homogeneous metrics which are critical for some quadratic curvature functional in dimension four.
\end{abstract}

\maketitle

\section{Introduction}
Geometric structures often arise as critical points of geometric functionals. Einstein metrics, which are critical for the Hilbert-Einstein functional constrained to variations with constant volume, are well-known examples and play a distinguished role in mathematics and physics \cite{Besse}. The Hilbert-Einstein functional is defined in terms of the scalar curvature, which generates the space of first-order curvature invariants. From this perspective, it is natural to investigate other functionals which are also defined by Riemannian invariants, such as quadratic curvature functionals \cite{Besse,viaclovsky}.

Describing critical metrics for a given functional is often an unfeasible task. Therefore, we focus on homogeneous manifolds and aim at determining all metrics which are critical for a quadratic curvature functional in dimension four. Although the homogeneous context is very rigid for Einstein metrics, since every four-dimensional homogeneous Einstein manifold is necessarily symmetric \cite{Jensen}, we show that the class of critical metrics for quadratic curvature functionals is very rich and there exist examples exhibiting a variety of features. As a consequence of the provided descriptions, we find an interesting relation between critical metrics with zero energy and Ricci solitons.

\subsection{Critical metrics for quadratic curvature functionals}

The space of scalar quadratic curvature invariants of a Riemannian manifold is generated by $\{\tau^2, \Delta\tau,$ $\|\rho\|^2, \|R\|^2\}$, where $\Delta\tau$ denotes the Laplacian of the scalar curvature, $R$ is the curvature tensor and $\rho$ stands for the Ricci tensor. 

Dimension four is special in the study of quadratic curvature functionals due to the Chern-Gauss-Bonnet Theorem, which states that $\int_M\{\| R\|^2-4\|\rho\|^2+\tau^2\}\operatorname{dvol}_g=8\pi^2\chi(M)$. Thus, every quadratic curvature functional is equivalent to (see \cite{CM, GV1})
\begin{equation}\label{eq:functional}
\mathcal{S}: g\mapsto \mathcal{S}(g)=\int_M\,\tau^2\operatorname{dvol}_g\,,
\quad\text{or}\quad
\mathcal{F}_t: g\mapsto \mathcal{F}_t(g)=\int_M\{\|\rho\|^2+t\,\tau^2\}\operatorname{dvol}_g\,, 
\end{equation}
for some $t\in\mathbb{R}$.

Particular instances of the functionals $\mathcal{F}_t$ have been extensively studied. Four-dimensional critical metrics for the $L^2$-norm of the Weyl tensor (i.e., Bach-flat metrics) are $\mathcal{F}_{-1/3}$-critical. Note that the functional defined by the four-dimensional Branson $Q$-curvature, $Q=\frac{1}{12}\left(-\Delta\tau-3\|\rho\|^2+\tau^2\right)$, or that given by the second symmetric elementary function of the Schouten tensor, $\sigma_2(A)=-\frac{1}{2(n-2)^2}\|\rho\|^2+\frac{n}{8(n-1)(n-2)^2}\tau^2$, are also equivalent to $\mathcal{F}_{-1/3}$ when we fix dimension $n=4$ (see \cite{BO, GV}).
The functional $g\mapsto\int_M\,\| R\|^2\operatorname{dvol}_g$ is equivalent to  $\mathcal{F}_{-1/4}$  (which coincides with the functional defined by the $L^2$-norm of the trace-free Ricci tensor). Other geometric functionals like the Schouten functional (defined by the $L^2$-norm of the Schouten tensor), or the volumal functional (determined by the quadratic invariant $18\Delta\tau-3\| R\|^2+8\|\rho\|^2+5\tau^2$) are equivalent to $\mathcal{F}_t$ for $t= -2/9$ and $t=-2$, respectively (see \cite{Gray, HNS}). The spectral functional, which is determined by the quadratic invariant $2\| R\|^2-2\|\rho\|^2+5\tau^2+12\Delta\tau$ (see \cite{Gray}), is equivalent to $\mathcal{F}_{1/2}$ in dimension four.

The gradients of the functionals in \eqref{eq:functional} follow from the work of Berger \cite{Berger}, which shows that the corresponding Euler-Lagrange equations for these functionals restricted to metrics of volume one in dimension $n$ become
\begin{equation}\label{eq:S}
\nabla^2 \tau -\tfrac{1}{n}\Delta \tau\, g - \tau \left(\rho-\tfrac{1}{n}\tau\, g \right)=0
\end{equation}
for the $\mathcal{S}$-functional, and
\begin{equation}\label{eq:E-L}
-\Delta \rho 
+(2t+1)\nabla^2 \tau 
-\tfrac{2t}{n} (\Delta \tau)g 
+\tfrac{2}{n} (\|\rho\|^2 +t\tau^2) g 
-2 R[\rho] 
-2t\tau\rho=0
\end{equation}
for the $\mathcal{F}_t$-functionals,
where $\nabla^2 \tau$ denotes the Hessian of the scalar curvature and $R[\rho]_{ij}=R_{ikj\ell}\rho^{k\ell}$.  
It immediately follows from \eqref{eq:S} and \eqref{eq:E-L} that Einstein metrics are critical for all quadratic curvature functionals in dimensions three and four. The product $\mathbb{S}^2\times\mathbb{H}^2$ carries a non-Einstein homogeneous metric which is also critical for all quadratic curvature functionals. We refer to \cite{BLMS} for examples of metrics which are critical for all quadratic curvature functionals with non-constant scalar curvature.

An important object that plays a role in the study of variations is the energy of the functional (see, for example, the discussion in \cite{GV, HuLi}). In the homogeneous case, for each functional $\mathcal{F}_t$, it is determined by the quantity $\energy_t=\|\rho\|^2+t\tau^2$. 
In dimension four every Einstein metric is critical with zero energy for the $\mathcal{F}_{-1/4}$-functional.
In contrast, the metric of $\mathbb{S}^2\times\mathbb{H}^2$ is critical for all quadratic curvature functionals, but its energy is $\energy_t=\|\rho\|^2+t\tau^2=4\neq 0$ since its scalar curvature vanishes.

Quadratic curvature functionals are invariant by homotheties in dimension four, and so are their gradients and the critical equations \eqref{eq:S}-\eqref{eq:E-L}.
Hence, along this work, we mostly work at the homothetical level in order to simplify calculations and the statements in the classification results.


\subsection{Homogeneous spaces}
We work in the homogeneous setting. Hence both $\tau$ and $\|\rho\|$ are constant and equations \eqref{eq:S}-\eqref{eq:E-L} reduce to 
\begin{equation}
\label{gradFt}
\tau\left(\rho-\tfrac{1}{n}\tau\, g\right)=0 \,,
\quad\text{and}\quad
-\Delta \rho 
+\tfrac{2}{n} (\|\rho\|^2 +t\tau^2) g 
-2 R[\rho] 
-2t\tau\rho =0 . 
\end{equation}
Since Einstein metrics trivially satisfy these equations for all $t$, we focus in non-Einstein metrics.  
In the case that the energy of the functional $\mathcal{F}_t$ is zero, if $\tau=0$ then $\|\rho\|^2=0$, and the metric is necessarily flat \cite{AK}. Otherwise, if $\energy_t=0$ but $\tau\neq 0$, the value of $t$ is given by $t=-\|\rho\|^2\tau^{-2}$ and the second equation in \eqref{gradFt} further reduces to
\begin{equation}\label{eq:E-L_eq}
-\Delta\rho -2 R[\rho]+2\tau^{-1}\|\rho\|^2 \rho=0,
\end{equation}
which characterizes homogeneous $\mathcal{F}_t$-critical metrics with zero energy.
Furthermore, any Riemannian manifold satisfies $\|\rho\|^2\geq\frac{1}{n}\tau^2$, with equality if and only if the metric is Einstein. Hence, if a non-flat homogeneous metric is $\mathcal{F}_t$-critical with zero energy, then $t=-\|\rho\|^2\tau^{-2}\leq-\frac{1}{n}$.

It directly follows from equations \eqref{eq:S} and \eqref{eq:E-L} that if a metric is critical for two quadratic curvature functionals then it is critical for all. From \eqref{gradFt} we have that a non-Einstein metric is $\mathcal{S}$-critical if and only if $\tau=0$, but an explicit calculation shows that the scalar curvature of  non-symmetric homogeneous Bach-flat metrics, which were given in \cite{CL-GM-GR-GR-VL}, never vanishes. Moreover, it was already shown by Jensen in \cite{Jensen} that homogeneous Einstein metrics are symmetric in dimension four. Hence, we conclude that four-dimensional homogeneous metrics which are critical for all quadratic curvature functionals are symmetric.

\subsubsection{Symmetric spaces}\label{ss:symmetric}
 In dimension four, simply connected homogeneous spaces are either symmetric or locally isometric to a Lie group with a left-invariant metric~\cite{B81}. Considering the possible eigenvalues of the Ricci operator, one has that any four-dimensional symmetric space is $\mathcal{F}_t$-critical for  some  $t\in\mathbb{R}$ as follows.
\begin{enumerate}
	\item[(i)] Einstein metrics are critical for all quadratic curvature functionals. They correspond to real and complex space forms and products $N^2_1(\kappa)\times N^2_2(\kappa)$ of two surfaces of equal constant sectional curvature.
	\item[(ii)] Locally conformally flat products $\mathbb{R}\times N^{3}(\kappa)$, with $\kappa\neq 0$, are $\mathcal{F}_{-1/3}$-critical. 
	\item[(iii)] Products $\mathbb{R}^2\times N^2(\kappa)$, with $\kappa\neq 0$, are $\mathcal{F}_{-1/2}$-critical.
	\item[(iv)] Locally conformally flat products $N^2_1(\kappa)\times N^2_2(-\kappa)$, with $\kappa\neq 0$, are critical for all quadratic curvature functionals.
	\item[(v)] Products $N_1^2(\kappa_1)\times N_2^2(\kappa_2)$, with $\kappa_1^2\neq \kappa_2^2$ and $\kappa_1\kappa_2\neq  0$, are $\mathcal{F}_{-1/2}$-critical.	
\end{enumerate}
Observe that all metrics in cases (i), (ii) and (iii) are critical for some functional $\mathcal{F}_t$ with zero energy, but metrics in cases (iv) and (v) are not. Also, notice that, in the homogeneous four-dimensional context, metrics which are critical for all quadratic curvature functionals are either Einstein or a product $N^2_1(\kappa)\times N^2_2(-\kappa)$ (i.e., homothetic to $\mathbb{S}^2\times\mathbb{H}^2$).

\subsubsection{Four-dimensional Lie groups}\label{ss:4D-groups}

Any real Lie algebra is a semi-direct product of the radical (i.e., the maximal solvable ideal) and a semi-simple subalgebra (the Levi factor) \cite{Knapp}. Hence the classification of four-dimensional Lie algebras reduces to classifying low-dimensional semi-simple Lie algebras, solvable Lie algebras and semi-direct products of semi-simple Lie algebras and solvable ones. Semi-simple Lie algebras can be decomposed into a direct sum of simple subalgebras which are orthogonal with respect to the Killing-Cartan form. Solvable Lie algebras are also classified in low dimensions (up to six). We refer, for example, to \cite{Patera} for a description of all four-dimensional Lie algebras.
Simply connected four-dimensional Lie groups are isomorphic to the products $\widetilde{SL}(2, \mathbb{R})\times \mathbb{R}$ or $SU(2) \times \mathbb{R}$ or, otherwise, they are solvable Lie groups which correspond to semi-direct extensions of the Euclidean and Poincar\'e groups 
$\mathbb{R}\ltimes \widetilde{E}(2)$ and $\mathbb{R}\ltimes E(1,1)$, the Heisenberg group $\mathbb{R}\ltimes \mathcal{H}^3$, or the Abelian group $\mathbb{R}\ltimes \mathbb{R}^3$ \cite{B81}. 

Left-invariant metrics on three-dimensional Lie groups were described by Milnor as follows \cite{Milnor}. Let $(\mathfrak{g},\langle\cdot,\cdot\rangle)$ be a three-dimensional Lie algebra equipped with a positive definite inner product, and let $L$ be the structure operator $[x,y]=L(x\times y)$, where `$\times$' denotes the vector cross product on $(\mathfrak{g},\langle\cdot,\cdot\rangle)$. Then the associated Lie group $G$ is unimodular if and only if $L$ is self-adjoint, in which case there is an orthonormal basis $\{e_1,e_2,e_3\}$ of $L$-eigenvectors so that 
$$
[e_2,e_3]=\lambda_1 e_1,\qquad
[e_3,e_1]=\lambda_2 e_2,\qquad
[e_1,e_2]=\lambda_3 e_3,
$$ 
where $\lambda_1,\lambda_2,\lambda_3$ are the corresponding eigenvalues of $L$.
The Lie algebra is $\mathfrak{su}(2)$ (resp.,  $\mathfrak{sl}(2,\mathbb{R})$) if $L$ is non-singular and eigenvalues $\lambda_k$  do not change sign (resp., change sign). If $\operatorname{ker}L$ is one-dimensional, then the Lie algebra is $\mathfrak{e}(2)$ (resp., $\mathfrak{e}(1,1)$) if the non-zero eigenvalues of $L$ have the same sign (resp., opposite sign). The Heisenberg Lie algebra $\mathfrak{h}^3$ corresponds to the case when $\operatorname{ker}L$ is two-dimensional, and the Abelian Lie algebra $\mathfrak{r}^3$ corresponds to $L=0$.

Non-unimodular three-dimensional Lie algebras are semi-direct extensions of the Abelian Lie algebra $\mathfrak{r}\ltimes\mathfrak{r}^2$ and there exists an orthonormal basis $\{e_1,e_2,e_3\}$ so that
$$
[e_1,e_2]=\alpha e_2+\beta e_3,\qquad
[e_1,e_3]=\gamma e_2+\delta e_3,\qquad
[e_2,e_3]=0,
$$
where $\alpha,\beta,\gamma,\delta\in\mathbb{R}$ and $\alpha\gamma+\beta\delta=0$. Moreover, one may further normalize the structure constants so that $\alpha+\delta=2$ within the same homothety class.
If the derivation determining the semi-direct extension is singular, then the metric Lie group is homothetic to the unimodular Lie group $\widetilde{SL}(2, \mathbb{R})$ equipped with a suitable left-invariant metric or the Heisenberg group (see \cite{TV}). If the derivation is self-adjoint, then the orthonormal basis may be further specialized so that $\beta=\gamma=0$. 

Homogeneous $\mathcal{F}_t$-critical metrics in dimension three were determined in \cite{BV-GR-CO}, obtaining the classification summarized
in Figure~\ref{figura-1} (see Section~\ref{sect:summary} for an explanation of the legend).
The energy of the corresponding functional, $\energy_t=\|\rho\|^2+t\tau^2$, is zero in Einstein metrics for $\mathcal{F}_{-1/3}$, in products $\mathbb{R}\times N(\kappa)$, in the Heisenberg group, in the Poincaré group, and in non-unimodular Lie groups with $\operatorname{ad}_{e_1}$ self-adjoint for a functional $\mathcal{F}_t$ with $t$ a value satisfying $-1<t<-1/3$.
In the other cases, corresponding to $\mathcal{F}_t$-critical metrics on $SU(2)$ (resp., $\widetilde{SL}(2,\mathbb{R})$), the energy is positive (resp., negative).
  
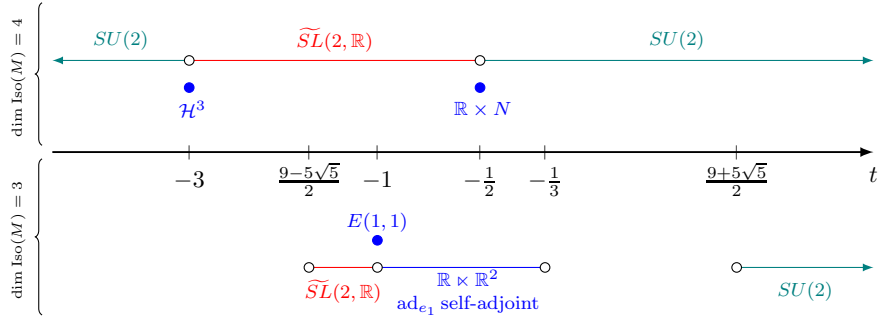
\begin{figure}[h]
	\centering
	\begin{tikzpicture}[scale=0.9, every node/.style={scale=0.9}]
		\draw[decorate,	decoration = {brace}] (-6.15,0.1)--(-6.15,2.2);
		\draw[decorate,	decoration = {brace}] (-6.15,-2.4)--(-6.15,-0.1);
		\node[rotate=90] at (-6.5,1.1) {\tiny $\operatorname{dim\,Iso}(M)=4$};
		\node[rotate=90] at (-6.5,-1.3) {\tiny $\operatorname{dim\,Iso}(M)=3$};
		\draw[line width=0.25mm,-latex] (-6,0)--(6,0);
		\draw (6,-0.1) node[below] {$t$};

		\draw (-4,0) node {\tiny$\vert$};
		\draw (-4,-0.17) node[below] {$-3$};
		
		\draw (-2.25,0) node {\tiny$\vert$};
		\draw (-2.25,-0.03) node[below] {$\frac{9-5\sqrt{5}}{2}$};
		
		\draw (-1.25,0) node {\tiny$\vert$};
		\draw (-1.25,-0.17) node[below] { $-1$};
		
		\draw (0.25,0) node {\tiny$\vert$};
		\draw (0.25,-0.1) node[below] {$-\frac{1}{2}$};
		
		\draw (1.2,0) node {\tiny$\vert$};
		\draw (1.2,-0.1) node[below] {$-\frac{1}{3}$};
		
		\draw (4,0) node {\tiny$\vert$};
		\draw (4,-0.03) node[below] {$\frac{9+5\sqrt{5}}{2}$};

		\def\YSup{0.65}
		\draw [color=teal,latex-] (-6,0.7+\YSup) -- (-4,0.7+\YSup) node [black,midway,yshift=+0.3cm] 
		{\footnotesize\color{teal} $SU(2)$};
		
		\draw [color=red,-] (-4,0.7+\YSup) -- (0.25,0.7+\YSup) node [black,midway,yshift=+0.3cm] 
		{\footnotesize\color{red} $\widetilde{SL}(2,\mathbb{R})$};
		
		\draw [color=teal,-latex] (0.25,0.7+\YSup) -- (6,0.7+\YSup) node [black,midway,yshift=+0.3cm] 
		{\footnotesize\color{teal} $SU(2)$};
		
		\filldraw [color=black,fill=white] (-4,0.7+\YSup) circle (2pt);
		\filldraw [color=black,fill=white] (0.25,0.7+\YSup) circle (2pt);

		\filldraw [color=blue,fill=blue] (-4,0.3+\YSup) circle (2pt);
		\filldraw [color=blue,fill=blue] (0.25,0.3+\YSup) circle (2pt);
		
		\draw (-3.95,-0.25+\YSup) node[above] {\footnotesize \blue $\mathcal{H}^3$};
		\draw (0.29,-0.25+\YSup) node[above] {\footnotesize\blue $\mathbb{R}\times N$};

		\blue
		\def\Ysubir{0.1}
		
		\draw (-1.25,-1.4+\Ysubir) node[above] {\footnotesize \blue $E(1,1)$};
		\filldraw [color=blue,fill=blue] (-1.25,-1.4+\Ysubir) circle (2pt);
		\draw [-] (-1.25,-1.8+\Ysubir) -- (1.2,-1.8+\Ysubir) node [black,midway,yshift=-0.37cm] {\footnotesize \blue \hspace{0.2cm}$\begin{array}{c}\mathbb{R}\ltimes\mathbb{R}^2\\\operatorname{ad}_{e_1} \text{self-adjoint}\end{array}$};
		\filldraw [color=black,fill=white] (1.2,-1.8+\Ysubir)  circle (2pt);
		
		\draw [color=teal,latex-] (6,-1.8+\Ysubir) -- (4,-1.8+\Ysubir) node [black,midway,yshift=-0.35cm] 
		{\footnotesize\color{teal} $SU(2)$};
		\filldraw [color=black,fill=white] (4,-1.8+\Ysubir) circle (2pt);

		\def\YSL{-0.5}
		\draw [color=red,-] (-1.25,-1.8+\Ysubir) -- (-2.25,-1.8+\Ysubir) 
		node [black,midway,yshift=-0.35cm] {\footnotesize \color{red}  $\widetilde{SL}(2,\mathbb{R})$};
		\filldraw [color=black,fill=white] (-2.25,-1.8+\Ysubir)  circle (2pt);
		\filldraw [color=black,fill=white] (-1.25,-1.8+\Ysubir)  circle (2pt);

		\black
	\end{tikzpicture}
	\caption{Non-Einstein homogeneous three-dimensional $\mathcal{F}_t$-critical metrics.}
	\label{figura-1}
\end{figure}


\subsection{Ricci solitons}\label{ss:1-3-2}
A Riemannian manifold $(M,g)$ is a \emph{Ricci soliton} if there exists a vector field $X$ on $M$ so that 
$
\rho=\lambda g+\mathcal{L}_Xg
$
for some $\lambda\in\mathbb{R}$, where $\mathcal{L}$ denotes the Lie derivative. Ricci solitons not only generalize Einstein metrics, but also correspond to self-similar solutions of the Ricci flow. A Ricci soliton is expanding, steady or shrinking depending on whether $\lambda<0$, $\lambda=0$ or $\lambda>0$, respectively.
Furthermore, if the vector field $X$ is a gradient, then the soliton is said to be a \emph{gradient Ricci soliton}. 

In the homogeneous context, steady Ricci solitons are flat and shrinking Ricci solitons are gradient, as follows from the work of Naber and Perelman \cite{Naber, Perelman}.
Petersen and Wylie \cite{rigid} showed that homogeneous gradient Ricci solitons are \emph{rigid}, i.e., they are isometric to a product $\mathbb{R}^k\times N$, where $N$ is an Einstein manifold and the potential function is the projection on the Euclidean factor $f=\frac{\lambda}{2}\|\pi_{\mathbb{R}^k}\!\|^2$. Hence, non-Einstein four-dimensional homogeneous gradient Ricci solitons are symmetric products $\mathbb{R}^\ell\times N^{4-\ell}(\kappa)$, where $N^{4-\ell}(\kappa)$ is a manifold of constant sectional curvature $\kappa$, as in cases (ii) and (iii) in Section~\ref{ss:symmetric}. Therefore, they are $\mathcal{F}_t$-critical metrics with zero energy for $t=-1/3$ or $t=-1/2$.
In contrast with this situation, Lauret \cite{Lauret1} showed the existence of non-symmetric homogeneous expanding Ricci solitons induced by algebraic Ricci solitons.

Let $G$ be a Lie group equipped with a left-invariant metric $\langle\cdot,\cdot\rangle$.
Then $(G,\langle\cdot,\cdot\rangle)$ is said to be an \emph{algebraic Ricci soliton} if the Ricci operator satisfies 
$
\operatorname{Ric}=\lambda \operatorname{Id}+\mathfrak{D}
$,
where $\mathfrak{D}$ is a derivation of $\mathfrak{g}=\operatorname{Lie}(G)$. Any algebraic Ricci soliton gives rise to a Ricci soliton \cite{Lauret1}. The converse is true for nilpotent Lie groups: if a left-invariant metric on a simply connected nilpotent Lie group is a Ricci soliton, then it is an algebraic Ricci soliton \cite{Lauret1}. 
Moreover, any homogeneous expanding Ricci soliton is isometric to an algebraic Ricci soliton which is realized on a solvable Lie group (solvsoliton) up to dimension five \cite{Arroyo-Lafuente, Jablonski2, Jablonski1}.

Algebraic Ricci solitons on non-solvable four-dimensional Lie groups are isomorphically homothetic to the left-invariant metric on $SU(2)\times\mathbb{R}$ determined by
$$
[e_1,e_2]= e_3,\quad [e_1,e_3]= - e_2,\quad [e_2,e_3]= e_1,
$$
where $\{ e_1,\dots,e_4\}$ is an orthonormal basis of the Lie algebra $\mathfrak{su}(2)\times\mathfrak{r}$.
Moreover it is locally conformally flat (hence symmetric by \cite{Takagi}) and homothetic to $\mathbb{R}\times \mathbb{S}^3$.

Four-dimensional solvsolitons were described by Lauret in \cite{Lauret2}. We work modulo homotheties (which do not necessarily preserve the group structure). Now, a long but direct calculation shows that four-dimensional non-symmetric solvsolitons are homothetic to one of the following four families (where $\{ e_1,\dots,e_4\}$ is an orthonormal basis of the corresponding Lie algebra). 
\begin{itemize}
	\item[(i)] The left-invariant $\mathcal{F}_{-3}$-critical metric on $\mathbb{R}\ltimes\mathbb{R}^3$ given by $[e_1,e_4]=e_1+ e_3$,\, $[e_3,e_4]=-  (e_1+e_3)$. 
		
	\smallskip
	\item[(ii)] The left-invariant $\mathcal{F}_{-3/2}$-critical  metric on $\mathbb{R}\ltimes\mathbb{R}^3$ given by 
	$$
	[e_1,e_4]=e_1+ \tfrac{1}{\sqrt{2}} e_3,\quad  
	[e_2,e_4]=-e_2+\tfrac{1}{\sqrt{2}} e_3,\quad  
	[e_3,e_4]=-\tfrac{1}{\sqrt{2}} (e_1+e_2).
	$$

	\smallskip	
	\item[(iii)]The left-invariant metrics on $\mathbb{R}\ltimes\mathbb{R}^3$ given by a self-adjoint derivation
	$$
	[e_1,e_4]=e_1, \quad [e_2,e_4]=f e_2, \quad [e_3,e_4]=p e_3,
	$$ 
	which are $\mathcal{F}_t$-critical for 
	$t=-\|\rho\|^2\tau^{-2}=-\frac{f^2+p^2+1}{2(f^2+p^2+fp+f+p+1)}\in[-1,-\frac{1}{4})$, 	
	where the parameters 			 $\{(f,p)\in\mathbb{R}^2;\, -1\leq f\leq p\leq 1\}\setminus \{(-1,p);\, -1\leq p<0\}$
	and $(f,p)\notin\{ (0,0), (0,1),(1,1) \}$.
	
	\smallskip
	\item[(iv)] The left-invariant metrics on $\mathbb{R}\ltimes \mathcal{H}^3$ given by 
	$$
	[e_1,e_2]=e_3, \quad [e_1,e_4]=a e_1, \quad [e_2,e_4]=d e_2, \quad [e_3,e_4]=(a+d)e_3,
	$$ 
	which are $\mathcal{F}_t$-critical for $t=-\|\rho\|^2\tau^{-2}=-\frac{3}{2 (4ad+5)}\in [-\frac{3}{4},-\frac{1}{4})$,
	with $a\in[-\frac{\sqrt{3}}{2},\frac{1}{2})$. For a fixed $a$, the parameter $d$ is given by the only positive solution of $4(a^2+d^2+ad)-3=0$.
\end{itemize}
In view of the value of $t$ for which metrics are critical, it is clear that metrics in cases (i) and (ii) are not homothetic to any other in this list. Furthermore, a direct computation of the set of homothetic invariants $\{t$, $\| R\|^2\, \tau^{-2}$, $\|\nabla\rho\|^2\,\tau^{-3}\}$ shows that metrics in case (iii) are never homothetic to those in case (iv), and also that different values of the parameters in (iii) or (iv) give rise to metrics which are not homothetic.

\begin{remark}\rm
	Although the left-invariant metrics on $\mathbb{R}\ltimes\mathbb{R}^3$ given by 
	$$
	[e_1,e_4]=e_1+ce_3,\quad [e_2,e_4]=fe_2,\quad [e_3,e_4]=-ce_1+e_3,
	$$ 
	are algebraic Ricci solitons, a straightforward calculation shows that the sectional curvature is independent of the structure constant $c$. Hence it follows from  \cite{Kulkarni} that these metrics   are homothetically equivalent (although not isomorphically equivalent) to a metric with $c=0$, which is a particular case of  (iii).
\end{remark}
\begin{remark}\label{re:xx}\rm
	The product metric on $\mathbb{R}\times \mathcal{H}^3$ given by $[e_1,e_2]=e_3$ is an algebraic Ricci soliton which is  isomorphically homothetic  to the metric (i) above, as follows from the change of basis given by 
	$$
	\bar e_1 = -\sqrt{2}(e_2-e_3),\quad
	\bar e_2 = 2 e_4,\quad
	\bar e_3 = \sqrt{2}(e_2+e_3),\quad
	\bar e_4 = 2 e_1.
	$$
	
	Analogously, the left-invariant metric on $\mathbb{R}\ltimes \mathcal{H}^3$ determined by $[e_1,e_2]=e_3$,  $[e_2,e_4]=-e_1$ is an algebraic Ricci soliton  which is isomorphically homothetic to the metric (ii), as deduced from the change of basis given by
	$$
	\bar e_1 = e_1+\tfrac{1}{\sqrt{2}} (e_3+e_4),\quad
	\bar e_2 = e_1-\tfrac{1}{\sqrt{2}}(e_3+e_4),\quad
	\bar e_3 = e_3-e_4,\quad 
	\bar e_4 = \sqrt{2} e_2.
	$$ 
\end{remark}
\begin{remark}\label{re:xxx}\rm
	The left-invariant metric on the product $E(1,1)\times \mathbb{R}$ given by
	$$
	[e_1,e_3]= e_2,\quad [e_2,e_3]= e_1,
	$$
	is a $\mathcal{F}_{-1}$-critical algebraic Ricci soliton.
	Considering the orthogonal basis
	$$
	\bar e_1 =\tfrac{1}{\sqrt{2}}(e_1+e_2),\quad
	\bar e_2 = \tfrac{1}{\sqrt{2}}(e_2-e_1),\quad
	\bar e_3 =e_4,
	\quad
	\bar e_4 = e_3,
	$$
	one gets that  this case is isomorphically homothetic to (iii) for the special values $f=-1$ and $p=0$.
\end{remark}

Four-dimensional algebraic homogeneous Ricci solitons described above are $\mathcal{F}_t$-critical for some quadratic curvature functional. Moreover, a straightforward case by case calculation shows that the energy $\energy_t=\|\rho\|^2+t\tau^2$ vanishes for all of them. 
Therefore, it follows from the work in \cite{Arroyo-Lafuente} that
\begin{quote}
\emph{Four-dimensional homogeneous Ricci solitons are critical for some quadratic curvature functional with zero energy}.
\end{quote} 
While the converse trivially holds in the four-dimensional symmetric setting, we show in Corollary~\ref{cor:hoy} that the class of homogeneous critical metrics with zero energy is strictly larger than that of Ricci solitons.

\subsection{Summary of results}\label{sect:summary}
Connected and simply connected non-symmetric homogeneous manifolds which are critical for some quadratic curvature functional are described. The resulting left-invariant metrics are summarized in Figures~\ref{figura-3} and \ref{figura-4}.

We work modulo homotheties to represent all non-symmetric $\mathcal{F}_t$-critical metrics with vanishing energy in Figure~\ref{figura-3}. 
Every such $\mathcal{F}_t$-critical metric is homothetic to a left-invariant metric on $\mathbb{R}\ltimes  \mathcal{H}^3$ or $\mathbb{R}\ltimes \mathbb{R}^3$. Thus,
the numbering in each semi-direct product refers to the corresponding item in the classification of metrics on $\mathbb{R}\ltimes  \mathcal{H}^3$ (Theorem~\ref{th:heisenberg}) or    $\mathbb{R}\ltimes\mathbb{R}^3$ (Theorem~\ref{th:criticas-R3}). The critical metric corresponding to Theorem~\ref{th:E(1,1)}--(1) is isomorphically homothetic to a metric in Theorem~\ref{th:criticas-R3}--(3) (see Remark~\ref{re:xxx}) and thus omitted in Figure~\ref{figura-3}. Analogously, metrics in Families~(1) and (3) of  Theorem~\ref{th:heisenberg} are omitted since they are isomorphically homothetic to metrics corresponding to Families~(1) and (2) in Theorem~\ref{th:criticas-R3}, as shown in Remark~\ref{re:xx}.

In order to assist the interpretation of figures, we explain the legend we use along the paper as follows. Each row indicates the range of $t$ for the corresponding family of metrics. Intervals indicating the range of $t$ are represented with a segment for each homothetic class and are marked with a dotting above if the number of homothetic classes is infinite. The arrow on the left (resp. on the right) indicates that the interval extends to $-\infty$ (resp. to $+\infty$). An empty dot means that the point is not included in the interval, whereas a filled dot indicates that the point belongs to the range of $t$.
Finally, colors are used to inform about the energy and Ricci solitons. Thus, blue tones correspond to zero energy, red color to negative energy and green color to positive energy. Also,
values colored in dark blue correspond to Ricci solitons, while those colored in other color (for example, cyan in Figure~\ref{figura-3}) are not homothetic to any Ricci soliton.

\def\RTres{\mathbb{R}\!\ltimes\!\mathbb{R}^3}
\def\HTres{\mathbb{R}\!\ltimes\! \mathcal{H}^3}
\def\EUnoUno{\mathbb{R}\!\ltimes\! E(1,1)}
\def\EDos{\mathbb{R}\!\ltimes\! \widetilde{E}(2)}

\def\Xmenos{-6}
\def\Xmenost{\Xmenos+1.9}
\def\Xmas{6}
\def\Xmast{6-0.4}
\def\XUnCuarto{\Xmast}

\def\XPequenho{0.52}
\def\XMedio{0.6}
\def\XGrande{0.7}
\def\XMasGrande{0.8}
\def\XTresOnceavos{\XUnCuarto-\XPequenho}
\def\XCincoMenosDosRaizSiete{\XTresOnceavos-\XPequenho}
\def\XTresDecimos{\XCincoMenosDosRaizSiete-\XPequenho}
\def\XUnTercio{\XTresDecimos-\XPequenho}
\def\XVeintiunCincuentaydosavos{\XUnTercio-\XPequenho}
\def\XSieteDieciseisavos{\XVeintiunCincuentaydosavos-\XPequenho}
\def\XCincoOnceavos{\XSieteDieciseisavos-\XPequenho}
\def\XSieteQuinceavos{\XCincoOnceavos-\XPequenho}
\def\XUnMedio{\XSieteQuinceavos-\XPequenho}

\def\XSieteDecimos{\XUnMedio-\XMedio}

\def\XTresCuartos{\XSieteDecimos-\XPequenho}

\def\XUno{\XTresCuartos-\XMedio}

\def\XTresMedios{\XUno-\XGrande}

\def\XZeta{\Xmenost+0.5}
\def\XTres{\XZeta+\XMasGrande}

\def\ColorECeroARS{blue}
\def\ColorECeroNoARS{cyan}
\def\ColorENegativa{red}
\def\ColorEPositiva{teal}
\def\EscalaGrupo{0.9}
\def\XGrupo{\Xmenos-0.1}
\def\XNumeroGrupo{\XGrupo+1.27}
\def\YPrimerGrupo{-0.45}
\def\YSepEntreClases{0.5}
\def\YSepInfMetricas{0.18}  
\def\YSepDosMetricas{0.075}

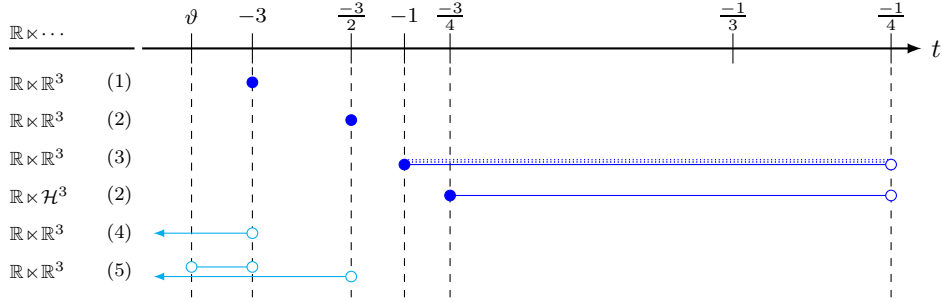
\begin{figure}[h]
	\centering

\def\NumeroCasos{6}
\def\YLineaPunteada{-\NumeroCasos*\YSepEntreClases-0.3}

\begin{tikzpicture}
	\draw[thick]   (\Xmenos,0) -- (\Xmenost-0.25,0); 
	\draw (\XGrupo,0.2) node[anchor=west,scale=\EscalaGrupo] {\footnotesize 
			$\mathbb{R}\!\ltimes\!\cdots$ };
	\draw[thick,-latex]   (\Xmenost-0.15,0) -- (\Xmas,0) node[right] {$t$};  
		
	\foreach \i/\n in 
	{
		{\XZeta}/$\vartheta$, 
		{\XTres}/$-3$, 
		{\XTresMedios}/$-\frac{3}{2}$, 
		{\XUno}/$-1$, 
		{\XTresCuartos}/$-\frac{3}{4}$,
		{\XUnTercio}/$-\frac{1}{3}$,
		{\XUnCuarto}/$-\frac{1}{4}$
	}
	{\draw (\i,-.2)--(\i,.2);}
		
	\foreach \i/\n in 
	{
		{\XZeta}/$\vartheta$, 
		{\XTres}/$-3$, 
		{\XUno}/$-1$
	}
	{\draw  (\i,.2) node[above] {\footnotesize \n};}
		
	\foreach \i/\n in 
	{
		{\XTresMedios}/$\frac{-3}{2}$, 
		{\XTresCuartos}/$\frac{-3}{4}$,
		{\XUnTercio}/$\frac{-1}{3}$,
		{\XUnCuarto}/$\frac{-1}{4}$
	}
	{\draw  (\i,.12) node[above] {\footnotesize \n};}
		
	\foreach \i/\n in 
	{
		{\XZeta}/$\vartheta$, 
		{\XTres}/$-3$, 
		{\XTresMedios}/$-\frac{3}{2}$, 
		{\XUno}/$-1$, 
		{\XTresCuartos}/$-\frac{3}{4}$,
		{\XUnCuarto}/$-\frac{1}{4}$
	}{
		\draw[dashed] (\i,-.25)--(\i,\YLineaPunteada);
	}

	\def\nGrupo{1}
	\def\YGrupo{\YPrimerGrupo-\nGrupo*\YSepEntreClases+\YSepEntreClases}
		
	\draw (\XNumeroGrupo,\YGrupo) node[anchor=west,scale=\EscalaGrupo] {\footnotesize (1) };
	\draw (\XGrupo,\YGrupo) node[anchor=west,scale=\EscalaGrupo] {\footnotesize 
			$\color{black} \RTres$};
		
	\filldraw [color=\ColorECeroARS,fill=\ColorECeroARS] (\XTres,\YGrupo) circle (2pt);

	\def\nGrupo{2}
	\def\YGrupo{\YPrimerGrupo-\nGrupo*\YSepEntreClases+\YSepEntreClases}
				
	\draw (\XNumeroGrupo,\YGrupo) node[anchor=west,scale=\EscalaGrupo] {\footnotesize (2) };
	\draw (\XGrupo,\YGrupo) node[anchor=west,scale=\EscalaGrupo] {\footnotesize $ 
		\color{black} \RTres$};

	\filldraw [color=\ColorECeroARS,fill=\ColorECeroARS] (\XTresMedios,\YGrupo) circle (2pt);

	\def\nGrupo{3}
	\def\YGrupo{\YPrimerGrupo-\nGrupo*\YSepEntreClases+\YSepEntreClases}
						
	\draw (\XNumeroGrupo,\YGrupo) node[anchor=west,scale=\EscalaGrupo] {\footnotesize (3) };
	\draw (\XGrupo,\YGrupo) node[anchor=west,scale=\EscalaGrupo] {\footnotesize $ 
		\color{black} \RTres$};

	\draw[densely dotted,color=\ColorECeroARS] 	
		(\XUno,\YGrupo-\YSepInfMetricas/4)--(\XUnCuarto,\YGrupo-\YSepInfMetricas/4);
	\draw[densely dotted,color=\ColorECeroARS] 	
(\XUno,\YGrupo-\YSepInfMetricas/8)--(\XUnCuarto,\YGrupo-\YSepInfMetricas/8);
	\draw[densely dotted,color=\ColorECeroARS] 	
(\XUno,\YGrupo-\YSepInfMetricas/3)--(\XUnCuarto,\YGrupo-\YSepInfMetricas/3);

	\draw[color=\ColorECeroARS] 	
		(\XUno,\YGrupo-\YSepInfMetricas/2)--(\XUnCuarto,\YGrupo-\YSepInfMetricas/2);
	\filldraw [color=\ColorECeroARS,fill=\ColorECeroARS] 
		(\XUno,\YGrupo-\YSepInfMetricas/2) circle (2pt);
	\filldraw [color=\ColorECeroARS,fill=white]
		(\XUnCuarto,\YGrupo-\YSepInfMetricas/2) circle (2pt);

	\def\nGrupo{4}
	\def\YGrupo{\YPrimerGrupo-\nGrupo*\YSepEntreClases+\YSepEntreClases}
							
	\draw (\XNumeroGrupo,\YGrupo) node[anchor=west,scale=\EscalaGrupo] {\footnotesize (2) };
	\draw (\XGrupo,\YGrupo) node[anchor=west,scale=\EscalaGrupo] {\footnotesize $ 
		\color{black} \HTres$};
			
	\draw[color=\ColorECeroARS](\XTresCuartos,\YGrupo)--(\XUnCuarto,\YGrupo);
	\filldraw [color=\ColorECeroARS,fill=\ColorECeroARS] 
		(\XTresCuartos,\YGrupo) circle (2pt);
	\filldraw [color=\ColorECeroARS,fill=white] 
		(\XUnCuarto,\YGrupo) circle (2pt);

	\def\nGrupo{5}
	\def\YGrupo{\YPrimerGrupo-\nGrupo*\YSepEntreClases+\YSepEntreClases}
	
	\draw (\XNumeroGrupo,\YGrupo) node[anchor=west,scale=\EscalaGrupo] {\footnotesize (4) };
	\draw (\XGrupo,\YGrupo) node[anchor=west,scale=\EscalaGrupo] {\footnotesize $ 
		\color{black} \RTres$};
	
	\draw[color=\ColorECeroNoARS,latex-=] (\Xmenost,\YGrupo)--(\XTres,\YGrupo);
	\filldraw [color=\ColorECeroNoARS,fill=white] (\XTres,\YGrupo) circle (2pt);

	\def\nGrupo{6}
	\def\YGrupo{\YPrimerGrupo-\nGrupo*\YSepEntreClases+\YSepEntreClases}
	
	\draw (\XNumeroGrupo,\YGrupo) node[anchor=west,scale=\EscalaGrupo] {\footnotesize (5) };
	\draw (\XGrupo,\YGrupo) node[anchor=west,scale=\EscalaGrupo] {\footnotesize $ 
		\color{black} \RTres$};
	
	\draw[color=\ColorECeroNoARS,latex-=] 			
		(\Xmenost,\YGrupo-\YSepDosMetricas)--(\XTresMedios,\YGrupo-\YSepDosMetricas);
	\filldraw[color=\ColorECeroNoARS,fill=white] 
		(\XTresMedios,\YGrupo-\YSepDosMetricas) circle (2pt);
	
	\draw[color=\ColorECeroNoARS] 			
		(\XZeta,\YGrupo+0.7*\YSepDosMetricas)--(\XTres,\YGrupo+0.7*\YSepDosMetricas);
	\filldraw[color=\ColorECeroNoARS,fill=white] 
		(\XZeta,\YGrupo+0.7*\YSepDosMetricas) circle (2pt);
	\filldraw[color=\ColorECeroNoARS,fill=white] 
		(\XTres,\YGrupo+0.7*\YSepDosMetricas) circle (2pt);
	 
	\end{tikzpicture}

		\caption{Range of the parameter $t$ for non-symmetric homogeneous $\mathcal{F}_t$-critical metrics with zero energy.}
	\label{figura-3}
\end{figure}

Metrics corresponding to different families in Figure~\ref{figura-3} are never homothetic. Moreover, metrics in each family are also non-homothetic. Furthermore, for each admissible value of $t$ there is a single $\mathcal{F}_t$-critical metric in Theorem~\ref{th:heisenberg}--(2) and in Families~(1), (2) and (4) of Theorem~\ref{th:criticas-R3}. For each $t<-3/2$, there is a unique $\mathcal{F}_t$-critical metric in Theorem~\ref{th:criticas-R3}--(5), except if $\vartheta<t<-3$, where $\vartheta=-3,753199\dots$ is the only real solution of $192 p^3+  1152 p^2+1865 p+923=0$. For any such value of $t$, there are two non-homothetic $\mathcal{F}_t$-critical metrics. In contrast, for each $t\in[-1,-\frac{1}{4})$ there is an infinite family of non-homothetic $\mathcal{F}_t$-critical metrics corresponding to Family~(3) in Theorem~\ref{th:criticas-R3}.

Homogeneous critical metrics with non-zero energy are represented in Figure~\ref{figura-4}. The numbering in the second column refers to the corresponding item in the classification result for the semi-direct product $\mathbb{R}\ltimes\mathbb{R}^3$ (Theorem~\ref{th:criticas-R3}), $\mathbb{R}\ltimes H^3$ (Theorem~\ref{th:heisenberg}), or $\mathbb{R}\ltimes E(1,1)$ and $\mathbb{R}\ltimes \widetilde{E}(2)$ (Theorem~\ref{th:E(1,1)}), respectively.

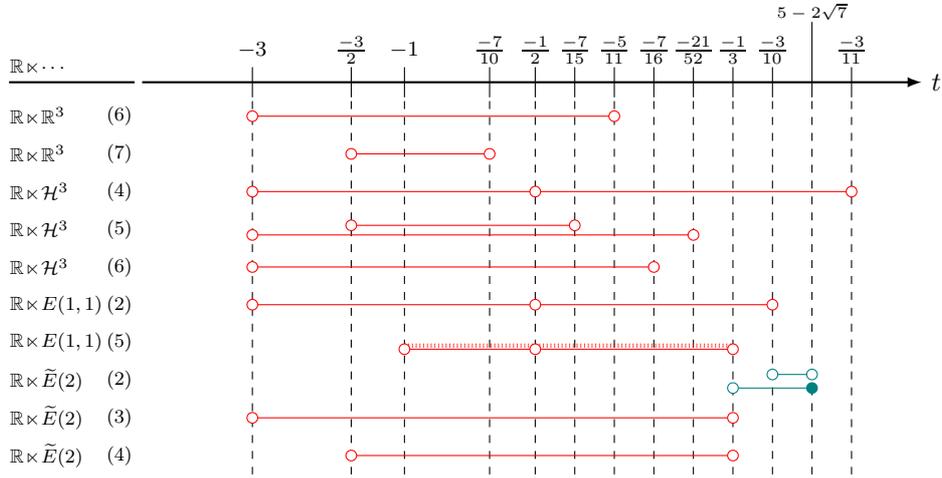
\begin{figure}[h]
	\centering

\def\NumeroCasos{10}
\def\YLineaPunteada{-\NumeroCasos*\YSepEntreClases-0.2}
	\begin{tikzpicture}
		\draw[thick]   (\Xmenos,0) -- (\Xmenost-0.25,0); 
		\draw (\XGrupo,0.2) node[anchor=west,scale=\EscalaGrupo] {\footnotesize 			
			$\mathbb{R}\!\ltimes\!\cdots$ };
		\draw[thick,-latex]   (\Xmenost-0.15,0) -- (\Xmas,0) node[right] {$t$}; 
		
		\foreach \i/\n in 
		{
			{\XTres}/$-3$, 
			{\XTresMedios}/$-\frac{3}{2}$, 
			{\XUno}/$-1$, 
			{\XSieteDecimos}/$-\frac{7}{10}$, 
			{\XUnMedio}/$-\frac{1}{2}$,
			{\XSieteQuinceavos}/$-\frac{7}{15}$,
			{\XCincoOnceavos}/$-\frac{5}{11}$, 
			{\XSieteDieciseisavos}/$-\frac{7}{16}$,
			{\XVeintiunCincuentaydosavos}/$-\frac{21}{52}$,	
			{\XUnTercio}/$-\frac{1}{3}$,
			{\XTresDecimos}/$-\frac{3}{10}$,
			{\XTresOnceavos}/$-\frac{3}{11}$
		}
		{\draw (\i,-.2)--(\i,.2);}

		\foreach \i/\n in 
		{
			{\XCincoMenosDosRaizSiete}/$5-2\sqrt{7}$ 
		}
		{\draw (\i,-.2)--(\i,.8);}

		\foreach \i/\n in 
		{
			{\XTres}/$-3$, 
			{\XUno}/$-1$ 
		}
		{\draw  (\i,.2) node[above] {\footnotesize   \n};}
		
		\foreach \i/\n in 
		{
			{\XCincoMenosDosRaizSiete}/$5-2\sqrt{7}$
		}
		{\draw  (\i,.7) node[above] {\footnotesize \tiny \n};}

		\foreach \i/\n in 
		{
			{\XTresMedios}/$\frac{-3}{2}$, 
			{\XSieteDecimos}/$\frac{-7}{10}$, 
			{\XUnMedio}/$\frac{-1}{2}$,
			{\XSieteQuinceavos}/$\frac{-7}{15}$,
			{\XCincoOnceavos}/$\frac{-5}{11}$, 
			{\XSieteDieciseisavos}/$\frac{-7}{16}$,
			{\XVeintiunCincuentaydosavos}/$\frac{-21}{52}$,	
			{\XUnTercio}/$\frac{-1}{3}$,
			{\XTresDecimos}/$\frac{-3}{10}$,
			{\XTresOnceavos}/$\frac{-3}{11}$
		}
		{\draw  (\i,.12) node[above] {\footnotesize  \n};}

		\foreach \i/\n in 
		{
			{\XTres}/$-3$, 
			{\XTresMedios}/$-\frac{3}{2}$, 
			{\XUno}/$-1$, 
			{\XSieteDecimos}/$-\frac{7}{10}$, 
			{\XUnMedio}/$-\frac{1}{2}$,
			{\XSieteQuinceavos}/$-\frac{7}{15}$, 
			{\XCincoOnceavos}/$-\frac{5}{11}$, 
			{\XSieteDieciseisavos}/$-\frac{7}{16}$,
			{\XVeintiunCincuentaydosavos}/$-\frac{21}{52}$,	
			{\XUnTercio}/$-\frac{1}{3}$,
			{\XTresDecimos}/$-\frac{3}{10}$,
			{\XCincoMenosDosRaizSiete}/$5-2\sqrt{7}$ ,
			{\XTresOnceavos}/$-\frac{3}{11}$
		}
		{
			\draw[dashed] (\i,-.25)--(\i,\YLineaPunteada);
		}

		\def\nGrupo{1}
		\def\YGrupo{\YPrimerGrupo-\nGrupo*\YSepEntreClases+\YSepEntreClases}
		
		\draw (\XNumeroGrupo,\YGrupo) node[anchor=west,scale=\EscalaGrupo] {\footnotesize (6) };
		\draw (\XGrupo,\YGrupo) node[anchor=west,scale=\EscalaGrupo] {\footnotesize 
			$\color{black} \RTres$};
		
		\draw[color=\ColorENegativa] (\XTres,\YGrupo)--(\XCincoOnceavos,\YGrupo);
		\filldraw [color=\ColorENegativa,fill=white] (\XTres,\YGrupo) circle (2pt);
		\filldraw [color=\ColorENegativa,fill=white] (\XCincoOnceavos,\YGrupo) circle (2pt); 
		 
		\def\nGrupo{2}
		\def\YGrupo{\YPrimerGrupo-\nGrupo*\YSepEntreClases+\YSepEntreClases}
		
		\draw (\XNumeroGrupo,\YGrupo) node[anchor=west,scale=\EscalaGrupo] {\footnotesize (7) };
		\draw (\XGrupo,\YGrupo) node[anchor=west,scale=\EscalaGrupo] {\footnotesize 
			$\color{black} \RTres$};
		
		\draw[color=\ColorENegativa] (\XTresMedios,\YGrupo)--(\XSieteDecimos,\YGrupo);
		\filldraw [color=\ColorENegativa,fill=white] (\XTresMedios,\YGrupo) circle (2pt);
		\filldraw [color=\ColorENegativa,fill=white] (\XSieteDecimos,\YGrupo) circle (2pt);

		\def\nGrupo{3}
		\def\YGrupo{\YPrimerGrupo-\nGrupo*\YSepEntreClases+\YSepEntreClases}
		
		\draw (\XNumeroGrupo,\YGrupo) node[anchor=west,scale=\EscalaGrupo] {\footnotesize (4) };
		\draw (\XGrupo,\YGrupo) node[anchor=west,scale=\EscalaGrupo] {\footnotesize 
			$\color{black} \HTres$};
		
		\draw[color=\ColorENegativa] (\XTres,\YGrupo)--(\XTresOnceavos,\YGrupo);
		\filldraw [color=\ColorENegativa,fill=white] (\XTres,\YGrupo) circle (2pt);
		\filldraw [color=\ColorENegativa,fill=white] (\XUnMedio,\YGrupo) circle (2pt);
		\filldraw [color=\ColorENegativa,fill=white] (\XTresOnceavos,\YGrupo) circle (2pt);

		\def\nGrupo{4}
		\def\YGrupo{\YPrimerGrupo-\nGrupo*\YSepEntreClases+\YSepEntreClases}
		
		\draw (\XNumeroGrupo,\YGrupo) node[anchor=west,scale=\EscalaGrupo] {\footnotesize (5) };
		\draw (\XGrupo,\YGrupo) node[anchor=west,scale=\EscalaGrupo] {\footnotesize 
			$\color{black} \HTres$};
					
		\draw[color=\ColorENegativa] (\XTres,\YGrupo-\YSepDosMetricas)--(\XVeintiunCincuentaydosavos,\YGrupo-\YSepDosMetricas);
		\filldraw [color=\ColorENegativa,fill=white] (\XTres,\YGrupo-\YSepDosMetricas) circle (2pt);
		\filldraw [color=\ColorENegativa,fill=white] (\XVeintiunCincuentaydosavos,\YGrupo-\YSepDosMetricas) circle (2pt); 
		
		\draw[color=\ColorENegativa] 			
			(\XTresMedios,\YGrupo+0.7*\YSepDosMetricas)--(\XSieteQuinceavos,\YGrupo+0.7*\YSepDosMetricas);
		\filldraw [color=\ColorENegativa,fill=white] 	
			(\XTresMedios,\YGrupo+0.7*\YSepDosMetricas) circle (2pt);
		\filldraw [color=\ColorENegativa,fill=white] 
			(\XSieteQuinceavos,\YGrupo+0.7*\YSepDosMetricas) circle (2pt);

		\def\nGrupo{5}
		\def\YGrupo{\YPrimerGrupo-\nGrupo*\YSepEntreClases+\YSepEntreClases}
		
		\draw (\XNumeroGrupo,\YGrupo) node[anchor=west,scale=\EscalaGrupo] {\footnotesize (6) };
		\draw (\XGrupo,\YGrupo) node[anchor=west,scale=\EscalaGrupo] {\footnotesize 
			$\color{black} \HTres$};

		\draw[color=\ColorENegativa] (\XTres,\YGrupo)--(\XSieteDieciseisavos,\YGrupo);
		\filldraw [color=\ColorENegativa,fill=white] (\XTres,\YGrupo) circle (2pt);
		\filldraw [color=\ColorENegativa,fill=white] (\XSieteDieciseisavos,\YGrupo) circle (2pt);

		\def\nGrupo{6}
		\def\YGrupo{\YPrimerGrupo-\nGrupo*\YSepEntreClases+\YSepEntreClases}
		
		\draw (\XNumeroGrupo,\YGrupo) node[anchor=west,scale=\EscalaGrupo] {\footnotesize {(2)} };
		\draw (\XGrupo,\YGrupo) node[anchor=west,scale=\EscalaGrupo] {\footnotesize 
			$\color{black} \EUnoUno$};
		
		\draw[color=\ColorENegativa] (\XTres,\YGrupo)--(\XTresDecimos,\YGrupo);
		\filldraw [color=\ColorENegativa,fill=white] (\XTres,\YGrupo) circle (2pt);
		\filldraw [color=\ColorENegativa,fill=white] (\XUnMedio,\YGrupo) circle (2pt);
		\filldraw [color=\ColorENegativa,fill=white] (\XTresDecimos,\YGrupo) circle (2pt);

		\def\nGrupo{7}
		\def\YGrupo{\YPrimerGrupo-\nGrupo*\YSepEntreClases+\YSepEntreClases}
		
		\draw (\XNumeroGrupo,\YGrupo) node[anchor=west,scale=\EscalaGrupo] {\footnotesize {(5)} };
		\draw (\XGrupo,\YGrupo) node[anchor=west,scale=\EscalaGrupo] {\footnotesize 
			$\color{black} \EUnoUno$};

		\draw[densely dotted,color=\ColorENegativa] 	
			(\XUno,\YGrupo-\YSepInfMetricas/4)--(\XUnTercio,\YGrupo-\YSepInfMetricas/4);
					\draw[densely dotted,color=\ColorENegativa] 	
			(\XUno,\YGrupo-\YSepInfMetricas/8)--(\XUnTercio,\YGrupo-\YSepInfMetricas/8);
					\draw[densely dotted,color=\ColorENegativa] 	
			(\XUno,\YGrupo-\YSepInfMetricas/3)--(\XUnTercio,\YGrupo-\YSepInfMetricas/3);

		\draw[color=\ColorENegativa] (\XUno,\YGrupo-\YSepInfMetricas/2)--(\XUnTercio,\YGrupo-\YSepInfMetricas/2);
		\filldraw [color=\ColorENegativa,fill=white] (\XUno,\YGrupo-\YSepInfMetricas/2) circle (2pt);
		\filldraw [color=\ColorENegativa,fill=white] (\XUnMedio,\YGrupo-\YSepInfMetricas/2) circle (2pt);
		\filldraw [color=\ColorENegativa,fill=white] (\XUnTercio,\YGrupo-\YSepInfMetricas/2) circle (2pt);

		\def\nGrupo{8}
		\def\YGrupo{\YPrimerGrupo-\nGrupo*\YSepEntreClases+\YSepEntreClases}
		
		\draw (\XNumeroGrupo,\YGrupo) node[anchor=west,scale=\EscalaGrupo] {\footnotesize {(2)} };
		\draw (\XGrupo,\YGrupo) node[anchor=west,scale=\EscalaGrupo] {\footnotesize 
			$\color{black} \EDos$};

		\draw[color=\ColorEPositiva] 
			(\XUnTercio,\YGrupo-\YSepDosMetricas-0.03)--(\XCincoMenosDosRaizSiete,\YGrupo-\YSepDosMetricas-0.03);
		\filldraw [color=\ColorEPositiva,fill=white] 
			(\XUnTercio,\YGrupo-\YSepDosMetricas-0.03) circle (2pt);
		\filldraw[color=\ColorEPositiva,fill=\ColorEPositiva] 
			(\XCincoMenosDosRaizSiete,\YGrupo-\YSepDosMetricas-0.03) circle (2pt);

		\draw[color=\ColorEPositiva] 		
			(\XTresDecimos,\YGrupo+1*\YSepDosMetricas)--(\XCincoMenosDosRaizSiete,\YGrupo+1*\YSepDosMetricas);
		\filldraw [color=\ColorEPositiva,fill=white] 
			(\XTresDecimos,\YGrupo+1*\YSepDosMetricas) circle (2pt);
		\filldraw [color=\ColorEPositiva,fill=white] 
			(\XCincoMenosDosRaizSiete,\YGrupo+1*\YSepDosMetricas) circle (2pt);

		\def\nGrupo{9}
		\def\YGrupo{\YPrimerGrupo-\nGrupo*\YSepEntreClases+\YSepEntreClases}
		
		\draw (\XNumeroGrupo,\YGrupo) node[anchor=west,scale=\EscalaGrupo] {\footnotesize {(3)} };
		\draw (\XGrupo,\YGrupo) node[anchor=west,scale=\EscalaGrupo] {\footnotesize 
			$\color{black} \EDos$}; 
		
		\draw[color=\ColorENegativa] (\XTres,\YGrupo)--(\XUnTercio,\YGrupo);
		\filldraw [color=\ColorENegativa,fill=white] (\XTres,\YGrupo) circle (2pt);
		\filldraw [color=\ColorENegativa,fill=white] (\XUnTercio,\YGrupo) circle (2pt);

		\def\nGrupo{10}
		\def\YGrupo{\YPrimerGrupo-\nGrupo*\YSepEntreClases+\YSepEntreClases}
		
		\draw (\XNumeroGrupo,\YGrupo) node[anchor=west,scale=\EscalaGrupo] {\footnotesize {(4)} };
		\draw (\XGrupo,\YGrupo) node[anchor=west,scale=\EscalaGrupo] {\footnotesize 
			$\color{black} \EDos$};
		
		\draw[color=\ColorENegativa] (\XTresMedios,\YGrupo)--(\XUnTercio,\YGrupo);
		\filldraw [color=\ColorENegativa,fill=white] (\XTresMedios,\YGrupo) circle (2pt);
		\filldraw [color=\ColorENegativa,fill=white] (\XUnTercio,\YGrupo) circle (2pt);
		
	\end{tikzpicture}
\caption{Range of the parameter $t$ for non-symmetric homogeneous $\mathcal{F}_t$-critical metrics with non-zero energy.}
\label{figura-4}
\end{figure}

There is only one family of left-invariant $\mathcal{F}_t$-critical metrics with positive energy (colored in green in Figure~\ref{figura-4}), which corresponds to  left-invariant metrics on $\mathbb{R}\ltimes \widetilde{E}(2)$ given in Theorem~\ref{th:E(1,1)}--(2). In all other cases the energy is negative. 
As in the case of zero energy, metrics corresponding to different families in Figure~\ref{figura-4} are not homothetic, and metrics within the same family correspond to different homothety classes too. 
For each admissible value of $t$, the corresponding $\mathcal{F}_t$-critical metrics are unique in Families~(6) and (7) of Theorem~\ref{th:criticas-R3}, (4) and (6) of Theorem~\ref{th:heisenberg}, (3) and (4) of Theorem~\ref{th:E(1,1)}. Family~(2) of Theorem~\ref{th:E(1,1)} includes metrics which are realized on $\mathbb{R}\ltimes E(1,1)$, with one homothetic class for each value of $t$, and metrics which are realized on $\mathbb{R}\ltimes \widetilde{E}(2)$, with two non-homothetic solutions for each $t$ if $-3/10<t<5-2\sqrt{7}$ and only one otherwise. There are also two non-homothetic solutions for Family~(5) in Theorem~\ref{th:heisenberg} if $-3/2<t<-7/15$,  and a unique $\mathcal{F}_t$-critical metric otherwise. Finally, in Family~(5) of  Theorem~\ref{th:E(1,1)}, which is realized in $\mathbb{R}\ltimes E(1,1)$, there are infinitely many non-homothetic $\mathcal{F}_t$-critical metrics.

\begin{remark}\rm  \label{re:Bach-llanas}
Bach-flat four-dimensional metrics are critical for the functional $\mathcal{F}_{-1/3}$. Thus, the main result in \cite{CL-GM-GR-GR-VL}, which provides all Bach-flat four-dimensional homogeneous metrics, follows from the results above (see also \cite{AGS}). Homogeneous metrics which are  $\mathcal{F}_{-1/3}$-critical with zero energy are locally conformally flat products, or the left-invariant metric given in an orthonormal basis $\{ e_1,\dots, e_4\}$ on $\mathbb{R}\ltimes \mathcal{H}^3$ with Lie brackets (see Theorem~\ref{th:heisenberg}--(2))
	$$
	{[e_1,e_2]}=  e_3, \,\,
	{[e_1,e_4]} =-\tfrac{1}{4}\!\sqrt{7\!-\!3\sqrt{5}}\, e_1,\,\,
	{[e_2,e_4]}=\tfrac{1}{4}\!\sqrt{7\!+\!3\sqrt{5}}\, e_2,\,\,
	{[e_3,e_4]}=\tfrac{\sqrt{5}}{2\sqrt{2}}\, e_3,
	$$
or on $\mathbb{R}\ltimes\mathbb{R}^3$ with (see  Theorem~\ref{th:criticas-R3}--(3))
$$
[e_1,e_4] = e_1,\quad 
[e_2,e_4] = (\sqrt{p}-1)^2\,e_2,\quad 
[e_3,e_4] =p\,e_3,
\qquad \tfrac{1}{4}\leq p< 1\,.
$$
Non-Einstein Bach-flat metrics with non-zero energy correspond to the left-invariant metric on $\mathbb{R}\ltimes E(1,1)$ given by (see Theorem~\ref{th:E(1,1)}--(2))
$$
[e_1,e_3]= (2\!-\!\sqrt{3})\, e_2,\,\,\,
[e_1,e_4] = \sqrt{\!6\!-\!3\sqrt{3}}\, e_1,\,\,\, 
[e_2,e_3]= e_1,\,\,\, 
[e_2,e_4]=\sqrt{\!6\!-\!3\sqrt{3}}\, e_2\,,
$$	
or the self-dual left-invariant metric 
$$
[e_1, e_2]=e_3,\,\quad
[e_1, e_4]= e_1,\,
\quad 
[e_2,  e_4]=  e_2,\,
\quad
[e_3, e_4]= 2  e_3,\,
$$
on a semi-direct extension $\mathbb{R}\ltimes \mathcal{H}^3$ of the Heisenberg algebra   (see Theorem~\ref{th:heisenberg}--(4) and Remark~\ref{re:H3-grafico}). We shall notice that the families of Bach-flat metrics described above are simplified with respect to those in \cite{CL-GM-GR-GR-VL} in such a way that any two of them are non-homothetic.
\end{remark}

\begin{remark}\rm
	It follows from the results in sections \ref{se:SL(2,R)}--\ref{se:R} that homogeneous $\mathcal{F}_t$-critical metrics with $t\geq-1/4$ are either Einstein or homothetic to the product $\mathbb{S}^2\times \mathbb{H}^2$. Therefore, in this context any $\mathcal{F}_t$-critical metric with $t\geq-1/4$ is critical for all quadratic curvature functionals. This is in sharp contrast with the three-dimensional setting, where there are non-symmetric homogeneous critical metrics for any quadratic curvature functional (see Figure~\ref{figura-1} and \cite{BV-GR-CO}).
	
In particular, for the functionals $\mathcal{F}_{-1/4}$ (which is equivalent to the $L^2$-norm of the curvature tensor) or $\mathcal{F}_{-2/9}$ (which is equivalent to the $L^2$-norm of the Schouten tensor) any homogeneous critical metric is necessarily symmetric (see \cite{Kang} for the special case of homogeneous manifolds with finite volume). This behavior significantly differs from the three-dimensional analogues, which were considered in \cite{HNS, Lamontagne}. 
\end{remark}

\subsection{Structure of the paper}
We analyze the existence of non-symmetric homogeneous $\mathcal{F}_t$-critical metrics by considering the possible left-invariant metrics on four-dimensional Lie groups through sections \ref{se:SL(2,R)}--\ref{se:R}.
In Section~\ref{se:SL(2,R)} we show that there are no $\mathcal{F}_t$-critical metrics in $\widetilde{SL}(2, \mathbb{R})\times\mathbb{R}$ and the only $\mathcal{F}_t$-critical metric in $SU(2)\times\mathbb{R}$ is isomorphically homothetic to the symmetric locally conformally flat metric on $\mathbb{R}\times\mathbb{S}^3$ (see Theorem~\ref{th:non-solvable}). This situation strongly contrasts with the three-dimensional case, where there are plenty of $\mathcal{F}_t$-critical metrics on $\widetilde{SL}(2, \mathbb{R})$ and $SU(2)$ (see Figure~\ref{figura-1}).
Non-symmetric critical metrics on semi-direct extensions of the Euclidean and Poincaré groups are considered in Section~\ref{se:E(1,1)}. We describe all such critical metrics in Theorem~\ref{th:E(1,1)}.
Left-invariant $\mathcal{F}_t$-critical metrics on semi-direct extensions of the Heisenberg group are considered in Section~\ref{se:H}. Theorem~\ref{th:heisenberg} provides a complete list of the metrics with the corresponding functionals for which they are critical.
Semi-direct extensions of the Abelian group are considered in Section~\ref{se:R}. The situation is more involved for this group, as shown in Theorem~\ref{th:criticas-R3}. Indeed, while all Ricci solitons are $\mathcal{F}_t$-critical with zero energy, there are critical metrics on $\mathbb{R}\ltimes\mathbb{R}^3$ with vanishing energy which are not homothetic to any Ricci soliton (see Corollary~\ref{cor:hoy}).

In order to describe the left-invariant critical metrics above one needs to solve some systems of
polynomial equations on the parameter $t$ and the structure constants. Whenever a system of
polynomial equations is simple, finding all common roots may be an
easy task, but if the number of polynomials and their degree
increase, it usually becomes unmanageable. Let
$\mathfrak{S}$ be a system of polynomials $\mathfrak{P}_i\in
\mathbb{R}[x_1,\dots,x_n]$. An $n$-tuple $\vec{x}$ is a solution for
$\mathfrak{S}$ if and only if $\mathfrak{P}_i(\vec{x})=0$, for all
$i$. It is a fundamental observation that $\vec{x}$ is a solution
for the system if and only if it is a solution of the ideal
$\mathcal{I}=\langle\mathfrak{P}_i\rangle$  generated by the
polynomials of the system.
If two sets of polynomials generate the same ideal, the
corresponding zero sets must be identical. Hence one may try to
simplify the problem looking for ``better'' polynomials in the
ideal, since the zero sets are the same. Gr\"obner bases provide a
useful tool in solving systems of polynomial equations by following
the above strategy.
The Hilbert Basis Theorem guarantees that any non-zero ideal $\mathcal{I}$  admits a Gr\"obner basis. Furthermore,
any Gr\"obner basis for an ideal $\mathcal{I}$ is a basis of $\mathcal{I}$  and, thus, it is specially suited to analyze the ideal membership problem. 
We refer to \cite{Cox} for more information on Gr\"obner bases.

All left-invariant $\mathcal{F}_t$-critical metrics are obtained by solving the corresponding polynomial equations directly.  However, in order to show that no other solutions may exist, we use Gröbner bases in some particular cases.
The calculations of the needed Gr\"obner bases in this paper have been performed with the computer software of {\sc Singular} \cite{Singular} and double checked with {\sc Mathematica}.

\section{Left-invariant $\mathcal{F}_t$-critical metrics on $\widetilde{SL}(2, \mathbb{R}) \times \mathbb{R}$ and $SU(2) \times \mathbb{R}$}\label{se:SL(2,R)}

Let $\mathfrak{g}=\mathfrak{g}_3\times \mathbb{R}$ be a direct extension of the unimodular Lie algebra $\mathfrak{g}_3=\mathfrak{sl}(2,\mathbb{R})$ or $\mathfrak{g}_3=\mathfrak{su}(2)$. Let $\langle \cdot, \cdot \rangle$ be an inner product on $\mathfrak{g}$ and let $\langle \cdot, \cdot \rangle_3$ denote its restriction to $\mathfrak{g}_3$. Following the work of Milnor \cite{Milnor}, there exists an orthonormal basis $\{\mathbf{v}_1,\mathbf{v}_2, \mathbf{v}_3\}$ of $\mathfrak{g}_3$ such that
\begin{equation}\label{basis_non-solvable-1}
[\mathbf{v}_2,\mathbf{v}_3]= \lambda_1 \mathbf{v}_1,
\quad 
[\mathbf{v}_3, \mathbf{v}_1]= \lambda_2 \mathbf{v}_2,
\quad 
[\mathbf{v}_1, \mathbf{v}_2]= \lambda_3 \mathbf{v}_3,
\end{equation}
where $\lambda_1, \lambda_2,\lambda_3 \in \mathbb{R}$ and $\lambda_1 \lambda_2\lambda_3 \neq 0$. Moreover, the associated Lie group corresponds to $SU(2)$  (resp., $\widetilde{SL}(2,\mathbb{R})$) if  $\lambda_1, \lambda_2,\lambda_3$ do not change sign  (resp.,   change sign).

Let $\{\mathbf{v}_1, \mathbf{v}_2, \mathbf{v}_3, \mathbf{v}_4\}$ be a basis of $\mathfrak{g}$ such that $\{\mathbf{v}_1, \mathbf{v}_2, \mathbf{v}_3\}$ are given by Equation~\eqref{basis_non-solvable-1} and $\mathfrak{g}=\mathfrak{g}_3\oplus\mathbb{R}\mathbf{v}_4$. Since $\mathbb{R}\mathbf{v}_4$ need not be orthogonal to $\mathfrak{g}_3$, we set $\tilde k_i=\langle \mathbf{v}_i, \mathbf{v}_4 \rangle$, for $i = 1,2,3$. Let $\bar e_4= \mathbf{v}_4-\sum_i \tilde k_i\mathbf{v}_i$ and normalize it to get an orthonormal basis $\{ e_1,\dots,e_4\}$ of $\mathfrak{g}= \mathfrak{g}_3 \oplus \mathbb{R}$, where $e_i=\mathbf{v}_i$ and $\tilde k_i=\|\overline{e}_4\| k_i$ for $i=1,2,3$, so that 
\begin{equation}\label{eq:non-solvable}
\begin{array}{l}
 [e_2,e_3]= \lambda_1 e_1,
	\qquad 
	[e_3, e_1]= \lambda_2 e_2,
	\qquad 
	[e_1, e_2]= \lambda_3 e_3,
\\
\noalign{\medskip}

[e_1,e_4] =  k_3 \lambda_2 e_2- k_2 \lambda_3  e_3,
\qquad\!
[e_2,e_4]= k_1 \lambda_3  e_3-  k_3 \lambda_1 e_1,

\\ 
\noalign{\medskip} 

[e_3,e_4]= k_2 \lambda_1 e_1-  k_1 \lambda_2 e_2.
\end{array}
\end{equation}	

Left-invariant metrics above are never Einstein and 
have scalar curvature 
$$
\begin{array}{l}
\tau=-\frac{1}{2}\left\{ 
\lambda_1^2+\lambda_2^2+\lambda_3^2-2(\lambda_1\lambda_2+\lambda_1\lambda_3+\lambda_2\lambda_3)
\right.

\\
\noalign{\medskip}
\phantom{\tau=-\frac{1}{2}\{}

\left.
+k_1^2(\lambda_2-\lambda_3)^2
+k_2^2(\lambda_1-\lambda_3)^2
+k_3^2(\lambda_1-\lambda_2)^2\right\}.
\end{array}
$$
The scalar curvature vanishes for values of the parameters $\lambda_i, k_i$ on a five-dimensio\-nal manifold of $\mathbb{R}^6$, since $\tau=0$ is a regular value of the function above.
Hence $\mathcal{S}$-critical metrics exist for suitable values of the structure constants.

\begin{example}\label{ex:lcf-product}\rm
The case $\lambda_1=\lambda_2=\lambda_3=\lambda$ is of particular interest. The vector field $e_4$ is parallel, so the manifold splits as a Riemannian product. Moreover, the submanifold corresponding to $\mathfrak{g}_3$ is Einstein and, hence, of constant sectional curvature $\kappa=\lambda^2/4$. Therefore, these examples are locally isometric to a product $\mathbb{R}\times N^3(\kappa)$, so they are locally conformally flat and locally symmetric (see  the work of Takagi \cite{Takagi}).
\end{example}

The following result shows that left-invariant metrics \eqref{eq:non-solvable} are critical for some quadratic curvature functional $\mathcal{F}_t$ if and only if they correspond to those described in Example~\ref{ex:lcf-product}.

\begin{theorem}\label{th:non-solvable}
A left-invariant metric on $\widetilde{SL}(2, \mathbb{R}) \times \mathbb{R}$ or $SU(2) \times \mathbb{R}$ is critical for some quadratic curvature functional $\mathcal{F}_t$ if and only if  it is  locally isometric to a product $\mathbb{R}\times N^3(\kappa)$, which occurs if $\lambda_1=\lambda_2=\lambda_3=\lambda$ in \eqref{eq:non-solvable} and $\kappa=\lambda^2/4$.
\end{theorem}
\begin{proof}
A left-invariant metric \eqref{eq:non-solvable} is $\mathcal{F}_t$-critical if and only if  the symmetric $(0,2)$-tensor field  
$
\mathfrak{F}_t=
-\Delta \rho 
+\frac{1}{2} (\|\rho\|^2 +t\tau^2) g 
-2 R[\rho] 
-2t\tau\rho
$
vanishes. 
The condition $\mathfrak{F}_t=0$ determines a system of polynomial equations on   the structure constants \eqref{eq:non-solvable} together with $t$. 
Since the structure constant  $\lambda_1$ is non-zero, we may consider a representative with $\lambda_1=1$ in the homothetic class. Moreover, we  introduce additional variables $\tilde\lambda_2$ and $\tilde\lambda_3$ to express that the structure constants $\lambda_2$ and $\lambda_3$ are non-zero by means of the polynomials $\lambda_2 \tilde\lambda_2-1$ and $\lambda_3 \tilde\lambda_3-1$.  
In the rest of the proof we will work in  the polynomial ring $\mathbb{R}[\tilde\lambda_3,\lambda_3,\tilde\lambda_2,\lambda_2,\lambda_1,k_1,k_2,k_3,t]$, where we consider the lexicographic order. 

In order to analyze the critical condition we consider the case  $k_1=$~$k_2=0$, the case $k_1=0$, $k_2k_3\neq 0$, and the case $k_1k_2k_3\neq0$ separately. In the second case we simplify the non-zero factors $k_2$ and $k_3$ by considering the polynomials  $\mathfrak{\overline F}_t(e_2,e_3)=\frac{1}{k_2k_3}\mathfrak{F}_t(e_2,e_3)$, $\mathfrak{\overline F}_t(e_2,e_4)=\frac{1}{k_2}\mathfrak{F}_t(e_2,e_4)$ and $\mathfrak{\overline F}_t(e_3,e_4)=\frac{1}{k_3}\mathfrak{F}_t(e_3,e_4)$. Similarly, if   $k_1k_2k_3\neq0$
we simplify the non-zero variables $k_1$, $k_2$, $k_3$ by multiplying the polynomials $\mathfrak{F}_t(e_i,e_j)$ with $i\neq j$ as in the previous case. Let $\mathfrak{\overline F}_t(e_i,e_j)$ denote the polynomials $\mathfrak{F}_t(e_i,e_j)$ including those simplifications in each case and let   $\mathcal{I}$  be the ideal   spanned by $\{\mathfrak{\overline F}_t(e_i,e_j)\}\cup\{\lambda_1-1, \lambda_2 \tilde\lambda_2-1, \lambda_3 \tilde\lambda_3-1\}$.

Starting with the case $k_1=k_2=0$, we   compute  a Gr\"obner basis   of the ideal 
${\mathcal{I}_1}=\langle \mathcal{I}\cup\{k_1,k_2\} \rangle$. We obtain $9$ polynomials, including the following:
$$
\mathbf{g}_{1}^1=3t+1,
\quad 
\mathbf{g}_{1}^2=(\lambda_2-1)^2,
\quad 
\mathbf{g}_{1}^3=2\lambda_3-\lambda_2-1.
$$
Thus $\lambda_1=\lambda_2=\lambda_3=1$ and $t=-1/3$. A straightforward calculation now shows that the underlying metric is locally conformally flat and isometric to a product $\mathbb{R}\times N^3(\kappa)$, where $N^3(\kappa)$ is a three-manifold of constant sectional curvature $\kappa=1/4$.

We analyze the case   $k_1=0$,  $k_2k_3\neq0$ by constructing  a Gr\"obner basis  of the ideal
${\mathcal{I}_2}=\langle \mathcal{I}\cup\{k_1\} \rangle$.  We obtain $26$ polynomials, among which we have 
$$
\mathbf{g}_{2}^{1}=(k_2^2+k_3^2+1)(3t+1),
\,\, 
\mathbf{g}_{2}^{2}=(k_2^2+k_3^2+1)(\lambda_2-1)^2,
\,\, 
\mathbf{g}_{2}^{3}=(k_2^2+k_3^2+1)(\lambda_3+\lambda_2-2) .
$$
We conclude that $\lambda_1=\lambda_2=\lambda_3=1$, $t=-1/3$ and proceed as in the previous case.

Finally, if   $k_1k_2k_3\neq0$  we compute a  Gr\"obner basis for the ideal $\mathcal{I}$. This basis consists in $86$ polynomials that include
$$
\begin{array}{l}
\mathbf{g}_{3}^{1}=(k_1^2+k_2^2+k_3^2+1)^2(3t+1),
\quad 
\mathbf{g}_{3}^{2}= (k_1^2+k_2^2+k_3^2+1)(t+1)(\lambda_2-1),

\\
\noalign{\medskip}

\mathbf{g}_{3}^{3}=(k_1^2+k_2^2+k_3^2+1)(2\lambda_2+2\lambda_3-9t-7) .
\end{array}
$$
Again, these lead to $\lambda_1=\lambda_2=\lambda_3=1$, $t=-1/3$, and the result follows.
\end{proof}

\begin{remark}\rm
Since $\mathcal{F}_t$-critical metrics in Theorem~\ref{th:non-solvable} have $\lambda_1$, $\lambda_2$ and $\lambda_3$ of the same sign, they are realized on $SU(2)\times\mathbb{R}$. Remarkably, 
	the product Lie group $\widetilde{SL}(2,\mathbb{R})\times\mathbb{R}$ does not admit any $\mathcal{F}_t$-critical left-invariant metric, although the three-dimensional Lie groups $SU(2)$ and $\widetilde{SL}(2,\mathbb{R})$ admit many $\mathcal{F}_t$-critical metrics as shown in Figure~\ref{figura-1} (see \cite{BV-GR-CO}). 
	 
	Critical metrics with zero energy exhibit a special behaviour. Indeed, let $(N,g_N)$ be a homogeneous manifold and let $M=\mathbb{R}\times N$ with $g=dr^2+g_N$ be the $n$-dimensional product manifold. Then a direct calculation of the symmetric tensor field $\mathfrak{F}_t=-\Delta \rho 
		+\tfrac{2}{n} (\|\rho\|^2 +t\tau^2) g 
		-2 R[\rho] 
		-2t\tau\rho$ shows that $\mathfrak{F}_t(\partial_r,X)=0$ for all vector fields $X$  tangent to $N$, and moreover
		$$
		\begin{array}{l}
		\mathfrak{F}_t(\partial_r,\partial_r)=\tfrac{2}{n} (\|\rho_N\|^2 +t\tau_N^2),
		\\
		\noalign{\medskip}
		\mathfrak{F}_t(X,Y)=\{-\Delta \rho_N 
		+\tfrac{2}{n-1} (\|\rho_N\|^2 +t\tau_N^2) g_N
		-2 R_N[\rho_N] 
		-2t\tau_N\rho_N\}(X,Y) \\
		\noalign{\medskip}
		\phantom{\mathfrak{F}_t(X,Y)=\{\}}
		-\tfrac{2}{n(n-1)}(\|\rho_N\|^2 +t\tau_N^2)g_N(X,Y),
		\end{array}
		$$
		for all vector fields $X,Y$ tangent to $N$. Hence it immediately follows that
		\begin{quote}
			\emph{For any homogeneous manifold $(N,g_N)$, the product manifold $\mathbb{R}\times N$ is critical for some quadratic curvature functional if and only if $N$ is critical for the same curvature functional with zero energy.}
		\end{quote} 
	The $\mathcal{F}_{-1/3}$-critical metric in Theorem~\ref{th:non-solvable} is obtained by the above construction. Since the sphere $\mathbb{S}^3$ is Einstein, it is critical for all quadratic curvature functionals and, since $\|\rho\|^2=\frac{1}{3}\tau^2$, the energy is zero for the functional $\mathcal{F}_{-1/3}$.
	
	The non-symmetric four-dimensional $\mathcal{F}_t$-critical product metrics correspond to $\mathbb{R}\times \mathcal{H}^3$  as in Theorem~\ref{th:heisenberg}--(1), $\mathbb{R}\times E(1,1)$ as in Theorem~\ref{th:E(1,1)}--(1), or semi-direct products $\mathbb{R}\ltimes\mathbb{R}^3$ determined by a diagonal derivation with eigenvalues $(1,f,0)$ as in Theorem~\ref{th:criticas-R3}--(3).

Remarkably, although any homogeneous three-dimensional manifold with a four-dimensional isometry group is critical for some quadratic curvature functional (see Figure~\ref{figura-1}), there are four-dimensional homogeneous manifolds with a higher-dimensional isometry group which are not critical for any quadratic curvature functional.
	 \end{remark}

\begin{remark}\rm
It follows from Theorem~\ref{th:non-solvable} that any non-symmetric four-dimensional homogeneous $\mathcal{F}_t$-critical metric is locally isometric to a left-invariant metric on a solvable Lie group.	
\end{remark}

\section{Left-invariant $\mathcal{F}_t$-critical metrics on $\mathbb{R}\ltimes E(1,1)$ and $\mathbb{R}\ltimes \widetilde{E}(2)$}\label{se:E(1,1)}

Let $\mathfrak{g}=\mathbb{R}\ltimes\mathfrak{g}_3$ be a semi-direct extension of the unimodular Lie algebra $\mathfrak{g}_3=\mathfrak{e}(1,1)$ or $\mathfrak{g}_3=\mathfrak{e}(2)$. Let $\langle \cdot, \cdot \rangle$ be an inner product on $\mathfrak{g}$ and $\langle \cdot, \cdot \rangle_3$ its restriction to $\mathfrak{g}_3$. Following the work of Milnor \cite{Milnor}, there exists an orthonormal basis $\{\mathbf{v}_1,\mathbf{v}_2, \mathbf{v}_3\}$ of $\mathfrak{g}_3$ such that 
\begin{equation}\label{basis_non-solvable}
[\mathbf{v}_2,\mathbf{v}_3]= \lambda_1 \mathbf{v}_1,
\quad 
[\mathbf{v}_3, 
\mathbf{v}_1]= \lambda_2 \mathbf{v}_2,
\quad
[\mathbf{v}_1,\mathbf{v_2}]=0,
\end{equation}
where $\lambda_1, \lambda_2 \in \mathbb{R}$ and $\lambda_1 \lambda_2 \neq 0$. Moreover, the associated Lie group corresponds to $\widetilde{E}(2)$ (resp., $E(1,1)$) if $\lambda_1, \lambda_2$ do not change sign (resp., change sign).
The algebra of derivations of $\mathfrak{g}_3$ is given by
\[ 
\operatorname{der}(\mathfrak{g}_3)=\left \{  \left( \begin{array}{ccc}
\tilde b & \tilde a& \tilde c
\\
-\frac{\lambda_2}{\lambda_1} \tilde a& \tilde b & \tilde d 
\\
0& 0 & 0 \end{array} \right);\, \tilde a, \tilde b, \tilde c, \tilde d \in \mathbb{R} \right \}.
\]
Let $\{\mathbf{v}_1, \mathbf{v}_2, \mathbf{v}_3, \mathbf{v}_4\}$ be a basis of $\mathfrak{g}$ such that $\{\mathbf{v}_1, \mathbf{v}_2, \mathbf{v}_3\}$ are given by \eqref{basis_non-solvable} and $\mathfrak{g}=\mathbb{R}\mathbf{v}_4\oplus  \mathfrak{g}_3$. 
Since $\mathbb{R}\mathbf{v}_4$ is not necessarily orthogonal to $\mathfrak{g}_3$, we set $\tilde k_i=\langle \mathbf{v}_i, \mathbf{v}_4 \rangle$, for $i = 1,2,3$. Let $\bar e_4= \mathbf{v}_4-\sum_i \tilde k_i\mathbf{v}_i$ and normalize it to get an orthonormal basis $\{ e_1,\dots,e_4\}$ of $\mathfrak{g}=  \mathbb{R} \oplus \mathfrak{g}_3 $, where $e_i=\mathbf{v}_i$ for $i=1,2,3$. 
Now, setting $a=-\tilde a\|\bar e_4\|^{-1}$, $b=-\tilde b\|\bar e_4\|^{-1}$, $c=-\tilde c\|\bar e_4\|^{-1}$, $d=-\tilde d\|\bar e_4\|^{-1}$, $A=-(\frac{\tilde a}{\lambda_1 }+\tilde k_3)\|\bar e_4\|^{-1}$, $C=-(\tilde c-\tilde k_2\lambda_1)\|\bar e_4\|^{-1}$ and $D=-(\tilde d+\tilde k_1 \lambda_2)\|\bar e_4\|^{-1}$ we express the Lie brackets as follows  
	\begin{equation}\label{eq:E(1,1), E(2)}
	\begin{array}{ll}
	[e_2,e_3]= \lambda_1 e_1,
	&\quad
	[e_1,e_3]= -\lambda_2 e_2,
	
	\\
	\noalign{\medskip}
	
	[e_1,e_4] = b e_1-\lambda_2 A e_2,
	&\quad
	[e_2,e_4]=\lambda_1 A e_1+ b e_2,
	
	\\
	\noalign{\medskip} 
	
	[e_3,e_4]=C e_1+ D e_2.
	\end{array}
	\end{equation}	

\begin{remark}\label{re:E11 E2 sc}
	\rm
	Since
	$\tau=-\frac{1}{2}\left\{(A^2+1)(\lambda_1-\lambda_2)^2+ 12 b^2+C^2+D^2\right\}$, the scalar curvature vanishes if and only if $\lambda_2=\lambda_1$ and
	$b=C=D=0$, in which case the metric is flat.
	A straightforward calculation shows that a metric \eqref{eq:E(1,1), E(2)} is Einstein if and only if it is flat or a product of two surfaces $N_1(\kappa)\times N_2(\kappa)$. In the latter case, the manifold is realized on  $\mathbb{R}\ltimes E(1,1)$ with $\lambda_1=-\lambda_2=\pm b$ and $A=C=D=0$.
	Furthermore, a non-Einstein left-invariant metric \eqref{eq:E(1,1), E(2)} is locally symmetric if and only if it is a locally conformally flat product of the form $\mathbb{R}\times N^3(\kappa)$ (which is realized on $\mathbb{R}\ltimes \widetilde{E}(2)$ with $\lambda_1=\lambda_2$  and $C=D=0$) or it is a product of two surfaces $N_1^ 2(\kappa_1)\times N_2^2 (\kappa_2)$ (which is realized on $\mathbb{R}\ltimes E(1,1)$ with $\lambda_1=-\lambda_2$ and $C=D=0$,  $b^2=(A^2+1)\lambda_2^2$). Any of the cases above is covered by the discussion in Section~\ref{ss:symmetric}.
\end{remark}

\begin{remark}\label{re:E11 E2 isometry}
	\rm
	Let $\langle\cdot,\cdot\rangle$ be a left-invariant metric on $\mathbb{R}\ltimes E(1,1)$ or $\mathbb{R}\ltimes \widetilde{E}(2)$ as in \eqref{eq:E(1,1), E(2)} with parameters $(\lambda_1,\lambda_2,A,b,C,D)$  determining the structure constants. The isometry $( e_1, e_2, e_3, e_4)\mapsto(-e_2, e_1, e_3, e_4)$ transforms the parameters into $(\lambda_2$, $\lambda_1$, $A$, $b$, $-D$, $C)$. Hence any left-invariant metric \eqref{eq:E(1,1), E(2)} with  $D=0$ is isomorphically isometric to a left-invariant metric with $C=0$.
\end{remark}

\begin{theorem}\label{th:E(1,1)}
	A non-symmetric left-invariant metric on  $\mathbb{R}\ltimes E(1,1)$ or $\mathbb{R}\ltimes \widetilde{E}(2)$ is critical for a quadratic curvature functional $\mathcal{F}_t$ if and only if it is isomorphically homothetic to one of the following:
	\begin{enumerate}
		\item The direct product $E(1,1)\times \mathbb{R}$ with Lie brackets given by $[e_1,e_3]= e_2$ and $[e_2,e_3]=e_1$. In this case $t=-1$ and $\mathcal{E}_{-1}=0$.
		\item $[e_1,e_3]=-\lambda e_2$, $[e_1,e_4]=b e_1$, $[e_2,e_3]=e_1$, $[e_2,e_4]=b e_2$, 
		\smallskip
		\\
		with $\lambda\in(-1,0)\cup (0,1)$. For a fixed $\lambda$, the parameter $b$ is given by the two positive solutions of $ 4b^4-(3\lambda^2-2\lambda+3)b^2-(\lambda-1)^2 \lambda=0$ if $\lambda\in(-1,0)$, and the unique positive solution if $\lambda\in(0,1)$.
		In both cases  $t=-\frac{3\lambda^2-2\lambda+3}{12 b^2+(\lambda-1)^2}$ and the energy is given by $\energy_t=4\lambda(\lambda-1)^2$.

		\smallskip
		
		\item $[e_1,e_3]=-e_2$, $[e_1,e_4]=  e_1$, $[e_2,e_3]=e_1$, $[e_2,e_4]=  e_2$, $[e_3,e_4]=D e_2$,
		\smallskip
		\\
		with 
		$D> 0$. 
		In this case, $t=-\frac{3D^2+4}{D^2+12}$ and   the energy satisfies $\energy_t=-5D^2$.

		\smallskip
		\item $[e_1,e_3]=-\lambda e_2$, $[e_1,e_4]=b e_1$, $[e_2,e_3]=e_1$, $[e_2,e_4]=b e_2$, $[e_3,e_4]=D e_2$, 
		\smallskip
		\\
		with $\lambda\in(0,1)$. For a fixed $\lambda$, the parameters $D$ and $b$ are given by the only positive solution of $D^2+\lambda^2=1$ and $2b^4-(5\lambda^2-\lambda-2)b^2+\lambda(\lambda-1)=0$, respectively.
	In this case,  $t=-\frac{3b^2-2\lambda+3}{12 b^2-2\lambda+2}$ and  $\energy_t=(\lambda-1)
		\left(3  (\lambda + 3) b^2 - 2 \lambda
		\right)$.

		\smallskip
		\item  $[e_1,e_3]= e_2$, $[e_1,e_4]=b e_1+A e_2$, $[e_2,e_3]=e_1$,  $[e_2,e_4]=A e_1+b e_2$, 
		\smallskip
		\\
		with $b> 0$, $A\geq 0$ and $b^2-A^2\neq 1$.
		In this case, $t=-\frac{A^2+b^2+1}{A^2+3b^2+1}$ and $\energy_t=-16b^2$.
	\end{enumerate}
 Moreover, metrics in Family~(1) are the only non-Einstein algebraic Ricci solitons on $\mathbb{R}\ltimes E(1,1)$ (there are none on $\mathbb{R}\ltimes \widetilde{E}(2)$).
\end{theorem}

\begin{remark}\rm\label{re:E-grafico}
	The non-symmetric $\mathcal{F}_t$-critical metrics in Theorem~\ref{th:E(1,1)} are realized on $\mathbb{R}\ltimes E(1,1)$ or $\mathbb{R}\ltimes \widetilde{E}(2)$  as indicated in Figure~\ref{figura-E} below. Note that, although critical metrics on the three-dimensional Euclidean group $\widetilde{E}(2)$ are flat \cite{BV-GR-CO}, the semi-direct extension $\mathbb{R}\ltimes \widetilde{E}(2)$ admits many non-symmetric $\mathcal{F}_t$-critical metrics. The range of the parameter $t$ is also included in each case. 
	\begin{figure}[h]
	\begin{center}
		\begin{tikzpicture}
		\draw[thick,-latex]   (-3.1-0.15,0) -- (5.7,0) node[right] {$t$}; 
		
		\foreach \i/\n in 
		{-2.7/$-3$, {-0.2}/$-\frac{3}{2}$, {1.1}/$-1$, 
			{2.4}/$-\frac{1}{2}$, {3.05}/$-\frac{1}{3}$, {3.65}/$-\frac{3}{10}$, {4.45}/{\tiny $5-2\sqrt{7}$}
		}{\draw (\i,-.2)--(\i,.2);}
		
		\foreach \i/\n in 
		{-2.7/$-3$, {1.1}/$-1$
		}{\draw  (\i,.2) node[above] {\footnotesize \n};}
		\foreach \i/\n in 
		{{4.45}/{$5-2\sqrt{7}$}
		}{\draw  (\i,.215) node[above] {\tiny  \n};}
		\foreach \i/\n in 
		{{-0.2}/$-\frac{3}{2}$, {2.4}/$-\frac{1}{2}$, {3.05}/$-\frac{1}{3}$, {3.65}/$-\frac{3}{10}$
		}{\draw  (\i,.12) node[above] {\footnotesize \n};}
		
		\foreach \i/\n in 
		{
			-2.7/$-3$, {-0.2}/$-\frac{3}{2}$, {1.1}/$-1$, 
			{2.4}/$-\frac{1}{2}$, {3.05}/$-\frac{1}{3}$, {3.65}/$-\frac{3}{10}$, {4.45}/{\tiny $5-2\sqrt{7}$}
		}{
			\draw[dashed] (\i,-.25)--(\i,-3.27);
		}

		\draw (-5.45,-0.45) node {\footnotesize (1) };
		\draw (-4.2,-0.45) node {\footnotesize $\mathbb{R}\ltimes E(1,1)$\phantom{.}}; 
		\filldraw [color=\ColorECeroARS,fill=\ColorECeroARS] (1.1,-0.45) circle (2pt);

		\draw (-5.45,-1.2) node {\footnotesize (2) };
		\draw (-4.2,-1.2) node { \footnotesize $\begin{cases} 
			{\mathbb{R}\ltimes E(1,1)} \\
			{\mathbb{R}\ltimes \widetilde{E}(2)}
			\end{cases}$};
		\draw[color=\ColorENegativa] (-2.7,-1)--(3.65,-1); 
		\filldraw [color=\ColorENegativa,fill=white] (-2.7,-1) circle (2pt);
		\filldraw [color=\ColorENegativa,fill=white] (2.4,-1) circle (2pt);
		\filldraw [color=\ColorENegativa,fill=white] (3.65,-1) circle (2pt); 
			
		\draw[color=\ColorEPositiva] (3.05,-1.4-\YSepDosMetricas-0.03)--(4.45,-1.4-\YSepDosMetricas-0.03);
		\filldraw [color=\ColorEPositiva,fill=white] (3.05,-1.4-\YSepDosMetricas-0.03) circle (2pt);
		\filldraw [color=\ColorEPositiva,fill=\ColorEPositiva] (4.45,-1.4-\YSepDosMetricas-0.03) circle (2pt); 
			
		\draw[color=\ColorEPositiva] 		(3.65,-1.4+1*\YSepDosMetricas)--(4.45,-1.4+1*\YSepDosMetricas);
		\filldraw [color=\ColorEPositiva,fill=white] (3.65,-1.4+1*\YSepDosMetricas) circle (2pt);
		\filldraw [color=\ColorEPositiva,fill=white] (4.45,-1.4+1*\YSepDosMetricas) circle (2pt);

		\draw (-5.45,-1.95) node {\footnotesize (3) };
		\draw (-4.2,-1.95) node {\footnotesize $\mathbb{R}\ltimes \widetilde{E}(2) \phantom{(1)}$};
		\draw[color=\ColorENegativa] (-2.7,-1.95)--(3.05,-1.95);
		\filldraw [color=\ColorENegativa,fill=white] (-2.7,-1.95) circle (2pt);
		\filldraw [color=\ColorENegativa,fill=white] (3.05,-1.95) circle (2pt);

		\draw (-5.45,-2.45) node {\footnotesize (4) };
		\draw (-4.2,-2.45) node {\footnotesize $\mathbb{R}\ltimes \widetilde{E}(2) \phantom{(1)}$};
		\draw[color=\ColorENegativa] (-.2,-2.45)--(3.05,-2.45);
		\filldraw [color=\ColorENegativa,fill=white] (-.2,-2.45) circle (2pt);
		\filldraw [color=\ColorENegativa,fill=white] (3.05,-2.45) circle (2pt);

		\draw (-5.45,-2.95) node {\footnotesize (5) };
		\draw (-4.2,-2.95) node {\footnotesize $\mathbb{R}\ltimes E(1,1)$\phantom{.}}; 
		
		\draw[densely dotted,color=\ColorENegativa] 	
			(1.1,-2.95-\YSepInfMetricas/4)--(3.05,-2.95-\YSepInfMetricas/4);
				\draw[densely dotted,color=\ColorENegativa] 	
			(1.1,-2.95-\YSepInfMetricas/8)--(3.05,-2.95-\YSepInfMetricas/8);
				\draw[densely dotted,color=\ColorENegativa] 	
			(1.1,-2.95-\YSepInfMetricas/3)--(3.05,-2.95-\YSepInfMetricas/3); 
		
		\draw[color=\ColorENegativa] (1.1,-2.95-\YSepInfMetricas/2)--(3.05,-2.95-\YSepInfMetricas/2);
		\filldraw [color=\ColorENegativa,fill=white] (1.1,-2.95-\YSepInfMetricas/2) circle (2pt);
		\filldraw [color=\ColorENegativa,fill=white] (2.4,-2.95-\YSepInfMetricas/2) circle (2pt);
		\filldraw [color=\ColorENegativa,fill=white] (3.05,-2.95-\YSepInfMetricas/2) circle (2pt);

	\end{tikzpicture}
\caption{Range of the parameter $t$ for non-symmetric homogeneous $\mathcal{F}_t$-critical metrics on $\mathbb{R}\ltimes E(1,1)$ and $\mathbb{R}\ltimes \widetilde{E}(2)$.}
\label{figura-E}	
	\end{center}
\end{figure}
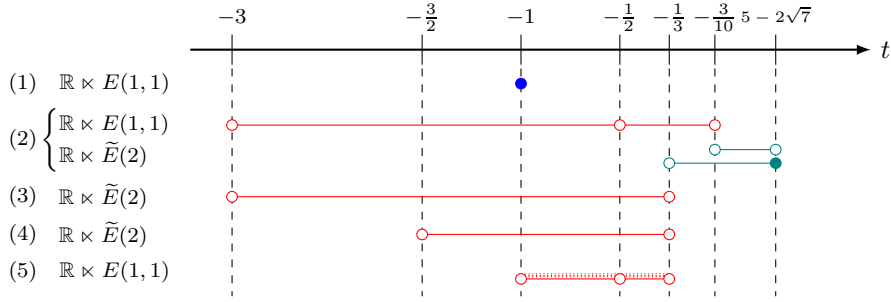
The metric in Family~(1), which is $\mathcal{F}_{-1}$-critical, is the only critical metric (modulo homotheties) with zero energy. Moreover, it is an algebraic Ricci soliton on $\mathbb{R}\ltimes E(1,1)$, which is isomorphically  homothetic  to an algebraic Ricci soliton (iii) in Section~\ref{ss:1-3-2} (see  Remark~\ref{re:xxx}). Left-invariant metrics in $\mathbb{R}\ltimes \widetilde{E}(2)$ corresponding to Family~(2) with $\lambda\in(0,1)$ have positive energy, whereas the energy is negative for all other families.
	The only non-symmetric Bach-flat metric  is the left-invariant metric on $\mathbb{R}\ltimes E(1,1)$ corresponding to Family~(2) for $t=-1/3$ (see \cite{CL-GM-GR-GR-VL} and Remark~\ref{re:Bach-llanas}).
	 
		The value of $t$ for critical metrics is a homothetic invariant. Other homothetic invariants can be built as quotients of second and third order scalar curvature invariants as long as the numerator and the denominator have the same order. The set of homothetic invariants given by $\{t,\|\rho\|^2\,\tau^{-2}, \| R\|^2\,\tau^{-2}, \|\nabla\rho\|^2\,\tau^{-3},\|\nabla R\|^2\tau^{-3}\}$ suffices to show that there are no homotheties between metrics in different classes in Theorem~\ref{th:E(1,1)}.
	Moreover, it follows from the restrictions on the parameters in each Family~(2)--(5) that different values of the parameters correspond to different homothety classes.

Furthermore, for each admissible value of $t\in(-3,5-2\sqrt{7}]$ there is a single $\mathcal{F}_t$-critical metric (up to homotheties) for each Family~(2)-(4) in Figure~\ref{figura-E},  with the exception of Family~(2), where there are two-distinct homothety classes of $\mathcal{F}_t$-critical metrics for $  t\in(-\frac{3}{10},5-2\sqrt{7})$. For each value of  $t\in(-1,-\frac{1}{3})\setminus\{-\frac{1}{2}\}$ there is an infinite number of non-homothetic $\mathcal{F}_t$-critical metrics in Family~(5).
\end{remark}

\begin{proof} 
	A left-invariant metric is $\mathcal{F}_t$-critical if and only if the symmetric $(0,2)$-tensor field 
	$
	\mathfrak{F}_t=
	-\Delta \rho 
	+\frac{1}{2} (\|\rho\|^2 +t\tau^2) g 
	-2 R[\rho] 
	-2t\tau\rho
	$
	vanishes. Moreover, the components $\mathfrak{F}_t(e_i,e_j)$ are polynomials on $t$ and on the structure constants  determining the metric~\eqref{eq:E(1,1), E(2)}. A direct calculation shows that $\mathfrak{F}_t(e_1,e_4)=\lambda_2 \mathcal{P}_{14}$ and $\mathfrak{F}_t(e_2,e_4)=\lambda_1 \mathcal{P}_{24}$, where $\mathcal{P}_{14}$ and $\mathcal{P}_{24}$ are polynomials in the variables $(\lambda_1,\lambda_2,A,b,C,D,t)$. Thus, we simplify the non-zero parameters $\lambda_2$ and $\lambda_1$ to define $\mathfrak{\overline F}_t(e_i,e_4)=\mathcal{P}_{i4}$ for $i=1,2$, and we let $\mathfrak{\overline F}_t(e_i,e_j)=\mathfrak{F}_t(e_i,e_j)$ in other cases. In what follows we consider the cases $C=0$ and $C\neq 0$ separately in order to solve the system 
	$\{\mathfrak{\overline F}_t(e_i,e_j)=0\}$.

	\subsection{\boldmath Case $C=0$\unboldmath}\label{se: E11 E2 C=0}
	Assuming $C=0$ we consider a homothety to make $\lambda_1=1$. Note that the isometry $e_4\mapsto -e_4$ interchanges $(\lambda_2,A,b,D)$ and $(\lambda_2,-A,-b,-D)$, so we can take $b\geq 0$ in all the subcases below without loss of generality.  We start by calculating
	\[
	\begin{array}{l}
	\mathfrak{\overline F}_t(e_1,e_4)=
	-\tfrac{1}{2} D \left(
	\left(  
	(\lambda_2 - 1)^2 (A^2 + 1) + 12 b^2 + D^2
	\right) t 
	\right.
	
	\\
	\noalign{\medskip}
	\phantom{\mathfrak{\overline F}_t(e_1,e_4)=
		-\tfrac{1}{2} D (}
	
	\left.
	+ \lambda_2 (3 \lambda_2 - 2) (A^2+1) 
	+ 3 b^2 + 3 D^2
	\right),
	\end{array}
	\]
	which leads to the cases $D=0$ and 
	$t=-\frac{\lambda_2 (3 \lambda_2 - 2) (A^2+1) 
		+ 3 b^2 + 3 D^2}
	{(\lambda_2 - 1)^2 (A^2 + 1) + 12 b^2 + D^2}
	$ with $D\neq 0$. We analyze these cases separately.

	\subsubsection{Case $D=0$}
	For $D=0$ we have
	\[
	\begin{array}{l}
	\mathfrak{\overline F}_t(e_3,e_4)=
	-\tfrac{1}{2}  (\lambda_2 - 1)^2 A 
	\left(
	\left(
	(A^2 + 1) (\lambda_2 - 1)^2 + 
	12 b^2
	\right) t 
	\right.
	
	\\
	\noalign{\medskip}
	\phantom{\mathfrak{\overline F}_t(e_3,e_4)=
		-\tfrac{1}{2}  (\lambda_2 - 1)^2 A (}
	
	\left.
	+ (3 \lambda_2^2 + 2 \lambda_2 + 3) (A^2 + 1) 
	+ 4 b^2
	\right).
	\end{array}
	\] 
	If $\lambda_2=1$, then the space is
	symmetric (see  Remark~\ref{re:E11 E2 sc}). We assume $\lambda_2\neq 1$ and analyze the following two possibilities:    
	\begin{itemize}
		\item[(i)]  $A=0$. The non-null components of  $\mathfrak{\overline F}_t$ reduce to
		\noindent
		\begin{equation}\label{eq: E(1,1) E(2)-1}
		\begin{array}{l}
		\mathfrak{\overline F}_t(e_1,e_1)=
		-\frac{1}{8}  \left(12 b^2 + (\lambda_2 - 
		1)^2\right) 
		\left(4 b^2 + (\lambda_2 - 1) (3 \lambda_2 + 5)\right) t 
		
		\\
		\noalign{\medskip}
		\phantom{\mathfrak{\overline F}_t(e_1,e_1)=}
		
		- 2 b^4 - 
		\frac{1}{8} (\lambda_2 - 1) \left(9 \lambda_2^3 + 5 \lambda_2^2 + 
		3 \lambda_2 + 15\right),
		
		\\
		\noalign{\medskip}
		
		\mathfrak{\overline F}_t(e_2,e_2)=
		-\frac{1}{8}  \left(12 b^2 + (\lambda_2 - 
		1)^2\right) \left(4 b^2 - (\lambda_2 - 1) (5 \lambda_2 + 3)\right) t 
		
		\\
		\noalign{\medskip}
		\phantom{\mathfrak{\overline F}_t(e_2,e_2)=}
		
		- 2 b^4 + 
		\frac{1}{8} (\lambda_2 - 1) \left(15 \lambda_2^3 + 3 \lambda_2^2 + 
		5 \lambda_2 + 9\right) ,
		
		\\
		\noalign{\medskip}
		
		\mathfrak{\overline F}_t(e_3,e_3)=
		- 3 \mathfrak{\overline F}(e_4,e_4)=
		\frac{3}{8} \left(48 b^4 - 8 b^2 (\lambda_2 - 1)^2 - (\lambda_2 - 1)^4\right) t 
		
		\\
		\noalign{\medskip}
		\phantom{\mathfrak{\overline F_t}(e_3,e_3)=}
		+  6 b^4 - \frac{3}{8} (\lambda_2 - 1)^2 \left(3 \lambda_2^2 + 2 \lambda_2 + 3\right) ,
		
		\end{array}
		\end{equation}
		and a direct calculation shows that
		\[
		\mathfrak{\overline F}_t(e_1,e_1)
		-\mathfrak{\overline F}_t(e_2,e_2) 
		=
		-(\lambda_2^2-1)
		\left(  \left(12 b^2 + 
		(\lambda_2-1)^2\right)t 
		+ 3 \lambda_2^2 - 2 \lambda_2+3 \right) .
		\]

		\smallskip
		
		If $\lambda_2=-1$ and $b\neq\pm1$ (the metric is Einstein if $b^2=1$), then using (\ref{eq: E(1,1) E(2)-1}) one easily gets that  $\mathfrak{\overline F}_t$ vanishes if and only if $(3 b^2 + 1) t + b^2 + 1=0$. Therefore, the corresponding left-invariant metric on $\mathbb{R}\ltimes E(1,1)$, given by
		\[
		[e_1,e_3]= e_2,
		\quad 
		[e_1,e_4]=b e_1, 
		\quad 
		[e_2,e_3]=e_1,  
		\quad
		[e_2,e_4]=b e_2, 
		\]
		is critical for the functional $\mathcal{F}_t$ determined by $t=-\frac{b^2+1}{3b^2+1}\in[-1,-\frac{1}{3})\setminus \{-\frac{1}{2}\}$. Moreover, we have  $\energy_t=-16b^2$  and hence the energy vanishes if and only if $b=0$. Also, these metrics are algebraic Ricci solitons if and only if $b=0$,  in which case $\operatorname{Ric}+2\operatorname{Id}$ is a derivation of the Lie algebra. These metrics correspond to Family (1) if $b=0$, and a subfamily of Family~(5) given by $A=0$ otherwise.

		\smallskip

		If
		$\left(12 b^2 + 
		(\lambda_2-1)^2\right)t 
		+ 3 \lambda_2^2 - 2 \lambda_2+3=0$ and $\lambda_2\neq\pm1$,
		then by (\ref{eq: E(1,1) E(2)-1})  the critical condition reduces to
		\[
		4 b^4 - (3 \lambda_2^2 - 2 \lambda_2 + 3) b^2 - (\lambda_2 - 
		1)^2 \lambda_2 = 0.
		\]
		 The corresponding left-invariant metric, given by
		\[
		[e_1,e_3]= -\lambda_2 e_2,
		\quad 
		[e_1,e_4]=b e_1, 
		\quad 
		[e_2,e_3]=e_1,  
		\quad
		[e_2,e_4]=b e_2, 
		\]
		is $\mathcal{F}_t$-critical for  $t=-\frac{3\lambda_2^2-2\lambda_2+3}{12 b^2+(\lambda_2-1)^2}\in(-3,5-2\sqrt{7}]\setminus\{-\frac{1}{2}\}$, in which case the energy is given by  $\energy_t=4\lambda_2(\lambda_2-1)^2\neq 0$. More specifically, if $\lambda_2<0$ then the metric occurs on  $\mathbb{R}\ltimes E(1,1)$ and $t\in(-3,-\frac{3}{10})\setminus\{-\frac{1}{2}\}$, while for $\lambda_2>0$   	the metric occurs on  $\mathbb{R}\ltimes \widetilde{E}(2)$ and
		$t\in(-\frac{1}{3}, 5-2\sqrt{7}]$.
		These metrics correspond to Family~(2).
	Now, we proceed depending on the sign of $\lambda_2$.
	If $\lambda_2<0$, then 	$(e_1,e_2,e_3,e_4)\mapsto\frac{1}{\lambda_2}(e_2,-e_1,e_3,-e_4)$ determines a homothety which interchanges the parameters $(\lambda_2,b)$ with $(\frac{1}{\lambda_2},-\frac{b}{\lambda_2})$. Moreover, the homothety
	$(e_1,e_2,e_3,e_4)\mapsto\frac{1}{\lambda_2}(e_2,-e_1,e_3,e_4)$ interchanges the parameters $(\lambda_2,b)$ with $(\frac{1}{\lambda_2},\frac{b}{\lambda_2})$ if $\lambda_2>0$. Hence one may restrict the parameter $\lambda_2$ to $(-1,1)\setminus\{0\}$.	
		
		\smallskip
		
		\item[(ii)]   $t=
			-\frac{(3 \lambda_2^2 + 2 \lambda_2 + 3) (A^2 + 1) + 4 b^2}
			{(A^2 + 1) (\lambda_2 - 1)^2 + 12 b^2}
			$, with $A\neq 0$ and $\lambda_2\neq 1$. The non-null components of $\mathfrak{\overline F}_t$ are given by
		\[
		\begin{array}{l}
		\mathfrak{\overline F}_t(e_1,e_1)=
		\phantom{-}
		
		(\lambda_2 + 1) (A^2 + 
		1)  \left( 2 \lambda_2 (\lambda_2 - 1) (A^2 + 1) + (3 \lambda_2 - 
		1) b^2 \right) ,
		
		\\
		\noalign{\medskip}
		
		\mathfrak{\overline F}_t(e_2,e_2) =
		- (\lambda_2 + 1)(A^2 + 1)  \left( 2 \lambda_2 (\lambda_2 - 1)  (A^2 + 1)  + (\lambda_2 - 
		3) b^2 \right) ,
		
		\\
		\noalign{\medskip}
		
		\mathfrak{\overline F}_t(e_3,e_3) =
		-3\mathfrak{\overline F}_t(e_4,e_4) =
		-3 (\lambda_2 + 1)^2 (A^2 + 1) b^2 .
		\end{array}
		\]
		Hence the critical condition is equivalent to $\lambda_2=-1$ and  the left-invariant metric, which occurs on  $\mathbb{R}\ltimes E(1,1)$, is given by
		\[
		[e_1,e_3]= e_2,
		\quad 
		[e_1,e_4]=b e_1+A e_2, 
		\quad 
		[e_2,e_3]=e_1,  
		\quad
		[e_2,e_4]=A e_1+b e_2 . 
		\]
		 The isometry $( e_1, e_2, e_3,  e_4)\mapsto(-e_1,e_2,-e_3,e_4)$ changes the parameters $(b,A)$ to $(b,-A)$, so we may assume without loss of generality that $A>0$. 
		
		Moreover, the metric is locally symmetric if and only if $b^2-A^2=1$. Thus,  
		$t=-\frac{A^2+b^2+1}{A^2+3b^2+1}\in[-1,-\frac{1}{3})\setminus \{-\frac{1}{2}\}$ and the energy is given by $\energy_t=-16b^2$. For $b\neq 0$ these metrics correspond to Family (5).  
	For $b=0$, consider the change of basis given by		
		$$
	\bar e_1 =\tfrac{e_1}{\sqrt{\mu}},\quad
	\bar e_2 = \tfrac{e_2}{\sqrt{\mu}},\quad
	\bar e_3 = \tfrac{1}{\mu}(e_4-Ae_3),
	\quad
	\bar e_4 = \tfrac{1}{\mu}(e_3+Ae_4),
	$$
	where $\mu=A^2+1$, to check that this family is isomorphically homothetic to Family~(1). Furthermore, in this case $\operatorname{Ric}+2\operatorname{Id}$ is a derivation of the Lie algebra and is the only instance in which these metrics are algebraic Ricci solitons.	
	\end{itemize}

	\subsubsection{Case $t=-\frac{\lambda_2 (3 \lambda_2 - 2) (A^2+1)+ 3 b^2 + 3 D^2}
		{(\lambda_2 - 1)^2 (A^2 + 1) + 12 b^2 + D^2}
		$, $D\neq 0$}
	For this value of $t$ we have $\mathfrak{\overline F}_t(e_1,e_3)=-\frac{3}{2} (\lambda_2 + 1) A b^2 D 	
	$ so we have the following three possibilities:
	\begin{itemize}
		\item[(i)] $\lambda_2=-1$. In this case we have:
		\[
		\begin{array}{l}
		\mathfrak{\overline F}_t(e_1,e_1)=
		-\frac{1}{8} \left(
		4 (A^2 - b^2 + 1)^2 + (19 (A^2 + 1) - 11 b^2) D^2
		\right) ,
		
		\\
		\noalign{\medskip}
		
		\mathfrak{\overline F}_t(e_2,e_2)=
		-\frac{1}{8} \left(
		4 (A^2 - b^2 + 1)^2 - (5 A^2 - 13 b^2 + 5) D^2
		\right) ,
		\end{array}
		\]
		from where 
		\[
		\mathfrak{\overline F}_t(e_1,e_1)+
		\mathfrak{\overline F}_t(e_2,e_2)
		=
		-(A^2 - b^2 + 1)^2 - \tfrac{1}{4} (7 A^2 + b^2 + 7) D^2,
		\]
		which is strictly negative since $D\neq 0$. Hence there are no critical metrics in this case.

		\smallskip

		\item[(ii)]  $b=0$. We compute
		\[
		\small
		\begin{array}{l}
		\mathfrak{\overline F}_t(e_1,e_1)=
		\frac{1}{8}( A^2+1) \left( 
		(\lambda_2 - 1) (4 \lambda_2^2 - 13 \lambda_2 - 
		15) (A^2 + 1) + (4 \lambda_2 - 15) D^2 
		\right) ,
		
		\\
		\noalign{\medskip}
		
		\mathfrak{\overline F}_t(e_2,e_2)=
		\frac{1}{8} \left(
		A^2 +  1) ( (\lambda_2 - 1) (4 \lambda_2^2 + 11 \lambda_2 + 9) (A^2 + 
		1) + (4 \lambda_2 + 9) D^2 
		\right) .
		\end{array}
		\]
		Hence,
		\[
		(4\lambda_2+9)\mathfrak{\overline F}_t(e_1,e_1)-
		(4\lambda_2-15)\mathfrak{\overline F}_t(e_2,e_2)
		=
		-6 \lambda_2 (\lambda_2 - 1) (A^2 + 1)^2,
		\]
		which implies $\lambda_2=1$.  In this case, $\mathfrak{\overline F}_t(e_1,e_1)=
		-\frac{11}{8} ( A^2+1) D^2\neq 0$, so there are no critical metrics.

		\smallskip
		
		\item[(iii)]  $A=0$. We determine  
		the non-zero components of $\mathfrak{\overline F}_t$, which are  given by
		\[
		\begin{array}{l}
		\mathfrak{\overline F}_t(e_1,e_1)=
		-\frac{1}{8} \left(
		4 b^4 - (11 D^2 + 21 \lambda_2^2 - 2 \lambda_2 - 
		15) b^2 - (4 \lambda_2 - 15) D^2
		\right.
		
		\\
		\noalign{\medskip}
		\phantom{\mathfrak{\overline F}_t(e_1,e_1)=-\frac{1}{8}(}
		
		\left.
		- (\lambda_2 - 
		1) (4 \lambda_2^2 - 13 \lambda_2 - 15)
		\right) ,
		
		\\
		\noalign{\medskip}
		
		\mathfrak{\overline F}_t(e_2,e_2) =
		-\frac{1}{8} \left(
		4 b^4 + ( 13 D^2 +3 \lambda_2^2 + 2 \lambda_2- 9) b^2 - (4 \lambda_2 + 9) D^2
		
		\right.
		\\
		\noalign{\medskip}
		\phantom{\mathfrak{\overline F}_t(e_2,e_2) =
			-\frac{1}{8}(}
		\left.
		
		- (\lambda_2 - 
		1) (4 \lambda_2^2 + 11 \lambda_2 + 9)
		\right) ,

		\\
		\noalign{\medskip}
		
		\mathfrak{\overline F}_t(e_3,e_3) =
		-3\mathfrak{\overline F}_t(e_4,e_4)
		=
		\frac{3}{8} \left(
		4 b^4 + (D^2 - 9 \lambda_2^2 + 2 \lambda_2 + 3) b^2 
		\right.
		
		\\
		\noalign{\medskip}
		\phantom{\mathfrak{\overline F}_t(e_3,e_3)=-\frac{3}{8}(}
		
		\left.
		- (4 \lambda_2 - 3) D^2 
		- (\lambda_2 -  1)^2 (4 \lambda_2 + 3)
		\right) ,
		
		\end{array}
		\]
		and therefore
		\[
		\mathfrak{\overline F}_t(e_1,e_1)-
		\mathfrak{\overline F}_t(e_2,e_2)=
		3 ( b^2-1) (D^2 + \lambda_2^2-1) .
		\]

		\smallskip
		
		If $b^2=1$ then   $\mathfrak{\overline F}_t=0$ reduces to
		$
		( \lambda_2-1) \left(D^2 + ( \lambda_2+1)^2\right)=0
		$, so $\lambda_2=1$. Recall that we can take $b\geq 0$. Thus, for $b=1$, we have the left-invariant metric on  $\mathbb{R}\ltimes \widetilde{E}(2)$ given by
		\[
		[e_1,e_3]= - e_2,
		\,\,\, 
		[e_1,e_4]=e_1, 
		\,\,\, 
		[e_2,e_3]=e_1,  
		\,\,\,
		[e_2,e_4]= e_2,
		\,\,\,
		[e_3,e_4]=D e_2, 
		\]
		which is critical for $t=-\frac{3D^2+4}{D^2+12}\in(-3,-\frac{1}{3})$. The isometry $(e_1,e_2,e_3,e_4)\mapsto(-e_1,-e_2,e_3,e_4)$ interchanges $D$ and  $-D$, and thus one may assume $D>0$. These metrics correspond to Family~(3). Their energy, which is given by $\energy_t=-5D^2$, is strictly negative.

		\smallskip

		If  $D^2+\lambda_2^2=1$, with   $b^2\neq 1$, then the tensor field $\mathfrak{\overline F}_t$ above vanishes if and only if 
		\begin{equation}\label{eq: E11 E2 caso (3)}
		2 b^4 - (5 \lambda_2^2 - \lambda_2 - 2) b^2 + \lambda_2(\lambda_2 - 1) =0  ,
		\end{equation}
		so we get the left-invariant metric
		\[
		[e_1,e_3]= - \lambda_2 e_2,
		\,\,
		[e_1,e_4]=b e_1, 
		\,\, 
		[e_2,e_3]=e_1,  
		\,\,
		[e_2,e_4]= b e_2,
		\,\,
		[e_3,e_4]=D e_2, 
		\]
		which is critical for $t=
		-\frac{3b^2-2\lambda_2+3}{12b^2-2\lambda_2+2}\in(-\frac{3}{2},-\frac{1}{3})$.
		Note that,  since $D^2+\lambda_2^2=1$ and $D\neq 0$, we have $\lambda_2\in(-1,1)$, but from (\ref{eq: E11 E2 caso (3)}) it follows that $\lambda_2>0$, so $\lambda_2\in(0,1)$ and the metric above  occurs on $\mathbb{R}\ltimes \widetilde{E}(2)$.
		The isometry $(e_1,e_2,e_3,e_4)\mapsto (-e_1,e_2,-e_3,e_4)$ interchanges the parameters $(\lambda_2,b,D)$ with  $(\lambda_2,b,-D)$, thus showing that one may restrict to $D>0$.
		Furthermore, the energy is given by $\energy_t=(\lambda_2-1)(3  (\lambda_2 + 3) b^2 - 2 \lambda_2)$, which  does not vanish under the conditions above. 
		Indeed, since $\lambda_2\in(0,1)$, the energy would vanish if and only if $b^2=\frac{2\lambda_2}{3(\lambda_2+3)}$. Now, using this expression for $b^2$ in 
		(\ref{eq: E11 E2 caso (3)}), we get that
		$21 \lambda_2^3 + 39 \lambda_2^2 - 65 \lambda_2 + 45=0$, which is not possible for $\lambda_2\in(0,1)$.
		These metrics correspond to Family~(4).
	\end{itemize}

	\subsection{\boldmath  Case $ C\neq 0$\unboldmath}
	First of all note that, by Remark~\ref{re:E11 E2 isometry}, we may assume $D\neq 0$.
	Moreover, we choose a representative in the homothetic class with $C=1$. We use  Gröbner bases to show that no  other critical metrics than those obtained in the previous section may exist. We fix  the polynomial ring $
	\mathbb{R}[t, \lambda_1, \lambda_2,A,b,C,D]$ with   the lexicographic order and   denote by  
	$\mathcal{I}$  the ideal spanned  by the polynomials $\{\mathfrak{\overline F}_t(e_i,e_j)\} \cup \{  C-1\}$. A direct computation of
	a  Gröbner basis of this ideal provides $116$ polynomials, which include the following:
	\[
	\mathbf{g}=Ab^3(D^2-1) (16A^2+9)(D^2+1)^3(A^2D^4+6A^2D^2+A^2+9D^2) .
	\]  
	Hence,  it follows that $b=0$, $A=0$ or $D^2=1$.

	\subsubsection{Case $b=0$}
	We compute  a Gröbner basis of the ideal $\mathcal{I}_{1}=\langle\mathcal{I}\cup \{b\} \rangle$, obtaining $18$ polynomials. The polynomial    
		$
		\mathbf{g}_{1}=\lambda_2^2 D (A^2+1)(2(A^2+1)\lambda_2^2+7)
		$
		belongs to the basis so, since $\lambda_2 D\neq 0$,  it follows that   no critical metric  may exist in this case.

	\subsubsection{Case $A=0$,   $b\neq 0$}
	In this case, $\mathfrak{\overline F}_t(e_3,e_4)
	= -\frac{3}{2} b D(\lambda_1-\lambda_2)$, so necessarily $\lambda_2=\lambda_1$ and the left-invariant metric is given by
	\[
	[e_1,e_3]= - \lambda_1 e_2,
	\,\,
	[e_1,e_4]=b e_1, 
	\,\, 
	[e_2,e_3]=\lambda_1 e_1,  
	\,\,
	[e_2,e_4]= b e_2,
	\,\,
	[e_3,e_4]=e_1 + D e_2 . 
	\]
	Considering the orthonormal basis 
	\[
	\bar e_1 = \tfrac{1}{\sqrt{D^2+1}}(D e_1-e_2), 
	\quad
	\bar e_2 = \tfrac{1}{\sqrt{D^2+1}}(e_1+D e_2), 
	\quad
	\bar e_3 = e_3, 
	\quad
	\bar e_4 =  e_4,
	\]
	a direct calculation shows that the Lie brackets are determined by
	\[
	[\bar e_1,\bar e_3]=-\lambda_1\bar e_2, 
	\,
	[\bar e_1,\bar e_4]= b  \bar e_1,
	\,
	[\bar e_2,\bar e_3]=\lambda_1\bar e_1,
	\,
	[\bar e_2,\bar e_4]=b  \bar e_2, 
	\,
	[\bar e_3,\bar e_4]=\sqrt{D^2+1}\, \bar e_2,
	\]
	and therefore  this case is included in Section~\ref{se: E11 E2 C=0}.

	\subsubsection{Case $D^2=1$,  $b\neq0$, $A\neq 0$} 
	We compute  a Gröbner basis of  the ideal  $\mathcal{I}_{3}=\langle \mathcal{I}\cup \{D^2-1\} \rangle$, consisting of $19$ polynomials,   and we  get that   
	\[
	\mathbf{g}_{3}=Ab^3(256A^4b^4+416 A^4b^2+288 A^2b^2+144 A^4+225A^2+81)
	\] 
	belongs to the basis.  Hence,  since $A b\neq 0$, there are no critical metrics in this case. Notice that all the homotheties mentioned above preserve the Lie algebra structure, therefore the classification is given modulo isomorphic homotheties.
\end{proof}

\section{Left-invariant $\mathcal{F}_t$-critical metrics on $\mathbb{R}\ltimes \mathcal{H}^3$} \label{se:H}

Let $\mathfrak{g}=\mathbb{R}\ltimes\mathfrak{h}^3$ be  a semi-direct extension of the Heisenberg algebra $\mathfrak{h}^3$. Let $\langle \cdot, \cdot \rangle$ be an inner product on $\mathfrak{g}$ and $\langle \cdot, \cdot \rangle_3$ its restriction to $\mathfrak{h}^3$. 
Then, there exists  an orthonormal basis $\{\mathbf{v}_1, \mathbf{v}_2, \mathbf{v}_3\}$ of $\mathfrak{h}^3$ such that (see \cite{Milnor})
\begin{equation}\label{eq:rvl 1}
[\mathbf{v}_3,\mathbf{v}_2]=0,
\quad 
[\mathbf{v}_3,\mathbf{v}_1]=0, 
\quad 
[\mathbf{v}_1,\mathbf{v}_2]= \lambda_3 \mathbf{v}_3,
\end{equation}
where $\lambda_3\neq 0$ is a real number. The algebra of all derivations of $\mathfrak{h}^3$ is given with respect to the basis $\{ \mathbf{v}_1, \mathbf{v}_2, \mathbf{v}_3\}$ by
\[ 
\operatorname{der}(\mathfrak{h}^3)=\left \{  \left( \begin{array}{ccc}
\alpha_{11} & \alpha_{12} & 0 
\\
\alpha_{21} & \alpha_{22} & 0 
\\
\hat h & \hat f &\alpha_{11}+\alpha_{22}
\end{array} \right) ;\, \alpha_{ij},  \hat f, \hat h \in \mathbb{R} \right \}.
\]
We rotate the basis elements $\{ \mathbf{v}_1,\mathbf{v}_2\}$ so that the matrix $A=(\alpha_{ij})$ decomposes as the sum of a diagonal matrix and a skew-symmetric matrix. 
Hence 
\[ 
\operatorname{der}(\mathfrak{h}^3)=\left \{  \left( \begin{array}{ccc}
\tilde a & \tilde c & 0 
\\
-\tilde c & \tilde d & 0 
\\
\tilde h  & \tilde f  & \tilde a+\tilde d \end{array} \right) ;\, \tilde a,  \tilde c, \tilde d, \tilde f, \tilde h \in \mathbb{R} \right \}.
\]
Let $\{\mathbf{v}_1, \mathbf{v}_2, \mathbf{v}_3, \mathbf{v}_4\}$ be a basis of  $\mathfrak{g}=\mathbb{R}\mathbf{v}_4\oplus \operatorname{span}\{\mathbf{v}_1, \mathbf{v}_2, \mathbf{v}_3\}$.
Since $\mathbb{R}\mathbf{v}_4$ is not necessarily orthogonal to $\mathfrak{h}^3$, we set $\tilde k_i=\langle \mathbf{v}_i, \mathbf{v}_4 \rangle$, for $i = 1,2,3$. Let $\bar e_4= \mathbf{v}_4-\sum_i \tilde k_i\mathbf{v}_i$ and normalize it to get an orthonormal basis $\{ e_1,\dots,e_4\}$ of $\mathfrak{g}= \mathbb{R} \oplus  \mathfrak{h}^3 $ where $e_i=\mathbf{v}_i$ for $i=1,2,3$. 
We simplify notation using $a=-\tilde a\, \|e_4\|^{-1}$, $c=-\tilde c\,\|e_4\|^{-1}$, $d=-\tilde d\,\|e_4\|^{-1}$, $f=-\tilde f\|e_4\|^{-1}$, $h=-\tilde h\|e_4\|^{-1}$, $k_i=-\tilde k_i\|e_4\|^{-1}$  for $i=1,2,3$,
	and we set $F=f-k_1\gamma$ and $H=h+k_2\gamma$. Thus the Lie brackets are written as follows
	\begin{equation}\label{eq:corH3}
	\begin{array}{ll}
	[e_1, e_2]= \gamma e_3, 
	&
	[e_1,e_4] =   a e_1- c e_2+ H e_3,
	
	\\\noalign{\medskip} 
	
	[e_3,e_4]=( a+d)e_3,
	&
	[e_2,e_4]= c e_1+ d e_2+F e_3.
	\end{array}
	\end{equation}

\begin{remark}\label{re:einstein-H}\rm
The scalar curvature of left-invariant metrics on $\mathbb{R}\ltimes \mathcal{H}^3$, which is given by $\tau=-\frac{1}{2} \left(4 \left(3 a^2+3 d^2+5 a d\right)+F^2+H^2+\gamma^2\right)$, is strictly negative. Moreover, these metrics are Einstein if and only if 
	$$
	[e_1,e_2]=\gamma e_3, 
	\quad 
	[e_1,e_4]=\pm\tfrac{1}{2}\gamma e_1-c e_2,
	\quad 
	[e_2,e_4]=c e_1\pm\tfrac{1}{2}\gamma e_2,
	\quad 
	[e_3,e_4]=\pm\gamma e_3,
	$$
	in which case they correspond to the complex hyperbolic space, which is the only symmetric left-invariant metric on  $\mathbb{R} \ltimes \mathcal{H}^3$.
\end{remark}

\begin{remark}\label{re:4-2}
	\rm
	Let $\langle\cdot,\cdot\rangle$ be a left-invariant metric on $\mathbb{R}\ltimes \mathcal{H}^3$ as in \eqref{eq:corH3} with  parameters $(\gamma,a,c,d,H,F)$ determining the structure constants. The isometry $(e_1, e_2, e_3, e_4)\mapsto(-e_2, e_1, e_3, e_4)$ transforms this set of parameters into $(\gamma,d,c,a,-F,H)$. Hence any left-invariant metric \eqref{eq:corH3} with  $F=0$ is isometric to a left-invariant metric with $H=0$.
\end{remark}

\begin{theorem}\label{th:heisenberg}
	A non-symmetric left-invariant metric on $\mathbb{R}\ltimes \mathcal{H}^3$ is critical for a quadratic curvature functional $\mathcal{F}_t$ if and only if one of the following holds: 
	\begin{itemize}
		\item[(a)] The functional has zero energy and the metric is homothetic to one of the following:
		\begin{enumerate}
			\item The direct product $\mathbb{R}\times \mathcal{H}^3$ with Lie bracket $[e_1,e_2]=e_3$.
			In this case, $t=-3$.
			
			\smallskip
			\item 
			$[e_1,e_2]=e_3$, 
			$[e_1,e_4]=a e_1$, 
			$[e_2,e_4]=d e_2$  and 
			$[e_3,e_4]=(a+d)e_3$, 
			\\
		with $a\in [-\frac{\sqrt{3}}{2},\frac{1}{2} )$. For a fixed $a$, the parameter $d$ is given by the only positive solution of $4(a^2+d^2+ad)-3=0$.
			Moreover  $t=-\frac{3}{2(4ad+5)}$.
			
			\smallskip
			\item 
			$[e_1,e_2]=e_3$ and 
			$[e_2,e_4]= - e_1$. 
			In this case, $t=-\frac{3}{2}$.
		\end{enumerate}
		
		\medskip
		\item[(b)] The energy of the functional is non-zero and the metric is homothetic to a left-invariant metric determined by one of the following:
		\begin{enumerate}
			\item[(4)]
			$[e_1,e_2]=e_3$, 
			$[e_1,e_4]=ae_1$, 
			$[e_2,e_4]=a e_2$, 
			$[e_3,e_4]=2ae_3$, 
			\smallskip
			\\
			with $a\in\left(0,\frac{1}{2}\right)\cup\left(\frac{1}{2},+\infty\right)$.
			Moreover, $t=-\frac{3(4a^2+1)}{44a^2+1}$ and $\energy_t=-36a^2$.

			\smallskip
			\item[(5)]
			$[e_1,e_2]=e_3$,
			$[e_1,e_4]=ae_1-ce_2$,
			$[e_2,e_4]=c e_1+d e_2$ and 
			$[e_3,e_4]=(a+d)e_3$,
with $c\in (0,\frac{1}{2} )\cup\left(\sqrt{\frac{7}{8}},+\infty\right)$. For a fixed $c$, the parameters $a$ and $d$ are the only negative and positive solutions, respectively, of $3 a^2+4c^2+3 d^2+10 a d=0$ and $32 c^4 + (32 a d + 8) c^2 - 28 a d - 9=0$ satisfying $|a|<d$. 
			In this case, $t=-\frac{3(32c^4-7)}{4(8c^4+40c^2-13)}$ and    $\energy_t=-\frac{72c^2(4c^2-1)}{8c^2-7}$.
			\smallskip
			\item[(6)]
			$[e_1,e_2]=e_3$, \,\,
			$[e_1,e_4]=ae_1+He_3$, \,\,
			$[e_2,e_4]=de_2$ \,\, and \,\,
			$[e_3,e_4]=(a+d)e_3$,
with $H\in\left(0,\sqrt{\frac{3}{7}}\right)$. For a fixed $H$, the parameter $d$ is the only positive solution of 
$$
4 d^6 + 5(3 H^2 - 1) d^4 + (18 H^4 + 6 H^2 + 13) d^2 + (H^2 +  1)^2 (7 H^2 - 3)=0
$$
and the parameter $a$ is the positive value given by $a=-\frac{ ( 2 d^2+2 H^2 - 3) d }{2 ( d^2+H^2 + 1) }$.
In this case, 
$t=-\frac{ 4 a^2+3 d^2 + 2 a d+3 H^2+3}{ 12 a^2 + 12 d^2 + 20 a d+H^2 + 1}$   
and the energy is given by 
$$
\begin{array}{l}
\energy_t =\tfrac{1}{4} \left\{24 a^3 d+2 ad(22 d^2\!-\! 19H^2\!-\!31)+3d^2(4d^2\!-\!7H^2\!-\!13)\right.\\
\phantom{\energy_t =\tfrac{1}{4} \left\{ 24 a^3 d+2 ad(22 d^2\!-\! 19H^2\!-\!31)\right. \!}
\left. +4a^2(13d^2\!-\!6H^2\!-\!10)	\right\}.
\end{array}
$$
\end{enumerate}
\end{itemize}
Moreover, metrics in Families~(1)--(3) are algebraic Ricci solitons, while metrics corresponding to Families (4)--(6) are not.	
\end{theorem}

\begin{remark}\rm\label{re:H3-grafico}
	The range of the parameter $t$ for each family in Theorem~\ref{th:heisenberg} is indicated in Figure~\ref{figura-H} below. The energy vanishes ($\energy_t=0$) in Families~(1)--(3) and is strictly negative ($\energy_t<0$)  in Families~(4)--(6).
	\begin{figure}[h]
	\begin{center}
		\begin{tikzpicture}
		
		\draw[thick,-latex]   (-4.9-0.15,0) -- (5.7,0) node[right] {$t$}; 
		\foreach \i/\n in 
		{
		{-4.5}/$-3$, 
		{-2.35}/$-\frac{3}{2}$, 
		{-0.95}/$-\frac{3}{4}$, 
		{0.05}/$-\frac{1}{2}$, 
		{0.75}/$-\frac{7}{15}$,
		{1.45}/$-\frac{7}{16}$, 
		{2.15}/$-\frac{21}{52}$,
		{2.95}/$-\frac{1}{3}$, 
		{3.75}/$-\frac{3}{11}$,
		{4.45}/$-\frac{1}{4}$
		}
		{\draw (\i,-.2)--(\i,.2);}

		\foreach \i/\n in 
		{
		{-4.5}/$-3$
		}
		{\draw  (\i,.2) node[above] {\footnotesize \n};}

		\foreach \i/\n in 
		{
		{-2.35}/$-\frac{3}{2}$, 
		{-0.95}/$-\frac{3}{4}$, 
		{0.05}/$-\frac{1}{2}$, 
		{0.75}/$-\frac{7}{15}$,
		{1.45}/$-\frac{7}{16}$, 
		{2.15}/$-\frac{21}{52}$,
		{2.95}/$-\frac{1}{3}$, 
		{3.75}/$-\frac{3}{11}$,
		{4.45}/$-\frac{1}{4}$
		}
		{\draw  (\i,.12) node[above] {\footnotesize \n};}
		
		\foreach \i/\n in 
		{
		{-4.5}/$-3$, 
		{-2.35}/$-\frac{3}{2}$, 
		{-0.95}/$-\frac{3}{4}$, 
		{0.05}/$-\frac{1}{2}$, 
		{0.75}/$-\frac{7}{15}$,
		{1.45}/$-\frac{7}{16}$, 
		{2.15}/$-\frac{21}{52}$,
		{3.75}/$-\frac{3}{11}$,
		{4.45}/$-\frac{1}{4}$
		}{
			\draw[dashed] (\i,-.25)--(\i,-3.2);
		}

		\draw (-5.45,-0.45) node {\footnotesize (1) };
		\filldraw [color=\ColorECeroARS,fill=\ColorECeroARS] (-4.5,-0.45) circle (2pt);

		\draw (-5.45,-0.95) node {\footnotesize (2) };
		\draw[color=\ColorECeroARS] (-0.95,-0.95)--(4.45,-0.95);
		\filldraw [color=\ColorECeroARS,fill=\ColorECeroARS] (-0.95,-0.95) circle (2pt);
		\filldraw [color=\ColorECeroARS,fill=white] (4.45,-0.95) circle (2pt);

		\draw (-5.45,-1.45) node {\footnotesize (3) };
		\filldraw [color=\ColorECeroARS,fill=\ColorECeroARS] (-2.35,-1.45) circle (2pt);

		\draw (-5.45,-1.95) node {\footnotesize (4) };
		\draw[color=\ColorENegativa] (-4.5,-1.95)--(3.75,-1.95);
		\filldraw [color=\ColorENegativa,fill=white] (-4.5,-1.95) circle (2pt);
		\filldraw [color=\ColorENegativa,fill=white] (0.05,-1.95) circle (2pt);
		\filldraw [color=\ColorENegativa,fill=white] (3.75,-1.95) circle (2pt);

		\draw (-5.45,-2.45) node {\footnotesize (5) };
		\draw[color=\ColorENegativa] (-4.5,-2.45-\YSepDosMetricas)--(2.15,-2.45-\YSepDosMetricas);
		\filldraw [color=\ColorENegativa,fill=white] (-4.5,-2.45-\YSepDosMetricas) circle (2pt);
		\filldraw [color=\ColorENegativa,fill=white] (2.15,-2.45-\YSepDosMetricas) circle (2pt); 
		
		\draw[color=\ColorENegativa] (-2.35,-2.45+0.7*\YSepDosMetricas)--(0.75,-2.45+0.7*\YSepDosMetricas);
		\filldraw [color=\ColorENegativa,fill=white] (-2.35,-2.45+0.7*\YSepDosMetricas) circle (2pt); 
		\filldraw [color=\ColorENegativa,fill=white] (0.75,-2.45+0.7*\YSepDosMetricas) circle (2pt);

		\draw (-5.45,-2.95) node {\footnotesize (6) };
		\draw[color=\ColorENegativa] (-4.5,-2.95)--(1.45,-2.95);
		\filldraw [color=\ColorENegativa,fill=white] (-4.5,-2.95) circle (2pt);
		\filldraw [color=\ColorENegativa,fill=white] (1.45,-2.95) circle (2pt); 
		\end{tikzpicture}
\caption{Range of the parameter $t$ for non-symmetric homogeneous $\mathcal{F}_t$-critical metrics on $\mathbb{R}\ltimes \mathcal{H}^3$.}
\label{figura-H}		
\end{center}
\end{figure}
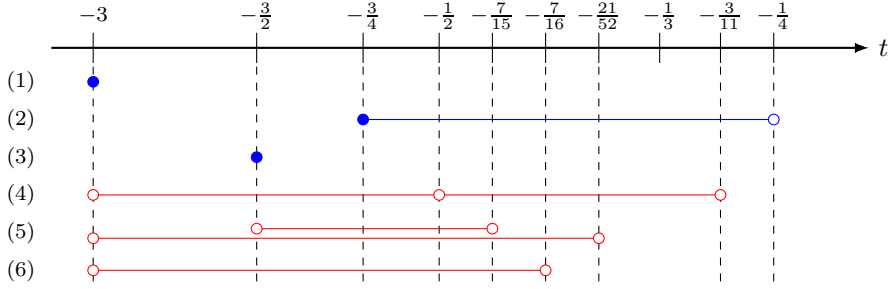
	
	\smallskip
	
 \noindent
For the special choice $a=1$ in Family~(4) the metric is self-dual and, thus, Bach-flat (equivalently $\mathcal{F}_{-1/3}$-critical)~\cite{Salamon}.
For the special choice of $a=-\tfrac{1}{4}\!\sqrt{7\!-\!3\sqrt{5}}$ in Family~(2) the metric is a Bach-flat Ricci soliton (see \cite{CL-GM-GR-GR-VL} and Remark~\ref{re:Bach-llanas}).

Proceeding as in Remark~\ref{re:E-grafico}, the set of homothetic invariants $\{t,\|\rho\|^2\, \tau^{-2}$, $\| R\|^2\, \tau^{-2}$, $\|\nabla\rho\|^2\,\tau^{-3}$, $\|\nabla R\|^2\tau^{-3}\}$ suffices to show that there are no homotheties between metrics in Theorem~\ref{th:heisenberg} and Theorem~\ref{th:E(1,1)}.
	Furthermore, in Theorem~\ref{th:heisenberg}, it also follows  that there are no homotheties between metrics in different families and, moreover,
	that different values of the parameters in any of the families give rise to metrics which are not homothetic.

Even more, for a fix value of $t\in[-3,-\frac{1}{4})$, there exists a unique homothetic class of $\mathcal{F}_t$-critical metrics in each of the six families given in Theorem~\ref{th:heisenberg}, with the only exception of Family~(5), where there are exactly two homothetic classes of $\mathcal{F}_t$-critical metrics if $t\in(-\frac{3}{2},-\frac{7}{15})$. 
\end{remark}

\begin{proof}
	A left-invariant metric is $\mathcal{F}_t$-critical if and only if the symmetric $(0,2)$-tensor field 
	$
	\mathfrak{F}_t=
	-\Delta \rho 
	+\frac{1}{2} (\|\rho\|^2 +t\tau^2) g 
	-2 R[\rho] 
	-2t\tau\rho
	$
	vanishes, where   the components $\mathfrak{F}_t(e_i,e_j)$ are polynomials on  the parameters determining the metric~\eqref{eq:corH3} and $t$.  In order to solve the system $\{ \mathfrak{F}_t(e_i,e_j)=0\}$, we analyze the different  possibilities depending on whether $F$ and $H$ vanish or not. Note that, by Remark~\ref{re:4-2}, the cases $F=0$ and $H=0$ are equivalent.
	
	\subsection{\boldmath Case $F=H=0$\unboldmath}
	Since $\gamma\neq 0$  we consider a representative in the homothetic class with $\gamma=1$. In this case, the non-vanishing components of $\mathfrak{F}_t$ are $\mathfrak{F}_t(e_i,e_i)$, $i=1,\dots,4$, and
	\[
	\mathfrak{F}_t(e_1,e_2)=
	c(d-a)  
	\left(  (12 a^2 + 12 d^2 + 20 a d + 1) t + a^2 + 4 c^2 + d^2 - 6 a d  \right).
	\]
	Next we analyze the vanishing of each one of the three factors in $\mathfrak{F}_t(e_1,e_2)$ separately.

	\subsubsection{Case $d=a$}
	In this case  the left-invariant metric is described by 
		$$
		[e_1,e_2]=e_3,
		\quad
		[e_1,e_4]=a e_1 - c e_2,
		\quad
		[e_2,e_4]=c e_1+a e_2,
		\quad
		[e_3,e_4]=2a e_3.
		$$
		The sectional curvature is independent of the parameter $c$ and, hence, the metric is homothetic to a metric with $c=0$ (see \cite{Kulkarni}). Moreover, the isometry $e_4\mapsto-e_4$ shows that the parameter $a$ may be replaced by $-a$. Hence, we assume $c=0$ and $a\geq 0$ to compute the non-vanishing terms of $\mathfrak{F}_t$:
	\[
	\begin{array}{l}
	\mathfrak{F}_t(e_1,e_1) = \mathfrak{F}_t(e_2,e_2)=
	-\tfrac{3}{5}\mathfrak{F}_t(e_3,e_3) =
	-3\mathfrak{F}_t(e_4,e_4)
	
	\\
	\noalign{\medskip}
	\phantom{\mathfrak{F}_t(e_1,e_1) = 
		\mathfrak{F}_t(e_2,e_2)}
	
	=\tfrac{3}{8}   (2 a - 1) (2 a + 1) \left( (44 a^2 + 1) t + 3 (4 a^2 + 1) \right).
	\end{array}
	\]
	Note that if $a=1/2$ then the metric is Einstein (see Remark~\ref{re:einstein-H}). Hence, it follows that  $t=-\frac{3 (4 a^2 + 1)}{44 a^2 + 1}$ and the energy is given by $\energy_t = -36 a^2$.  These metrics correspond to those in Family~(1) if $a=0$, with $t=-3$, and to those  in Family~(4) if $a> 0$ and $a\neq 1/2$, where $t=-\frac{3(4a^2+1)}{44a^2+1}\in(-3, -\frac{3}{11})\setminus\{-\frac{1}{2}\}$. Metrics in Family~(1) are algebraic Ricci solitons, since  $\operatorname{Ric}+\frac{3}{2}\operatorname{Id}$ is a derivation.
	
	
	\subsubsection{Case $c=0$, $d\neq a$}
 The  left-invariant metric is given by 
		$$
		[e_1,e_2]=e_3,
		\quad 
		[e_1,e_4]=a e_1,
		\quad
		[e_2,e_4]=de_2,
		\quad
		[e_3,e_4]=(a+d)e_3.
		$$
		The isometry $e_4\mapsto -e_4$ transforms the parameters $(a,d)$ into $(-a,-d)$, and the isometry $(e_1,e_2,e_3,e_4)\mapsto (e_2,-e_1,e_3,e_4)$ transforms $(a,d)$ into $(d,a)$. 
		Hence one may assume $|a|<d$ or $a=-d$.
		We  compute
	\[
	\begin{array}{l}
	\mathfrak{F}_t(e_3,e_3)=
	-\frac{1}{8} (4 (a^2 + d^2 + 3 a d) - 5) (12 a^2 + 12 d^2 + 20 a d + 1)  t 
	
	\\
	\noalign{\medskip}
	\phantom{\mathfrak{F}_t(e_3,e_3)=}
	
	- 2 (a^2 + d^2 + a d) (a^2 + d^2 + 3 a d) + \frac{15}{8} ,
	
	\\
	\noalign{\medskip}
	
	\mathfrak{F}_t(e_4,e_4)=
	-\frac{1}{8}(4 (a^2 + d^2 - a d) - 1) (12 a^2 + 12 d^2 + 20 a d + 1)  t 
	
	\\
	\noalign{\medskip}
	\phantom{\mathfrak{F}_t(e_4,e_4)=}
	
	- 2 (a^4 + d^4 + a^2 d^2) +  \frac{3}{8},
	\end{array}
	\]
	from where it follows that
	\[
	\mathfrak{F}_t(e_3,e_3) - 5\mathfrak{F}_t(e_4,e_4)=
	2(a-d)^2
	\left( (12 a^2 + 12 d^2 + 20 a d + 1) t 
	+ 4 (a^2 + d^2 + a d) \right).
	\]
	Since $d\neq a$, we have $t=-\frac{4 (a^2 + d^2 + a d)}{12 a^2 + 12 d^2 + 20 a d + 1}$. For this value of $t$ we get
	\[
	\begin{array}{l}
	\mathfrak{F}_t(e_1,e_1) = \mathfrak{F}_t(e_2,e_2)=
	-\tfrac{3}{5}\mathfrak{F}_t(e_3,e_3) =
	-3\mathfrak{F}_t(e_4,e_4)
	=\tfrac{3}{8}   \left(4 (a^2 + d^2 + a d) - 3
	\right) .
	\end{array}
	\]
	Therefore,  $4 (a^2 + d^2 + a d) - 3=0$. Since $|a|<d$ or $a=-d$, we have $a\in  [-\frac{\sqrt{3}}{2},\frac12 )$. The value of $t$ is $t=-\frac{3}{2(4ad+5)}\in[-\frac{3}{4}, -\frac{1}{4})$ and a direct computation shows that the energy vanishes.
	Moreover, $\operatorname{Ric}+\frac{3}{2}\operatorname{Id}$ is a derivation determining an algebraic Ricci soliton. This family of critical metrics corresponds to Family~(2).

	\subsubsection{Case $t=-\frac{a^2 + 4 c^2 + d^2 - 6 a d}{12 a^2 + 12 d^2 + 20 a d + 1}$, $c\neq 0$, $d\neq a$}\label{eq:H3-(3)}
	For this value of the parameter $t$ we get
	\[
	\begin{array}{l}
	\mathfrak{F}_t(e_3,e_3)=
	\phantom{-}
	
	\frac{1}{8}\left(
	4 c^2 (6 a^2 + 6 d^2 + 8 a d - 5)
	- 12 (a^4 + d^4)
	\right.
	
	\\
	\noalign{\medskip}
	\phantom{\mathfrak{F}_t(e_3,e_3)=-\frac{1}{8}(}
	
	\left.
	- 2 a d \left(38  a^2 + 38 d^2 + 72 a d - 15 \right)
	- 5 a^2 - 5 d^2 + 15 
	\right) ,
	\\
	\noalign{\medskip}
	
	\mathfrak{F}_t(e_4,e_4)=
	- \frac{1}{8}\left(
	4 c^2 (2 a^2 + 2 d^2 - 8  a d + 1)
	+12 (a^4 + d^4)
	\right.
	
	\\
	\noalign{\medskip}
	\phantom{\mathfrak{F}_t(e_4,e_4)=-\frac{1}{8}(}
	
	\left.
	+2 a d (14  a^2 + 14   d^2 - 8 a d - 3 )
	+a^2 + d^2 - 3 
	\right),
	\end{array}
	\]
	and a direct calculation shows that
	\[
	\mathfrak{F}_t(e_3,e_3)-5 \mathfrak{F}_t(e_4,e_4)
	=
	2(a-d)^2 (3 a^2 + 4 c^2 + 3 d^2 + 10 a d).
	\]
	Since $d\neq a$, it follows that $3 a^2 + 4 c^2 + 3 d^2 + 10 a d=0$ and this condition leads to
	\[
	\begin{array}{l}
	\mathfrak{F}_t(e_1,e_1) = \mathfrak{F}_t(e_2,e_2)=
	-\tfrac{3}{5}\mathfrak{F}_t(e_3,e_3) =
	-3\mathfrak{F}_t(e_4,e_4)
	
	\\
	\noalign{\medskip}
	\phantom{\mathfrak{F}_t(e_1,e_1) = \mathfrak{F}_t(e_2,e_2)}
	
	=
	\frac{1}{8}
	\left(32 c^4 + (32 a d + 8) c^2 - 28 a d - 9
	\right) .
	\end{array}
	\]
	Hence, if 
	\begin{equation}\label{eq: H3-1}
	3 a^2 + 4 c^2 + 3 d^2 + 10 a d=0
	\quad\text{and}\quad
	32 c^4 + (32 a d + 8) c^2 - 28 a d - 9=0,
	\end{equation}
	then the left-invariant metric given by
	\[
	[e_1,e_2]=e_3,
	\quad
	[e_1,e_4]=ae_1-ce_2,
	\quad
	[e_2,e_4]=c e_1+d e_2,
	\quad
	[e_3,e_4]=(a+d)e_3,
	\]
	is critical for $t=-\frac{4(7a d-2c^2)}{48c^2+60ad-3}$
	and the energy reads $\energy_t = 28 ad-8c^2+9$.
	
Now, $(e_1,e_2,e_3,e_4)\mapsto (e_1,-e_2,-e_3,e_4)$ defines an isometry which transforms the parameter $c$ into $-c$, so we may assume $c>0$. Moreover, the isometries $(e_1,e_2,e_3,e_4)\mapsto (-e_1,e_2,-e_3,-e_4)$ and  $(e_1,e_2,e_3,e_4)\mapsto (e_2,e_1,-e_3,-e_4)$ transform $(a,d)$ into $(-a,-d)$ and $(a,d)$ into $(-d,-a)$, respectively. Hence, as in the previous case,  we may further restrict the parameters  to satisfy $|a|<d$ or $d=-a$.
		
If $|a|<d$, then $t=-\frac{4(7a d-2c^2)}{48c^2+60ad-3}\in(-3,-\frac{21}{52})$ and the energy does not vanish, i.e., $28 ad-8c^2+9\neq0$. These metrics  correspond to  Family~(5). 
	
If $d=-a$, then using~(\ref{eq: H3-1}) and taking into account that $c\neq 0$ we get that $a^2=c^2=1/4$. Therefore the expression of the non-zero Lie brackets reduces to
		$$
		[e_1,e_2]=e_3,
		\quad
		[e_1,e_4]=\pm\tfrac{1}{2}\left(e_1- \varepsilon e_2\right),
		\quad
		[e_2,e_4]=\pm\tfrac{1}{2} \left(\varepsilon e_1- e_2\right), 
		$$
		where $\varepsilon^2=1$.  Now, 
		considering the orthonormal basis given by		
		\[
		\bar e_1=\mp\tfrac{1}{\sqrt{2}}\left(\varepsilon e_1-e_2\right),
		\quad
		\bar e_2=\tfrac{-1}{\sqrt{2}}\left( e_1+\varepsilon e_2\right),
		\quad
		\bar e_3=\pm e_3,
		\quad
		\bar e_4= -\varepsilon e_4,
		\]
		the only non-zero brackets are $[\bar e_1,\bar e_2]=\bar e_3$ and $[\bar e_2,\bar e_4]=-\bar e_1$, which show that the metrics above are  isometric to those in Family~(3). The value of $t$ for which these metrics are critical is $t=-3/2$ and the energy vanishes. Moreover, $\operatorname{Ric}+\frac{3}{2}\operatorname{Id}$ is a derivation determining an algebraic Ricci soliton.

	\subsection{\boldmath  Case $F=0$, $H\neq 0$\unboldmath}
	As in the first case we rescale to assume $\gamma=1$ and compute
	\[
	\begin{array}{l}
	\mathfrak{F}_t(e_1,e_4)
	=
	\tfrac{1}{2} c(a-5d) H,
	
	\\
	\noalign{\medskip}
	
	\mathfrak{F}_t(e_2,e_4)
	=
	-\tfrac{1}{2}H \left(
	( 12 a^2 + 12 d^2 + 20 a d +H^2 + 1) t 
	\right.
	\\
	\noalign{\medskip}
	\phantom{\mathfrak{F}_t(e_2,e_4)=-\tfrac{1}{2}H(}
	\left.
	+  4 a^2 + c^2 + 3 d^2 + 2 a d + 3 H^2+ 3
	\right) .
	\end{array}
	\]
	Since $H\neq 0$, from $\mathfrak{F}_t(e_2,e_4)=0$ it follows that 
	$t=-\frac{   4 a^2 + c^2 + 3 d^2 + 2 a d+3 H^2 + 3}
	{ 12 a^2 + 12 d^2 + 20 a d+H^2 + 1}$, so  we fix the value of $t$ and set  $\mathfrak{F}=\mathfrak{F}_t$ for that particular value.  Moreover, from $\mathfrak{F}_t(e_1,e_4)=0$ we have that either $c=0$ or $c\neq 0$ and $a=5d$. We analyze the two cases separately.

	\subsubsection{Case $c=0$}
		If $c=0$ the left-invariant metric is given by the Lie brackets 
		\begin{equation}\label{eq:Heisenberg-F=0-c=0}
		[e_1,e_2]=e_3,
		\quad
		[e_1,e_4]=ae_1+H e_3,
		\quad
		[e_2,e_4]=d e_2,
		\quad
		[e_3,e_4]=(a+d)e_3.
		\end{equation}
		The changes of basis given by $(e_1,e_2,e_3,e_4)\mapsto (-e_1,-e_2,e_3,e_4)$ and $(e_1,e_2,e_3,e_4)\mapsto (-e_1,-e_2,e_3,-e_4)$ are isometries  which transform the parameters $(a,d,H)$ into $(a,d,-H)$ and $(-a,-d,H)$, respectively. Hence we may assume $H>0$ and $a\leq d$.
		
		We compute
	\[
	\begin{array}{l}
	\mathfrak{F}(e_1,e_1)
	=
	\phantom{-}
	
	\tfrac{1}{8}\left(
	12 d^4 - 8 a^3 d  + 28 a d^3 + 
	4 a^2 d^2
	\right.
	
	\\
	\noalign{\medskip}
	\phantom{\mathfrak{F}(e_1,e_1)=-\tfrac{1}{8}(}
	\left.
	
	+ 24 a^2 - 27 d^2 - 6 a d
	+(24 a^2 + 3 d^2 + 34 a d ) H^2
	\right) ,
	
	\\
	\noalign{\medskip}
	
	\mathfrak{F}(e_3,e_3)
	=-\tfrac{1}{8}\left(
	4 d^4 + 8 a^3 d + 20 a d^3 + 28 a^2 d^2 
	\right.
	
	\\
	\noalign{\medskip}
	\phantom{\mathfrak{F}(e_3,e_3)	=-\tfrac{1}{8}(}
	\left.
	
	+ 8 a^2 + 3 d^2 - 26 a d
	+ (8 a^2 + 13 d^2 + 14 a d ) H^2 
	\right),

	\end{array}
	\]
	from where it follows that
	\[
	\mathfrak{F}(e_1,e_1)-\mathfrak{F}(e_3,e_3)
	= (2a+d) 
	\left( 2 (d^2+H^2  + 1) a + ( 2 d^2+2 H^2 - 3) d  \right).
	\]
	Hence, we have $d=-2a$ or $a=-\frac{ ( 2 d^2+2 H^2 - 3) d }{2 ( d^2+H^2 + 1) }$.
	
	\medskip
	
	If $d=-2a$ then $\mathfrak{F}$ is determined by
	\[
	\mathfrak{F}(e_1,e_1) = -\mathfrak{F}(e_2,e_2)=
	\mathfrak{F}(e_3,e_3) =
	-\mathfrak{F}(e_4,e_4) 
	=
	-a^2 (4H^2+9),
	\]
	so $a=d=0$ and the left-invariant metric is given by
	\[
	[e_1,e_2]=e_3,
	\quad
	[e_1,e_4]=H e_3.
	\]
	Considering the orthogonal basis given by
	$$
	\bar e_1=\tfrac{1}{H^2+1}\left(e_2+H e_4\right), \quad
	\bar e_2=\tfrac{1}{\sqrt{H^2+1}} e_1,
	\quad
	\bar e_3=\tfrac{-1}{\sqrt{H^2+1}} e_3,
	\quad
	\bar e_4= \tfrac{1}{H^2+1}(He_2-e_4),
	$$
	the only non-zero bracket now becomes $[\bar{e}_1,\bar{e}_2]=\bar{e}_3$.  Moreover, the inner product satisfies $\langle\bar{e}_i,\bar{e}_j\rangle=\frac{1}{H^2+1}\langle e_i,e_j\rangle$, so these metric Lie groups are isomorphically homothetic to the product metric in Family~(1).

	\medskip
	
	If $a=-\frac{ ( 2 d^2+2 H^2 - 3) d }{2 ( d^2+H^2 + 1) }$ and $d\neq-2a$, then $\mathfrak{F}$ is determined by
	\[
	\begin{array}{l}
	\mathfrak{F}(e_1,e_1) = \mathfrak{F}(e_3,e_3)=
	-\frac{(d^2+H^2+6)d^2}{8(d^2+H^2+1)^3}
	\, \mathbf{q} ,

	\\
	\noalign{\medskip}
	
	\mathfrak{F}(e_2,e_2) =
	-\frac{(d^2+H^2-14)d^2}{8(d^2+H^2+1)^3}
	\, {\mathbf{q}},
	
	\\
	\noalign{\medskip}
	
	\mathfrak{F}(e_4,e_4) =
	\frac{(3d^2+3H^2-2)d^2}{8(d^2+H^2+1)^3}
	\, {\mathbf{q}}	  ,
	\end{array}
	\]
	where ${\mathbf{q}} = 4 d^6 + 5(3 H^2 - 1) d^4 + (18 H^4 + 6 H^2 + 13) d^2 + (H^2 +  1)^2 (7 H^2 - 3)$. Note that $d\neq 0$ since $d\neq -2a$. Hence, from $\mathfrak{F}(e_1,e_1)=0$ we have that, necessarily, ${\mathbf{q}}=0$. Thus a straightforward calculation shows that $a\neq d$, since otherwise $H=0$.  Now, subject to this constraint,    $t=-\frac{   4 a^2  + 3 d^2 + 2 a d+3 H^2 + 3}
	{ 12 a^2 + 12 d^2 + 20 a d+H^2 + 1}\in(-3,-\frac{7}{16})$ 
	and the energy is given by 
	\small
	$$
	\energy_t =\tfrac{1}{4}\left\{ 24 a^3 d+2 ad(22 d^2\!-\! 19H^2\!-\!31)+3d^2(4d^2\!-\!7H^2\!-\!13)+4a^2(13d^2\!-\!6H^2\!-\!10)	
	\right\}.	
	$$\normalsize
	Since ${\mathbf{q}}=0$ and $d\neq 0$, a straightforward calculation shows that the energy never vanishes. Indeed, $\energy_t\in (-\frac{27}{4}, 0)$.
	These metrics correspond to Family~(6).
	
	\subsubsection{Case $a=5d$, $c\neq 0$}
	We calculate 
	\[
	\begin{array}{l}
	\mathfrak{F}(e_1,e_1)=
	-\frac{1}{8}\left(
	c^2 ( 1340 d^2+11 H^2 - 3) + d^2 (  748 d^2-773 H^2 -543)
	\right) ,
	
	\\
	\noalign{\medskip}
	
	\mathfrak{F}(e_2,e_2)=
	\frac{1}{8}\left(
	c^2 ( 1348 d^2+13 H^2 + 3) + d^2 ( 3476 d^2-619 H^2 - 609) 
	\right),
	\end{array}
	\]
	which imply
	\[
	\begin{array}{l}
	( 1348 d^2+13 H^2 + 3)\mathfrak{F}(e_1,e_1)
	+ ( 1340 d^2 +11 H^2- 3)\mathfrak{F}(e_2,e_2)
	
	\\
	\noalign{\medskip}
	\phantom{(13 H^2 + 13}
	
	= 
	27 d^2 \left(
	16896 d^4 + 4  (279 H^2 - 112) d^2 + (15 H^2 + 21 ) H^2 + 16
	\right).
	\end{array}
	\]
	This expression vanishes only if $d=0$ and, in such a case, we obtain  $\mathfrak{F}(e_2,e_2)=\frac{1}{8}c^2(13H^2+3)\neq 0$. Hence, there are no critical metrics with $a=5d$ and $c\neq 0$.

	\subsection{\boldmath  Case $FH\neq 0$\unboldmath}
	 We consider Lie brackets given by \eqref{eq:corH3} with $FH\neq 0$. From the expression of $\mathfrak{F}_t(e_2,e_4)$, given by
	\[
	\begin{array}{l}
	\mathfrak{F}_t(e_2,e_4)
	=
	\frac{-1}{2} 
	\gamma \left(   (12 a^2 + 12 d^2 + 20 a d + F^2 + 
	H^2 + \gamma^2) H t  
	\right.
	
	\\
	\noalign{\medskip}
	\phantom{\mathfrak{F}_t(e_2,e_4)=-\frac{1}{2} 
		\gamma(}
	
	\left.
	\!\!
	+(5 a - d) c  F + (4 a^2 + c^2 + 
	3 d^2 + 2 a d ) H + 3 ( F^2 + H^2 +  \gamma^2) H  \right) \!,
	\end{array}
	\]
	we obtain that
	\[
	t=
	-\tfrac{(5 a - d) c  F + (4 a^2 + c^2 + 
		3 d^2 + 2 a d ) H + 3 ( F^2 + H^2 +  \gamma^2) H }
	{\left(12 a^2 + 12 d^2 + 20 a d + F^2 + H^2 + \gamma^2\right) H} .
	\]
	 We fix this value of $t$ and observe that each non-null component $\mathfrak{F}_t(e_i,e_j)$  is multiplied by $1/H$. Set $\mathfrak{F}=H\,\mathfrak{F}_t$ to remove this factor.
	 Now we use Gröbner bases to show that there do not exist other critical metrics than those already found. Throughout the remaining of this section we fix the  order to be lexicographic and distinguish two cases: $d=0$ and $d\neq 0$.

	\subsubsection{Case $d=0$}
	We rescale to take $\gamma=1$ and  consider the polynomial ring $\mathbb{R}[a,c,d,F,H,\gamma]$. Computing a Gröbner basis for the ideal 
	$\mathcal{I}_1=\langle \{\mathfrak{F}(e_i,e_j)\}\cup \{d,\gamma-1\} \rangle$ we get $20$ polynomials which include the following:
	\small\[
	\begin{array}{l}
	\mathbf{g}_{1}^1 = 
	c^2 H^2 (20 H^4 + 29 H^2 + 50) (405 H^4 + 601 H^2 + 150) (4232 H^4 +    5888 H^2 + 4575) , 
	
	\\
	\noalign{\medskip}
	
	\mathbf{g}_{1}^2 = a (5 c F^2 + a F H - c H^2).
	\end{array}
	\]\normalsize
	 From $\mathbf{g}_{1}^1=0$, since $H\neq 0$, we see that $c=0$. Now, from $\mathbf{g}_{1}^2=0$, as $FH\neq 0$, we see that $a=0$ too. Hence we have a critical left-invariant metric given by
	\[
	[e_1,e_2]=e_3,\quad
	[e_1,e_4]=H e_3,\quad
	[e_2,e_4]=F e_3.
	\]
	Considering the orthogonal basis given by
	$$
	\begin{array}{l}
	\bar e_1=\frac{1}{(F^2+H^2+1)\sqrt{F^2+H^2}}\left(F e_1-H e_2- (F^2+H^2) e_4\right), \qquad
	\bar e_3=\frac{-1}{\sqrt{F^2+H^2+1}} e_3,\\
	\noalign{\smallskip}
	\bar e_2=\frac{-1}{\sqrt{(F^2+H^2+1)(F^2+H^2)}}\left(H e_1+F e_2\right),\phantom{......}
	\bar e_4= \frac{-1}{F^2+H^2+1}\left( Fe_1-He_2+e_4\right),	
	\end{array}
	$$
	the only non-zero bracket is $[\bar{e}_1,\bar{e}_2]=\bar{e}_3$. Moreover $\langle\bar{e}_i,\bar{e}_j\rangle=\frac{1}{F^2+H^2+1}\langle e_i,e_j\rangle$, which shows that the above metric Lie groups are isomorphically homothetic to the product metric in Family~(1).

	\subsubsection{Case $d\neq 0$}
 We consider a representative in the homothetic class with $d=1$. 
	Moreover, we introduce an auxiliary variable $\tilde\gamma$ to express that the structure constant $\gamma$ is non-zero using the polynomial $\gamma\tilde\gamma-1$. In the polynomial ring $\mathbb{R}[\gamma,\tilde\gamma,F,H,a,c,d]$ we   consider the ideal $\mathcal{I}_2$ generated by $\{\mathfrak{F}(e_i,e_j)\}\cup\{d-1,\gamma\tilde\gamma-1\}$. Computing a Gröbner basis we get $53$ polynomials, among which we have 
	\[
	\mathbf{g}_2 =
	c H (a - 1)^2  (a + c^2)   \mathbf{p}_2
	(2 a^2 + c^2 + 5 a + 
	2) (3 a^2 + 4 c^2 + 10 a + 3)  ,
	\] 
	where $\mathbf{p}_2 =  24 a^4 + 9 c^4 - 50 a^2 c^2 - 50 a^3 + 
	134 a c^2 + 77 a^2 - 50 c^2 - 50 a + 24$. Recall that $H\neq 0$. We analyze the vanishing of each one of the other factors in $\mathbf{g}_2$ separately.

	\smallskip

	\paragraph{\underline{\emph{Case $a=1$}}}
	We  compute
	\[ 
	\mathfrak{F}(e_1,e_2) = -\tfrac{1}{2} H \left(6 c^2 F H + c (9  F^2 - 13  H^2) - 20 F H\right),
	\;\;
	\mathfrak{F}(e_1,e_3) = -2 c^2(F^2+H^2).
	\]
	Since $FH\neq0$, from $\mathfrak{F}(e_1,e_3)=0$ we obtain $c=0$ and hence $\mathfrak{F}(e_1,e_2)\neq 0$. So there are no critical metrics in this case.
	
	\smallskip
	
	\paragraph{\underline{\emph{Case $c=0$}}}
	As in the previous case,  there is no solution  if $c=0$, because the expressions
	\[ 
	\mathfrak{F}(e_1,e_2) = \tfrac{1}{2} (a + 3) (4 a + 1) F H^2,
	\quad
	\mathfrak{F}(e_1,e_4) = -\tfrac{1}{2} (a^2 - 1) F H \gamma ,
	\]
	give rise to a not compatible system of equations for $F H \gamma\neq 0$.

	\smallskip

	\paragraph{\underline{\emph{Case $a+c^2=0$, $c\neq 0$}}}
	We consider the ideal  $\mathcal{I}_{23}=\langle\mathcal{I}_2\cup\{a+c^2\}\rangle$ in the same polynomial ring $\mathbb{R}[\gamma,\tilde\gamma,F,H,a,c,d]$ and compute a Gröbner basis. This basis consists of $22$ polynomials and includes the following:
	\[
	\begin{array}{l}
	\mathbf{g}_{23}^1 = c H (c - 1)^2  (c + 1)^2 (c^2 + 1)^2  (c^4 + c^2 + H^2) ,
	
	\\
	\noalign{\medskip}
	
	\mathbf{g}_{23}^2 =  (c F+H)\left((5 c^2 + 1) F - c (c^2 + 5) H \right) .
	\end{array}
	\]
	Since $cH\neq0$, it follows that $c=\varepsilon_1$ and $H =\varepsilon_2 F$, where $\varepsilon_1^2=\varepsilon_2^2=1$. Now, the  component  $\mathfrak{F}(e_1,e_1)$ is given by
	\[
	\mathfrak{F}(e_1,e_1) = -\tfrac{3}{4} F  
	\left(2 (\varepsilon_1 + \varepsilon_2) (F^2 - 
	2) + (3 \varepsilon_1 - \varepsilon_2) \gamma^2 \right),
	\] 
	which implies  $\varepsilon_1=\varepsilon_2$ and  $2F^2+\gamma^2-4=0$. Hence, we have the left-invariant metric determined by
	\[
	[e_1,e_2] = \gamma e_3,\quad
	[e_1,e_4] = -e_1-\varepsilon e_2 + \varepsilon F e_3,\quad
	[e_2,e_4] = \varepsilon e_1+ e_2+F e_3,
	\]
	where $\varepsilon^2=1$ and $2F^2+\gamma^2-4=0$. In the orthogonal basis given by
	\[
	\begin{array}{ll}
	\bar e_1=\frac{-1}{2\sqrt{2}}\left(\varepsilon_1 e_1+e_2\right),&
	\bar e_3=\frac{-1}{2} e_3,
	\\
	\noalign{\smallskip}
	\bar e_2=\frac{-1}{4\sqrt{2}}\left( \gamma e_1-\varepsilon_1 \gamma e_2-2F e_4\right),&
	\bar e_4= \frac{\varepsilon_1}{4}\left( F e_1-\varepsilon_1 F e_2+\gamma e_4\right),
	\end{array}
	\]
 the non-zero Lie brackets reduce to $[\bar e_1,\bar e_2]=\bar e_3$ and $[\bar e_2,\bar e_4]=-\bar e_1$. Moreover, $\langle\bar{e}_i,\bar{e}_j\rangle=\frac{1}{4}\langle e_i,e_j\rangle$, so the
	metrics above are isomorphically homothetic to the critical metric in Family~(3) (see also~$\S$\ref{eq:H3-(3)}).

	\smallskip
	
	\paragraph{\underline{\emph{Case $\mathbf{p}_2=0$, $a+c^2\neq0$}}}
	In this case we start considering the polynomial ring $\mathbb{R}[a,\gamma,\tilde\gamma,c,F,H,d]$ and computing  a Gröbner basis for the ideal $\mathcal{I}_{24}=\langle\mathcal{I}_2\cup \{\mathbf{p}_2\}\rangle$. Thus, we get $44$ polynomials, being  one of them
	\[
	\begin{array}{l}
	\mathbf{g}_{24} = 
	F  H^2 (F - H)^2  (F + H)^2 (F^2 + H^2) \mathbf{p}_{24}
	(64 H^4 - 182 H^2 + 133) 
	
	\\
	\noalign{\medskip}
	\phantom{\mathbf{g}_{4} = }
	
	\times (1045897298699 H^4 - 528162809548 H^2 + 226891146700) ,
	\end{array}
	\]
	where 
	\small\[
	\!\!\!\begin{array}{l}
	\mathbf{p}_{24} =
	22500 F^{10} + 16 H^{10} + 136450 F^8 + 2047 H^8 + 262506 F^6 + 
	65418 H^6 + 171264 F^4 
	
	\\
	\noalign{\medskip}
	\phantom{\mathbf{g}_4^2}
	
	- 2816 H^4 + 40960 F^2 + 40960 H^2
	- \left(11891 F^4 H^2 - 3585 F^2 H^4 - 32697 F^2 H^2 
	\right.
	
	\\
	\noalign{\medskip}
	\phantom{\mathbf{g}_4^2}
	
	\left.
	- 6600 F^6 + 
	408 H^6 - 308704 F^4 + 23246 H^4 - 475614 F^2 - 186366 H^2 - 
	188928\right) \! F^2 H^2.
	\end{array}
	\]\normalsize
	The  analysis of the   factor   $\mathbf{p}_{24} $ shows that it is non-negative and, moreover,  it vanishes if and only if $F=H=0$. Furthermore, the last two factors in $\mathbf{g}_{24}$ do not have real roots. Hence, since $F H\neq 0$,  $\mathbf{g}_{24}=0$ implies $F^2-H^2=0$. Now, the computation of a second Gröbner basis for the ideal $\mathcal{\widetilde I}_{24}$ spanned by $\mathcal{I}_2\cup\{\mathbf{p}_2, F^2-H^2\}$ leads to $24$ polynomials, among which we choose
	\[
	\mathbf{\tilde g}_{24} = (c F - H) H^2 (H^2 + 20) (10201 H^4 + 3904 H^2 + 1024).
	\]
	Thus, $cF-H=0$ which, together with $F^2-H^2=0$, gives $c=\varepsilon$ and $H=\varepsilon F$, where $\varepsilon^2=1$. For $d=1$ and these relations among $F$, $H$ and $c$ we have that
	\[
	\mathfrak{F}(e_1,e_3)= -\tfrac{1}{2} (a+1)(a+3) F^2.
	\]
	Consequently, either $a=-1$ or $a=-3$. Finally, note that   $a=-1$ is not possible since we are assuming    $a+c^2\neq 0$. Also, if $a=-3$, then $\mathbf{p}_2=3268\neq 0$. Hence, we conclude that there are no critical metrics in this case.

	\smallskip
	
	\paragraph{\underline{\emph{Case $2 a^2 + c^2 + 5 a + 2=0$, $a+c^2\neq0$, $c\neq 0$}}}
	We consider the ideal $\mathcal{I}_{25}$ generated by $\mathcal{I}_2\cup\{2 a^2 + c^2 + 5 a + 2\}$ in the polynomial ring $\mathbb{R}[\tilde\gamma,F,H,\gamma,c,a,d]$ and  compute a Gröbner basis, obtaining $34$ polynomials. We see that  
	\[
	\mathbf{g}_{25} =  c H^2 (a + 1)^2 \mathbf{p}_{25}
	\]
	belongs to the basis, where 
	\[
	\mathbf{p}_{25} = 
	4 (a^2 + 5 a + 1) H^2 + (a - 1) 
	\left((2 - 7 a) \gamma^2 + 18 a^3 + 44 a^2 + 10 a - 4\right) .
	\]
	Note that $cH\neq 0$. Moreover, $a=-1$ is not possible since $2 a^2 + c^2 + 5 a + 2=0$ and  $a+c^2\neq0$.
	Hence, we have $\mathbf{p}_{25}=0$. Now we compute a second Gröbner basis, in this case for the ideal $\mathcal{\widetilde I}_{25} = \langle\mathcal{I}_{25}\cup\{\mathbf{p}_{25}\}\rangle$. We obtain $20$ polynomials, among which we have
	\[
	\mathbf{\tilde g}_{25} = H 
	\left(164 F^2 - ( 60 a+364) H^2 + (a - 1) (105  \gamma^2 - 270 a^2 - 784 a - 374)\right).
	\]
	Next we show that there are no critical metrics in this case. In order to do this, it will be  crucial the restriction $a\in(-2,-\tfrac{1}{2})\setminus\{-1\}$ derived from  the conditions $2 a^2 + c^2 + 5 a + 2=0$, $a+c^2\neq0$ and  $c\neq 0$. 
	Firstly,   from $\mathbf{p}_{25} = 0$ we get
	\[
	H^2  = - \tfrac{(a - 1) \left((2 - 7 a) \gamma^2 + 18 a^3 + 44 a^2 + 10 a - 4\right)}
	{4 (a^2 + 5 a + 1)},
	\]
	which is well-defined because $a\in(-2,-\tfrac{1}{2})\setminus\{-1\}$.  Moreover, $H^2>0$  implies  
	\begin{equation}\label{eq: H3 ineq 5-1}
	\xi_1(\gamma,a) = (2 - 7 a) \gamma^2 + 18 a^3 + 44 a^2 + 10 a - 4 < 0.
	\end{equation}
	Since $a\in(-2,-\tfrac{1}{2})\setminus\{-1\}$, the study of the function $\xi_1$ shows that $\xi_1(\gamma,a)>\tfrac{-1}{4}$. Therefore, we have that necessarily $\xi_1(\gamma,a)\in(-\tfrac{1}{4},0)$.
	Secondly, we use the expression for  $H^2$ in $\frac{1}{H}\mathbf{\tilde g}_{25}=0$ to obtain
	\[
	F^2  = \tfrac{(a - 1) \left((2 a - 7) \gamma^2 - 4 a^3 + 10 a^2 + 44 a + 18\right)}
	{4 (a^2 + 5 a + 1)}.
	\]
	It follows from this expression that 
	\begin{equation}\label{eq: H3 ineq 5-2}
	\xi_2(\gamma,a) = (2 a - 7) \gamma^2 - 4 a^3 + 10 a^2 + 44 a + 18 > 0
	\end{equation}
	and, in view of the domain of the parameter $a$, $\xi_2$ satisfies $\xi_2(\gamma,a)\in(0,2)$.
	Finally, it is easy to check that (\ref{eq: H3 ineq 5-1}) and (\ref{eq: H3 ineq 5-2}) are incompatible  for
	$a\in(-2,-\tfrac{1}{2})\setminus\{-1\}$. Indeed, assuming (\ref{eq: H3 ineq 5-2}),   $\xi_1(\gamma,a)\in(8,\tfrac{120}{11})$ is obtained, which contradicts the negativity of $\xi_1$.

	\medskip
	
	\paragraph{\underline{\emph{Case 
				$3 a^2 + 4 c^2 + 10 a + 3=0$, $a+c^2\neq0$, $c\neq 0$, $a\neq 1$}}}
	In  the polynomial ring $\mathbb{R}[\tilde\gamma,H,\gamma,F,c,a,d]$ we compute a Gröbner basis for the ideal $\mathcal{I}_{26}$ generated by $\mathcal{I}_2\cup\{3 a^2 + 4 c^2 + 10 a + 3\}$. We get $32$ polynomials, being two of them  
	\[
	\mathbf{g}_{26}^1 = (a - 1) (a + 1)^2 (5 a - 1) c F^2  \mathbf{p}_{26}^1,\,\,
	\text{and } \,\,
	\mathbf{g}_{26}^2 = -(a+1)^2 (5a-1) c F^2 \mathbf{p}_{26}^2,
	\]
	where
	\small\[
	\begin{array}{l}
	\mathbf{p}_{26}^1 = 
	100 (a - 1) (a + 1)^2 (61 a^2 + 186 a + 61) F^4
	
	\\
	\noalign{\medskip}
	\phantom{\mathbf{p}_{26}^1 }
	
	+4 (745 a^7 + 311 a^6 - 24079 a^5 - 20153 a^4 + 116387 a^3 + 
	22285 a^2 - 22397 a - 5515) F^2
	
	\\
	\noalign{\medskip}
	\phantom{\mathbf{p}_{26}^1 }
	
	+(a - 5)^2 (a - 1) (a + 3) (3 a + 1) (305 a^4 + 164 a^3 - 2330 a^2 + 
	164 a + 305) ,
	
	\\
	\noalign{\medskip} 
	
	\mathbf{p}_{26}^2 =  
	6 (a - 1) (115 a^4 + 1372 a^3 + 3538 a^2 + 1372 a + 115) F^2
	
	\\
	\noalign{\medskip}
	\phantom{\mathbf{p}_{26}^2 }
	
	-(245 a^5 - 449 a^4 - 1878 a^3 + 6198 a^2 + 673 a - 565) \gamma^2
	
	\\
	\noalign{\medskip}
	\phantom{\mathbf{p}_{26}^2  }
	
	+575 a^7 - 535 a^6 - 12917 a^5 - 4619 a^4 + 43261 a^3 + 3499 a^2 - 
	10183 a - 2185 .
	\end{array}
	\]\normalsize
	Note that $cF\neq 0$.  Furthermore, from  the conditions $3 a^2 + 4 c^2 + 10 a + 3=0$, $a+c^2\neq0$  and  $c\neq 0$, it follows that $a\in(-3,-\tfrac{1}{3})\setminus\{-1\}$. Hence  $a=-1$ and $a=1/5$ are not admissible values and, from $\mathbf{g}_{26}^1=\mathbf{g}_{26}^2=0$, we get that $\mathbf{p}_{26}^1=\mathbf{p}_{26}^2=0$. In what follows we show that there are no critical metrics in this case. Firstly,  we consider  $\mathbf{p}_{26}^1=0$ as a biquadratic equation in the unknown $F$. A  direct analysis shows that the existence of  a real solution for $F$ restricts the values of the parameter $a$ from   $(-3,-\tfrac{1}{3})\setminus\{-1\}$ to  $\left( \frac{1}{61}\left(-93-8\sqrt{77}\right) , \frac{1}{61}\left(-93+8\sqrt{77}\right) \right)\setminus\{-1\}$.
	Secondly, note that the restriction on the parameter $a$ implies    $(a-1)(115 a^4 + 1372 a^3 + 3538 a^2 + 1372 a + 115)<0$, so  we can clear $F^2$ in $\mathbf{p}_{26}^2=0$  and its positivity leads to 
	\small\[
	\begin{array}{l}
	\xi(\gamma,a) = (245 a^5 - 449 a^4 - 1878 a^3 + 6198 a^2 + 673 a - 565) \gamma^2
	
	\\
	\noalign{\medskip}
	\phantom{\xi(\gamma,a)}
	
	-575 a^7 + 535 a^6 + 12917 a^5 + 4619 a^4 - 43261 a^3 - 3499 a^2 + 10183 a + 2185   <  0.
	\end{array}
	\]\normalsize
	$\xi(\gamma,a)$ is a polynomial for $\gamma$ of the form $\alpha\gamma^2+\beta<0$. But, for values of  $a\in \left( \frac{1}{61}\left(-93-8\sqrt{77}\right) , \frac{1}{61}\left(-93+8\sqrt{77}\right) \right)\setminus\{-1\}$, we check that $\alpha>0$ and $\beta>0$, which leads to a contradiction. Hence, we conclude that there are no critical metrics in this case, which finishes the proof.
\end{proof}

\begin{remark}\rm	
	The metrics in  Families~(1) and (3) in Theorem~\ref{th:heisenberg} are isomorphically homothetic  to left-invariant metrics in the Lie group $\mathbb{R}\ltimes\mathbb{R}^3$  as shown in Remark~\ref{re:xx}. Hence they correspond to algebraic Ricci solitons (i) and (ii) as discussed in Section~\ref{ss:1-3-2}.  Metrics in Family~(2) correspond to (iv) in Section~\ref{ss:1-3-2}.
\end{remark}

\section{Left-invariant $\mathcal{F}_t$-critical metrics on $\mathbb{R}\ltimes\mathbb{R}^3$}\label{se:R}

Let $\mathfrak{g}=\mathbb{R}\ltimes\mathfrak{r}^3$ be a semi-direct extension of the Abelian Lie algebra $\mathfrak{r}^3$. Let $\langle \cdot, \cdot \rangle$ be an inner product on $\mathfrak{g}$ and $\langle \cdot, \cdot \rangle_3$ its restriction to $\mathfrak{r}^3$. 
The algebra of all derivations $\mathfrak{D}$ of $\mathfrak{r}^3$ is $\mathfrak{gl}(3, \mathbb{R})$. If we fix $\mathfrak{D}\in\mathfrak{gl}(3, \mathbb{R})$, there exists a $\langle \cdot, \cdot \rangle_3$-orthonormal basis $\{ \mathbf{v}_1,\mathbf{v}_2,\mathbf{v}_3\}$ of $\mathfrak{r}^3$ where $\mathfrak{D}$ decomposes as a sum of a diagonal matrix and a skew-symmetric matrix. Hence  
\[ \operatorname{der}(\mathfrak{r}^3)=\left \{  \left( \begin{array}{ccc}
\tilde a & -\tilde b & -\tilde c 
\\
\tilde b & \tilde f & -\tilde h 
\\
\tilde c & \tilde h & \tilde p 
\end{array} \right);\, \tilde a, \tilde b, \tilde c, \tilde f, \tilde h, \tilde p \in \mathbb{R} \right \}.
\]
%
%
Let $\{\mathbf{v}_1, \mathbf{v}_2, \mathbf{v}_3, \mathbf{v}_4\}$ be a basis of    $\mathfrak{g}=\mathbb{R}\mathbf{v}_4\oplus \operatorname{span}\{\mathbf{v}_1, \mathbf{v}_2, \mathbf{v}_3\}$.
Since $\mathbb{R}\mathbf{v}_4$ is not necessarily orthogonal to $\mathfrak{r}^3$, we set $\tilde  k_i=\langle \mathbf{v}_i, \mathbf{v}_4 \rangle$, for $i = 1,2,3$. Let $\bar e_4= \mathbf{v}_4-\sum_i \tilde k_i\mathbf{v}_i$ and normalize it ($e_4=\bar e_4 \|\bar e_4\|^{-1}$) to get an orthonormal basis $\{ e_1,\dots,e_4\}$ of $\mathfrak{g}=  \mathbb{R} \oplus \mathfrak{r}^3$, where $e_i=\mathbf{v}_i$ for $i=1,2,3$. Now we set $a=-\tilde a \|e_4\|^{-1}$, $f=-\tilde f\|e_4\|^{-1}$, $p=-\tilde p\|e_4\|^{-1}$, $b=-\tilde b\|e_4\|^{-1}$, $c=-\tilde c\|e_4\|^{-1}$ and  $h=-\tilde h\|e_4\|^{-1}$ so that the Lie brackets are given by
\begin{equation}\label{eq:cor_R3}
\begin{array}{ll}
[e_1,e_4]= a e_1 + b e_2 + c e_3,
&
[e_2,e_4]=-b e_1 + f e_2 + h e_3,
\\ 
\noalign{\medskip} 
[e_3,e_4]=- c e_1 -  h e_2 +  p e_3.
\end{array}
\end{equation}

\begin{remark}\rm\label{re:r3}
	The scalar curvature of a left-invariant  metric given by~\eqref{eq:cor_R3}, $\tau=-2 \left(a^2+f^2+p^2+a (f+p)+f p\right)$, vanishes if and only if the metric is flat.
	Moreover, a metric given by \eqref{eq:cor_R3} is Einstein if and only if the self-adjoint part of the derivation is a multiple of the identity, $a=f=p$, in which case 	the sectional curvature is constant $K=-a^2$. 
	
	Additionally to Einstein metrics, those which are homothetic to metrics in \eqref{eq:cor_R3} with $a=b=p=h=0$, or $a=p$ and $b=f=h=0$ are locally symmetric. In the former case they are locally homothetic to a product $\mathbb{R}^2\times \mathbb{H}^2$, whereas in the latter case they are locally conformally flat and locally homothetic to $\mathbb{R}\times \mathbb{H}^3$.
\end{remark}

\begin{remark}\rm\label{re:iso-R3}
	Left-invariant metrics \eqref{eq:cor_R3} are determined by a vector $(a,f,p,b,c,h)\in\mathbb{R}^6$.
	The isometry  
	$( e_1, e_2, e_3, e_4)\mapsto(e_2,e_1,e_3,e_4)$
	shows that $(a,f,p,b,c,h)\sim (f,a,p,-b,h,c)$. Analogously, the isometry
	$( e_1, e_2, e_3, e_4)\mapsto(e_3,e_2,e_1,e_4)$ gives $(a,f,p,b,c,h)\sim (p,f,a,-h,-c,-b)$  and the isometry 
	$(e_1, e_2, e_3, e_4)\mapsto(e_1,e_3,e_2,e_4)$ shows that $(a,f,p,b,c,h)\sim (a,p,f,c,b,-h)$.
\end{remark}

\begin{theorem}\label{th:criticas-R3}
	A non-symmetric left-invariant metric on  $\mathbb{R}\ltimes\mathbb{R}^3$ is critical for a quadratic curvature functional $\mathcal{F}_t$  if and only if 
	one of the following holds:
	\begin{itemize}
		\item[(a)] The functional has zero energy and the metric  is  homothetic to a left-invariant metric determined by one of the following
		\begin{enumerate}
			\item 
			$[e_1,e_4]=e_1+ e_3$  and $[e_3,e_4]=-  (e_1+e_3)$.
			In this case, $t=-3$.

			\smallskip
			\item 
			$[e_1,e_4]=e_1+\frac{1}{\sqrt{2}}e_3$, 
			$[e_2,e_4]=-e_2+\frac{1}{\sqrt{2}}  e_3$  and 
			$[e_3,e_4]=-\frac{1}{\sqrt{2}}( e_1+ e_2)$.
			In this case, $t=-\frac{3}{2}$.

			\smallskip
			\item
			$[e_1,e_4]=e_1$, $[e_2,e_4]=f e_2$  and $[e_3,e_4]=p e_3$, where the parameters 			 $\{(f,p)\in\mathbb{R}^2;\, -1\leq f\leq p\leq 1\}\setminus \{(-1,p);\, -1\leq p<0\}$
			and $(f,p)\notin\{ (0,0), (0,1), (1,1) \}$.
			In this case,  $t=-\frac{f^2+p^2+1}{2(f^2+p^2+ f p+f+p+1)}$.

			\smallskip
			\item 
			$[e_1,e_4]=e_1$, $[e_2,e_4]=-(\kappa+\frac{1}{2})e_2+he_3$, $[e_3,e_4]=-he_2+(\kappa-\frac{1}{2})e_3$, 
with $\kappa\in (\frac{\sqrt{3}}{2},+\infty )$. For a fixed $\kappa$, the parameter $h$ is given by the only positive solution of $h^2(4 \kappa^2-3)-\kappa^2(4\kappa^2+3)=0$.
			In this case, $t=-\frac{48\kappa^4-9}{16\kappa^4-9}$.

			\smallskip
			\item 
			$[e_1,e_4]=e_1+c e_3$, $[e_2,e_4]=-(p+1)e_2+h e_3$,    $[e_3,e_4]=-c e_1-h e_2+p e_3$, 
with $p\in\left(\zeta,-1\right)\cup (\frac{1-\sqrt{21}}{10},0 )$, where $\zeta=-1,697464\dots$ is the only real solution of the equation $8p^3+15p^2+3p+1=0$.
For a fixed $p$, the parameters $c$ and $h$ are given by the only positive solutions of 
$$
\begin{array}{l}
c^2(p+2)(5p^2-p-1)+(2p+1)(8p^3+15p^2+3p+1)=0,
\\
h^2(p+2)(5p^2-p-1)-(p+1)(p-1)(5p^3+12p^2+1)=0.
\end{array}
$$
In this case, $t=-\frac{30p^4-3p^2-6p-3}{2(p^2+p+1)(5p^2-p-1)}$.
	
		\end{enumerate}

		\smallskip
			
		\item[(b)] The energy of the functional is non-zero and the metric is homothetic to one of the following:
		\begin{enumerate}
			\item[(6)] 
			$[e_1,e_4]=\tfrac{1}{3} e_1$, 
			$[e_2,e_4]=f e_2 + h e_3$, 
			$[e_3,e_4]=-h e_2 - (f-\tfrac{2}{3} )e_3$, 
with $h\in(0,+\infty)$. For a fixed $h$, the parameter $f$ is given by the only positive solution of 
$36 (f^2 - h^2) - 24 f - 5=0$.
			In this case, $t=-\frac{36h^2+5}{12h^2+11}$ and the energy is given by $\energy_t = -\frac{16}{3}h^2$.
	
			\smallskip			
			\item[(7)] 
			$[e_1,e_4]=a e_1 + b e_2$, 
			$[e_2,e_4]= -b e_1 +\tfrac{1}{3} e_2 +  b e_3$
			and 
			$[e_3,e_4]=- b e_2- (a-\tfrac{2}{3} ) e_3$,
			with $b\in(0,+\infty)$. For a fixed $b$, the parameter $a$ is given by the only positive solution of $9 a^2 - 18 b^2 - 6 a - 8=0$.
			In this case, $t=-\frac{18b^2+7}{12b^2+10}$ and the energy is given by	$\energy_t=-\frac{8}{3}b^2$.
		
			\end{enumerate}
	\end{itemize}
	Moreover, metrics in Families~(1)--(3) are algebraic Ricci solitons, while metrics corresponding to Families~(4)--(7) are not.
\end{theorem}

\begin{remark}\rm\label{re:grafico-R3}
	The range of the parameter $t$    in each family of  Theorem~\ref{th:criticas-R3} is indicated in  Figure~\ref{figura-R}. The energy vanishes in Families~(1)--(5), whereas $\energy_t<0$ otherwise.
	\begin{figure}[h]
	\begin{center}
		\begin{tikzpicture}
		\draw[thick,-latex]   (-4.9-0.15,0) -- (5.7,0) node[right] {$t$}; 
		\foreach \i/\n in 
		{
			{-4.05}/$\vartheta$,
			{-2.8}/$-3$, 
			{-0.15}/$-\frac{3}{2}$, 
			{1.05}/$-1$, 
			{1.95}/$-\frac{7}{10}$, 
			{2.95}/$-\frac{5}{11}$, 
			{3.75}/$-\frac{1}{3}$,
			{4.45}/$-\frac{1}{4}$
		}
		{\draw (\i,-.2)--(\i,.2);}
		
		\foreach \i/\n in 
		{
			{-4.05}/$\vartheta$,
			{-2.8}/$-3$,  
			{1.05}/$-1$, 
		}
		{\draw  (\i,.2) node[above] {\footnotesize \n};}

		\foreach \i/\n in 
		{
			{-0.15}/$-\frac{3}{2}$, 
			{1.95}/$-\frac{7}{10}$, 
			{2.95}/$-\frac{5}{11}$, 
			{3.75}/$-\frac{1}{3}$,
			{4.45}/$-\frac{1}{4}$
		}
		{\draw  (\i,.12) node[above] {\footnotesize \n};}
		
		\foreach \i/\n in 
		{
			{-4.05}/$\vartheta$,
			{-2.8}/$-3$,  
			{-0.15}/$-\frac{3}{2}$, 
			{1.05}/$-1$, 
			{1.95}/$-\frac{7}{10}$, 
			{2.95}/$-\frac{5}{11}$, 
			{4.45}/$-\frac{1}{4}$
		}{
			\draw[dashed] (\i,-.25)--(\i,-3.75);
		}

		\draw (-5.45,-0.45) node {\footnotesize (1) };
		\filldraw [color=\ColorECeroARS,fill=\ColorECeroARS] (-2.8,-0.45) circle (2pt);

		\draw (-5.45,-0.95) node {\footnotesize (2) };
		\filldraw [color=\ColorECeroARS,fill=\ColorECeroARS] (-0.15,-0.95) circle (2pt);

		\draw (-5.45,-1.45) node {\footnotesize (3) };
		
		\draw[densely dotted,color=\ColorECeroARS] 	
			(1.05,-1.45-\YSepInfMetricas/4)--(4.45,-1.45-\YSepInfMetricas/4);
				\draw[densely dotted,color=\ColorECeroARS] 	
			(1.05,-1.45-\YSepInfMetricas/8)--(4.45,-1.45-\YSepInfMetricas/8);
				\draw[densely dotted,color=\ColorECeroARS] 	
			(1.05,-1.45-\YSepInfMetricas/3)--(4.45,-1.45-\YSepInfMetricas/3);
	
		\draw[color=\ColorECeroARS] (1.05,-1.45-\YSepInfMetricas/2)--(4.45,-1.45-\YSepInfMetricas/2);
		\filldraw [color=\ColorECeroARS,fill=\ColorECeroARS] (1.05,-1.45-\YSepInfMetricas/2) circle (2pt);
		\filldraw [color=\ColorECeroARS,fill=white] (4.45,-1.45-\YSepInfMetricas/2) circle (2pt);

		\draw (-5.45,-1.95) node {\footnotesize (4) };
		\draw[color=\ColorECeroNoARS,latex-=] (-4.8,-1.95)--(-2.8,-1.95);
		\filldraw [color=\ColorECeroNoARS,fill=white] (-2.8,-1.95) circle (2pt);

		\draw (-5.45,-2.45) node {\footnotesize (5) };
		\draw[color=\ColorECeroNoARS,latex-=] (-4.8,-2.45-\YSepDosMetricas)--(-0.15,-2.45-\YSepDosMetricas);
		\filldraw [color=\ColorECeroNoARS,fill=white] (-0.15,-2.45-\YSepDosMetricas) circle (2pt);
		
		\draw[color=\ColorECeroNoARS] 			
			(-4.05,-2.45+0.7*\YSepDosMetricas)--(-2.8,-2.45+0.7*\YSepDosMetricas);
		\filldraw [color=\ColorECeroNoARS,fill=white] 		
			(-4.05,-2.45+0.7*\YSepDosMetricas) circle (2pt);		
		\filldraw [color=\ColorECeroNoARS,fill=white] 
			(-2.8,-2.45+0.7*\YSepDosMetricas) circle (2pt);

		\draw (-5.45,-2.95) node {\footnotesize (6) };
		\draw[color=\ColorENegativa] (-2.8,-2.95)--(2.95,-2.95);
		\filldraw [color=\ColorENegativa,fill=white] (-2.8,-2.95) circle (2pt);
		\filldraw [color=\ColorENegativa,fill=white] (2.95,-2.95) circle (2pt);

		\draw (-5.45,-3.45) node {\footnotesize (7) };
		\draw[color=\ColorENegativa] (-0.15,-3.45)--(1.95,-3.45);
		\filldraw [color=\ColorENegativa,fill=white] (-0.15,-3.45) circle (2pt);
		\filldraw [color=\ColorENegativa,fill=white] (1.95,-3.45) circle (2pt); 
		\end{tikzpicture}
		\caption{Range of the parameter $t$ for non-symmetric homogeneous $\mathcal{F}_t$-critical metrics on $\mathbb{R}\ltimes\mathbb{R}^3$.}
		\label{figura-R}
	\end{center}
\end{figure}
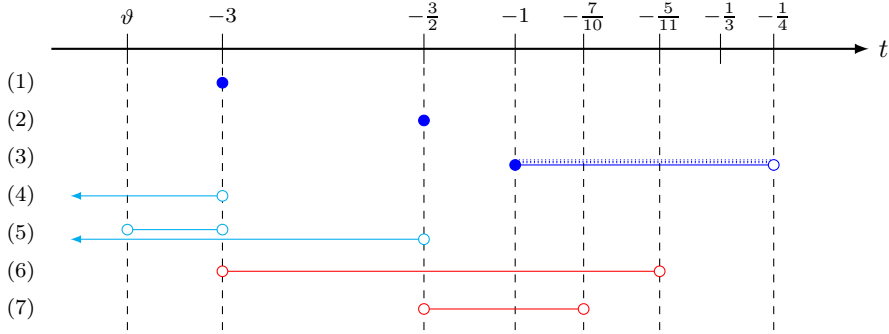

\noindent For the special choice of $ f=(\sqrt{p}-1)^2$ with $\tfrac{1}{4}\leq p< 1$ in Family~(3), the corresponding metrics are non-homothetic Bach-flat Ricci solitons (see \cite{CL-GM-GR-GR-VL} and Remark~\ref{re:Bach-llanas}).

Proceeding as in previous sections, the set of homothetic invariants $\{t,\|\rho\|^2\, \tau^{-2}$, $\| R\|^2\, \tau^{-2}$, $\|\nabla\rho\|^2\,\tau^{-3}$, $\|\nabla R\|^2\tau^{-3}\}$ distinguishes classes among metrics given in Theorem~\ref{th:criticas-R3}, Theorem~\ref{th:heisenberg} and Theorem~\ref{th:E(1,1)} with the exception of the algebraic Ricci solitons previously discussed in Section~\ref{ss:1-3-2}.
Furthermore, in Theorem~\ref{th:criticas-R3}, taking into account the restrictions on the parameters in each case, it follows that there are no homotheties between metrics in different families and, moreover,  that different values of the parameters in any of the families   correspond to metrics which are not homothetic.

Moreover, for any $t\in[-1,-\frac{1}{4})$ there is an infinite number of homothetically inequivalent $\mathcal{F}_t$-critical metrics in Family~(3).  For each $t\in(\vartheta,-3)$, where $\vartheta=-3,753199\dots$ is the only real solution of $192 p^3+   1152  p^2+1865 p+923=0$, Family~(5) provides two non homothetic $\mathcal{F}_t$-critical metrics.
However, for any other admissible value of $t$ there is a single (up to homothety) $\mathcal{F}_t$-critical metric in each family.
\end{remark}

\begin{proof}	
	A left-invariant metric is $\mathcal{F}_t$-critical if and only if the symmetric $(0,2)$-tensor field 
	$
	\mathfrak{F}_t=
	-\Delta \rho 
	+\frac{1}{2} (\|\rho\|^2 +t\tau^2) g 
	-2 R[\rho] 
	-2t\tau\rho
	$
	vanishes, where the components $\mathfrak{F}_t(e_i,e_j)$ are polynomials on $t$ and  the structure constants determining the metric~\eqref{eq:cor_R3}.
	In particular, we have 
	\[
	\begin{array}{l}
	\mathfrak{F}_t(e_4,e_4) 
	=
	\frac{1}{2}\left( (a - f)^2 + (a - p)^2 + (f - p)^2 \right) 
	\left(\tau t - (a^2+f^2+p^2)\right)
	
	\\ 
	\noalign{\medskip}
	\phantom{\mathfrak{F}_t(e_4,e_4) 
		=}	
	
	-3 b^2 (a-f)^2-3 c^2 (a-p)^2-3 h^2 (f-p)^2 ,
	\end{array}
	\]  
	where  the coefficient of $t$, $\frac{1}{2}\tau \left ((a - f)^2 + (a - p)^2 + (f - p)^2 \right)$, 
	vanishes if and only if the metric is Einstein (see Remark~\ref{re:r3}). Thus, in the non-Einstein case,
	\begin{equation}\label{eq: R3-t}
	t =   \tfrac{ \left(a^2+f^2+p^2\right) \left( (a - f)^2 + (a - p)^2 + (f - p)^2 \right)
		+6 b^2 (a-f)^2+6 c^2 (a-p)^2+6 h^2 (f-p)^2 }
	{\tau  \left( (a - f)^2 + (a - p)^2 + (f - p)^2 \right) }
	\end{equation}
	and the energy is given by
	\begin{equation}\label{eq: R3-energia}
	\energy_t = 
	-\tfrac{8 \left(a+f+p\right)^2 \left(b^2 (a-f)^2 + c^2 (a-p)^2 + h^2 (f-p)^2 \right)}
	{ (a - f)^2 + (a - p)^2 + (f - p)^2 }.
	\end{equation}
	Let $\mathfrak{F}$ be  the symmetric $(0,2)$-tensor field $\mathfrak{F}_t$ for the value of $t$ given by~(\ref{eq: R3-t}). In order to solve the system  $\{\mathfrak{F}(e_i,e_j)=0\}$, we consider the self-adjoint part of the derivation given by $\operatorname{diag}[a,f,p]$ and split the analysis into two cases: two of the parameters are equal or the three of them are different (notice that the metric is Einstein if $a=f=p$). Note that in the former case one may assume (see Remark~\ref{re:iso-R3}) that $a=f\neq p$.
	In the latter case, where $a$, $f$ and $p$ are distinct, we analyze separately derivations with trace-free self-adjoint part ($a+f+p=0$) and derivations with $a+f+p\neq0$. Moreover, if $a+f+p\neq0$ we further distinguish whether parameters $b$, $c$ and $h$ determining the skew-symmetric part of the derivation satisfy $bch=0$ or $bch\neq 0$.

	\subsection{\boldmath  Case  $a=f\neq p$\unboldmath}
	If $a=f\neq p$  we can simplify   the non-zero factor $a-p$ by considering the polynomials $\frac{1}{a-p}\mathfrak{F}(e_i,e_j)$. 	
	Denoting by   $\mathfrak{\overline F}(e_i,e_j)$  the polynomials $\mathfrak{F}(e_i,e_j)$ including the simplifications above, the non-zero components of $\mathfrak{\overline F}$ are determined by
	\[
	\begin{array}{l}
		\mathfrak{\overline F}(e_1,e_1)=-\mathfrak{\overline F}(e_2,e_2)= 
		-2 (c^2 - h^2) (2 a+p)-6 b c h,
		\\ \noalign{\medskip}
		\mathfrak{\overline F}(e_1,e_2)=
		3b (c^2- h^2)-4 c h (2 a+p),
		\\ \noalign{\medskip}
		\mathfrak{\overline F}(e_1,e_3)=
		c(b^2-2 c^2-2  h^2)-3 a c (a+2 p)-2 b h (2 a+p) ,
		
		\\ \noalign{\medskip}
		\mathfrak{\overline F}(e_2,e_3)=
		h(b^2 -2 c^2 -2 h^2)-3 a h (a+2 p)+2 b c (2 a+p),
	\end{array}
	\]
	from where it follows that $h\,  \mathfrak{\overline F}(e_1,e_3) - c\,  \mathfrak{\overline F}(e_2,e_3) = 
	-2 b (c^2+h^2) (2a+p)$. Thus, $c=h=0$, or  $b=0$, or $p=-2a$. We analyze the three cases separately.

	\subsubsection{Case $c=h=0$} Clearly $\mathfrak{\overline F}=0$ and  (\ref{eq: R3-energia}) implies that  the energy vanishes.  Moreover,   for $a=0$  the metric is locally symmetric and locally isometric to a product $\mathbb{R}^2\times N^2(\kappa)$, with $\kappa<0$ (see Remark~\ref{re:r3}), which is a rigid gradient Ricci soliton. Now, for  $a\neq 0$ we work with the homothetic metric determined by $a=1$, which corresponds to 
	\[
	[e_1,e_4]=  e_1+ b e_2,
	\quad 
	[e_2,e_4]=-b  e_1 +   e_2,
	\quad 
	[e_3,e_4]=p e_3,
	\]
	and a straightforward calculation shows that the sectional curvature is independent of the structure constant $b$. Hence it follows from the work of Kulkarni (see \cite{Kulkarni}) that the metric is homothetic to a metric with $b=0$ and Equation~(\ref{eq: R3-t}) implies that $t=-\frac{p^2+2}{2(p^2+2p+3)}$. Moreover, $\operatorname{Ric}+(p^2+2)\operatorname{Id}$ is a derivation determining an algebraic Ricci soliton. This is a subfamily of Family~(3).

	\subsubsection{Case $b=0$ and either $c\neq 0$ or $h\neq 0$} We have
	\[
	\mathfrak{\overline F}(e_1,e_1)= -2 (c^2 - h^2) (2 a + p),
	\quad
	\mathfrak{\overline F}(e_1,e_2)= -4 c h (2 a + p) ,
	\]
	which imply $p=-2a$. Note that, by (\ref{eq: R3-energia}), the energy is zero. Moreover, since  $a\neq p$, $a$ cannot vanish so we set $a=1$ and work with a homothetic metric.  Now $\mathfrak{\overline F}=0$ reduces to 
	$2 (c^2 + h^2)=9$ and the left-invariant metric is given by
	\[
	[e_1,e_4]=  e_1+ c e_3,
	\quad 
	[e_2,e_4]=  e_2+h e_3,
	\quad 
	[e_3,e_4]= - c e_1-he_2 - 2e_3,
	\]	
	while (\ref{eq: R3-t}) gives  $t=-13/4$.
	Note that the isometry $e_1\mapsto e_2$ interchanges $c$ and $h$. Moreover $e_2\mapsto-e_2$ interchanges $h$ and $-h$ and the isometry $e_3\mapsto-e_3$ interchanges the sign  of $c$ and also the sign of $h$. Therefore one may assume $0\leq c\leq 3/2$  and set $h=\frac{1}{\sqrt{2}}\sqrt{9-2c^2}$.
Now, 
		considering the orthonormal basis  
		\[
		\bar e_1=\tfrac{-\sqrt{2}}{3} \left(
		h e_1-c e_2\right),
		\quad
		\bar e_2= - e_3,
		\quad
		\bar e_3=\tfrac{\sqrt{2}}{3} \left(
		c e_1 + h e_2\right),
		\quad
		\bar e_4=  e_4,
		\]
		the  non-zero brackets are given by 
		\[
		[\bar e_1,\bar e_4]= \bar e_1,
		\qquad
		[\bar e_2,\bar e_4]= -2\bar e_2 +\tfrac{3}{\sqrt{2}}\bar e_3, 
		\qquad
		[\bar e_3,\bar e_4]= -\tfrac{3}{\sqrt{2}}\bar e_2 + \bar e_3.
		\]
		These metrics are homothetic  to  that in Family~(4) with $\kappa=3/2$ and $h= 3/\sqrt{2}$.

	\subsubsection{Case  $p=-2a$, $b\neq 0$ and either $c\neq 0$ or $h\neq 0$} We have
	\[
	\mathfrak{\overline F}(e_1,e_1)= -6 b c h,
	\quad
	\mathfrak{\overline F}(e_1,e_2)= 3 b (c^2 - h^2)  ,
	\]
	so there are no critical metrics in this case.

	\subsection{\boldmath  Case  $a\neq f\neq p$, $a\neq p$, with $a+f+p=0$\unboldmath}\label{case 5.2}
	Computing the expression in \eqref{eq: R3-energia}, we check that the energy vanishes. 
	
	Proceeding as in the previous case,  we remove the denominators in $\mathfrak{F}(e_i,e_j)$   multiplying by $\frac{1}{2}\left( (a - f)^2 + (a - p)^2 + (f - p)^2 \right)$ and  simplify the non-zero factors $f-p$, $a-p$ and $a-f$ by considering 
	$\frac{1}{f-p}\mathfrak{F}(e_1,e_1)$, $\frac{1}{a-p}\mathfrak{F}(e_2,e_2)$ and $\frac{1}{a-f}\mathfrak{F}(e_3,e_3)$.
	Denoting by   $\mathfrak{\overline F}(e_i,e_j)$  the polynomials $\mathfrak{F}(e_i,e_j)$ including those simplifications above in each case, the non-zero components  are determined by $\mathfrak{\overline F}(e_1,e_1)$, $\mathfrak{\overline F}(e_2,e_2)$, $\mathfrak{\overline F}(e_3,e_3)$, $\mathfrak{\overline F}(e_1,e_2)$, $\mathfrak{\overline F}(e_1,e_3)$ and $\mathfrak{\overline F}(e_2,e_3)$. We observe that $\mathfrak{\overline F}(e_1,e_1)=-\mathfrak{\overline F}(e_2,e_2)=\mathfrak{\overline F}(e_3,e_3)$, so    $\mathfrak{\overline F}$ vanishes if and only if the components
	$\mathfrak{\overline F}(e_1,e_1)$, $\mathfrak{\overline F}(e_1,e_2)$, $\mathfrak{\overline F}(e_1,e_3)$ and $\mathfrak{\overline F}(e_2,e_3)$ are zero. 	
	Setting $f=-a-p$ we have
	\[
	\mathfrak{\overline F}(e_1,e_1) = -18 b c h (a^2+ p^2 + a p),
	\]
	which shows that the skew-symmetric part of the derivation satisfies $bch=0$.
	By Remark~\ref{re:iso-R3} we can assume $b=0$ and $a\neq 0$. Hence we work with a representative in the homothetic class which has $a=1$ and Lie brackets given by
	$$
	[e_1,e_4]=e_1+ce_3,
	\quad 
	[e_2,e_4]=-(p+1)e_2+he_3,
	\quad 
	[e_3,e_4]=-ce_1-he_2+pe_3.
	$$
	Moreover,    $p\notin\{-2,-\frac{1}{2}, 1\}$ since  we are assuming that  $a=1$, that $f=-p-1$ and that $a$, $f$ and $p$ are different.  
	Now, the only non-zero components of the tensor field $\mathfrak{\overline F}$ are
	\begin{equation}\label{eq: R3 Fbar case 2}
	\begin{array}{l}
	\mathfrak{\overline F}(e_1,e_3)=
	-3c \left(
	2   (p-1)(p^2+4p+1) c^2
	- ( p^3-9p^2-15p-4) h^2 
	\right.
	\\
	\noalign{\medskip}
	\phantom{\mathfrak{\overline F}(e_1,e_3)=-3c (}
	
	\left.
	+(p-1)^3(p^2+p+1) \right),
	
	\\
	\noalign{\medskip}
	\mathfrak{\overline F}(e_2,e_3)=
	- 3h\left(
	( p^3+9p^2+3p-4) c^2
	- 2(4 p^3+ 6 p^2-1) h^2
	\right.
	\\
	\noalign{\medskip}
	\phantom{	\mathfrak{\overline F}(e_2,e_3)=- 3h (}
	\left.
	+(2 p+1)^3 (p^2+p+1)
	\right) .
	\end{array}
	\end{equation}
	Next we analyze different  possibilities depending on whether or not $c$ or $h$ vanish.

	\subsubsection{Case $c=h=0$} 
	In this case  $\mathfrak{\overline F}$ vanishes identically and (\ref{eq: R3-t}) gives $t=-1$. Moreover,  $\operatorname{Ric}+2( p^2+p+1)$ is a derivation determining an algebraic Ricci soliton. 
	The metric is determined by brackets
	$$
	[e_1,e_4]=e_1,
	\quad 
	[e_2,e_4]=-(p+1)e_2,
	\quad 
	[e_3,e_4]=pe_3,
	$$
	which correspond to those in Family~(3) for $f=-(p+1)$. As an application of Remark~\ref{re:iso-R3}, the range of the parameters can be specialized as in that family (see $\S$\ref{paragraph:5311} below).

	\subsubsection{Case $c=0$, $h\neq 0$} 
	If $c=0$  then the only non-zero component of $\mathfrak{\overline F}$ is
	\[
	\begin{array}{l}
	\mathfrak{\overline F}(e_2,e_3)=3 h (2 p+1) \left(
	h^2 (4p^2+4p-2)
	-(2 p+1)^2 (p^2+p+1)
	\right) 
	\end{array}
	\]
	and therefore  
	$h^2 (4p^2+4p-2)
	-(2 p+1)^2 (p^2+p+1)=0$.	
	Setting $\kappa=p+\frac{1}{2}$, this equation becomes 
	$\kappa^2(4\kappa^2+3)-h^2 (4 \kappa^2-3)=0$. Hence $\kappa>\sqrt{3}/2$ or $\kappa<-\sqrt{3}/2$  and
	$$
	[e_1,e_4]=e_1,
	\quad [e_2,e_4]=-\left(\kappa+\tfrac{1}{2}\right)e_2+he_3,\quad [e_3,e_4]=-he_2+\left(\kappa-\tfrac{1}{2}\right)e_3.
	$$	
	Moreover, $e_2\mapsto e_3$ determines an isometry which interchanges $(\kappa,h)$ and $(-\kappa,-h)$, and $e_3\mapsto-e_3$ is an isometry which interchanges $(\kappa,h)$ and $(\kappa,-h)$.
	Hence we restrict to $h>0$ and $\kappa>\sqrt{3}/2$. Also, since $p\neq 1$, we have $\kappa \neq 3/2$.
	As a consequence, using (\ref{eq: R3-t}), we get  $t=-\frac{48\kappa^4-9}{16\kappa^4-9}\in(-\infty,-3)\setminus\{-\frac{13}{4}\}$. 
	This corresponds to Family~(4) for  $\kappa \neq3/2$.	
	
	\subsubsection{Case $c\neq0$, $h= 0$} 
	If $f=-p-1$ does not vanish then, by Remark~\ref{re:iso-R3}, this case reduces to the previous one. Hence we set $p=-1$  and  observe that the only non-zero component of $\mathfrak{\overline F}$ is $\mathfrak{\overline F}(e_1,e_3)=-24 c(c^2-1)$. This implies $c^2=1$ and
	$$
	[e_1,e_4]=e_1+c e_3,
	\quad 
	[e_3,e_4]=-c e_1-e_3,\quad
	\text{with } c^2=1,
	$$
	is $\mathcal{F}_{-3}$-critical. Furthermore, the isometry $e_3\mapsto-e_3$ changes the sign of $c$, so we set $c=1$, which corresponds to Family~(1). Moreover, $\operatorname{Ric}+6\operatorname{Id}$ is a derivation determining an algebraic Ricci soliton.

	\subsubsection{Case $c h\neq0$} 
	The parameters $c$ and $h$ are determined by the expressions $\mathfrak{\overline F}(e_1,e_3)$ and $\mathfrak{\overline F}(e_2,e_3)$ in (\ref{eq: R3 Fbar case 2}).  Moreover,  $e_1\mapsto -e_1$ and $e_2\mapsto-e_2$ determine isometries which change the signs of $c$ and $h$, respectively. Hence we assume $c>0$ and $h>0$. 
	
	We recall that $p\notin\{-2,-\frac{1}{2}, 1\}$. If $5p^2-p-1\neq 0$, then
	\begin{equation}\label{eq:relationcandh}
	c^2=-\tfrac{(2 p+1) (8 p^3+15p^2+3p+1)}
	{(p+2) (5p^2-p-1)}
	\quad\text{and}\quad 
	h^2=\tfrac{(p-1) (p+1)  (5p^3+12p^2+1 )}{(p+2) (5p^2-p-1)}
	\end{equation}
and these metrics are  $\mathcal{F}_t$-critical. Moreover,  they are algebraic Ricci solitons if and only if $p=0$, in which case $\operatorname{Ric}+3\operatorname{Id}$ is a derivation,  $c=h=1/\sqrt{2}$, and $t=-3/2$ (see~(\ref{eq: R3-t})). This corresponds to Family~(2). Now, if $p\neq 0$,  (\ref{eq: R3-t}) gives $t=-\frac{30p^4-3p^2-6p-3}{2(p^2+p+1)(5p^2-p-1)}\in(-\infty,-\frac{3}{2})$ which corresponds to Family~(5). Moreover, the parameter $p$ satisfies $p\neq 0$, $p\in (-\infty,\zeta_1)\cup(\zeta_2,-1)\cup(\frac{1-\sqrt{21}}{10},\frac{1+\sqrt{21}}{10})$, where $\zeta_1$ and $\zeta_2$ are the only real solutions of the equations  $5p^3+12p^2+1=0$ and $8p^3+15p^2+3p+1=0$, respectively.	
	Let $\varepsilon_{(p+1)}=\pm 1$ denote the sign of $p+1$. The homothety $(e_1,e_2,e_3,e_4)\mapsto \frac{1}{|p+1|}(e_2, e_1, -\varepsilon_{(p+1)} e_3,-\varepsilon_{(p+1)} e_4)$ transforms the parameter $p$ into $-\frac{p}{p+1}$, interchanges $c\mapsto \frac{\varepsilon_{(p+1)} h}{p+1}$ and $h\mapsto \frac{\varepsilon_{(p+1)} c}{p+1}$ in accordance with the expressions in \eqref{eq:relationcandh}, as the solutions of $5p^3+12p^2+1=0$ and $8p^3+15p^2+3p+1=0$ are also interchanged. Hence, the transformation $p\mapsto -\frac{p}{p+1}$ maps the interval $(-\infty,\zeta_1)$ into $(\zeta_2,-1)$. Moreover, it maps $(0,\frac{1+\sqrt{21}}{10})$ into $(\frac{1-\sqrt{21}}{10},0)$, 	
	from where it follows that $p$ may be restricted to $(\zeta_2,-1)\cup(\frac{1-\sqrt{21}}{10},0)$ as stated in Family~(5).

	The   special case above, given by $5p^2-p-1=0$, corresponds to the   values of $p=\frac{1}{10}(1\pm\sqrt{21})$  and a straightforward calculation shows that it does not lead to $\mathcal{F}_t$-critical metrics.

\subsection{\boldmath  Case  $a\neq f\neq p$, $a\neq p$, with $a+f+p\neq 0$ and $bch=0$\unboldmath}\label{ref: R3 c=0}
Since $bch=0$, one of the three parameters is zero. Based on Remark~\ref{re:iso-R3}, we change basis if necessary to fix $c=0$.
First observe, from (\ref{eq: R3-energia}),  that the energy vanishes if and only if $b=h=0$.
	Secondly, after simplifying the components $\mathfrak{F}(e_i,e_j)$ exactly as in Case~\ref{case 5.2}, we have that 
	$\mathfrak{\overline F}$ is determined by the set of components
	$\{\mathfrak{\overline F}(e_1,e_1)$, $\mathfrak{\overline F}(e_1,e_2)$, $\mathfrak{\overline F}(e_1,e_3)$, $\mathfrak{\overline F}(e_2,e_3)\}$.
	Since $a+f+p\neq 0$ we can work with a homothetic metric satisfying $a+f+p=1$, which together with  $c=0$ leads to
	\[
	\mathfrak{\overline F}(e_1,e_3)=
	-2 b  h (3 f-1) (3 a^2+3 f^2+3 a f-3 a -3 f+1).
	\]
	Note that, by Remark~\ref{re:iso-R3},  the case $h=0$ is equivalent to $b=0$. 
	Moreover, the vanishing of the last factor implies  $a=f=1/3$, which does not satisfy the current assumptions.
	Hence, we analyze the cases $b=0$ and $3f-1=0$.

	\subsubsection{Case $b=0$}
	If $b=0$ then $\mathfrak{\overline F}(e_1,e_1)=2 h^2 (3 a - 1) (a + 2 f - 1)$. Since we are assuming $f\neq p$ and $a+f+p=1$, we have $a + 2 f - 1\neq 0$. Hence,  we consider the following two possibilities:

\smallskip	

\paragraph{\underline{\emph{Case $h=0$}}}\label{paragraph:5311} If $h=0$ then the tensor field $\mathfrak{\overline F}$ vanishes identically.
For convenience, we stop assuming $a+f+p=1$ and renormalize to choose a representative in the homothetic class with $a=1$. Note that this may need a reordering of the basis but is always possible since $c=b=h=0$ and $a\neq f\neq p\neq a$. Thus, the Lie brackets correspond to Family~(3):
		\[
[e_1,e_4]=e_1,
\quad 
[e_2,e_4]=f e_2,
\quad 
[e_3,e_4]=pe_3 .
\]
Now, (\ref{eq: R3-t}) gives  $t=-\frac{f^2+p^2+1}{2(f^2+p^2+ f p+f+p+1)}\in[-1,-\frac{1}{4})$ and  the energy vanishes (see~(\ref{eq: R3-energia})). Moreover,  $\operatorname{Ric}+( f^2+p^2+1)\Id$ is a derivation determining an algebraic Ricci soliton.

Since the Lie brackets are determined by the pair $(f,p)$ in this case, we explore the domain of these two parameters.
		The isometry $e_2\mapsto e_3$ interchanges $(f,p)$ with $(p,f)$, and the isometry $(e_1,e_2,e_3,e_4)\mapsto(e_2,e_1,e_3,-e_4)$ interchanges $(-1,p)$ with $(-1,-p)$. Also, if $f\neq 0$, the homothety $(e_1,e_2,e_3,e_4)\mapsto \frac{1}{f}(e_2,e_1$, $e_3$, $e_4)$ interchanges $(f,p)$ with $(\frac{1}{f},\frac{p}{f})$ whereas, if $p\neq 0$, the homothety $(e_1,e_2,e_3,e_4)$ $\mapsto$ $\frac{1}{p}(e_3,e_2,e_1,e_4)$ interchanges $(f,p)$ with $(\frac{f}{p},\frac{1}{p})$.
		
		 We use the identifications above to show that every metric in this case has a homothetic equivalent metric in Family~(3). Since $(f,p)\sim (p,f)$, we first restrict the domain of $(f,p)$ to $f\leq p$. Now, using the identification $(f,p)\sim (\frac{1}{f},\frac{p}{f})$, we remove the set $\{(f,p); f\leq p, |f|>1\}$ as every homothetic class has a representative with $|f|\leq 1$. Next, the identification $(f,p)\sim(\frac{f}{p},\frac{1}{p})$ allows us to remove the set $\{(f,p); f<p, |f|\leq 1,p>1\}$ so that the domain reduces to $\{(f,p); -1\leq f\leq p\leq 1\}$. 
Finally since $(-1,p)\sim (-1,-p)$ one can also eliminate the segment $\{(-1,p);p<0\}$. The points $(0,0)$, $(0,1)$, and $(1,1)$ are also excluded as the homogeneous space is symmetric in those cases.
	
\smallskip		
	
\paragraph{\underline{\emph{Case $a=1/3$ and $h\neq 0$}}} 
If $a=1/3$ then the only non-zero component of $\mathfrak{\overline F}$~is
		\[
		\mathfrak{\overline F}(e_2,e_3)=
		\tfrac{2}{81} h (3 f - 1)^3  (36 (f^2 - h^2) - 24 f - 5).
		\]
		Note that $3f-1\neq 0$ since we are assuming $a\neq f$. Hence  the left-invariant metric is given by
		\[
		[e_1,e_4]=\tfrac{1}{3} e_1,
		\quad 
		[e_2,e_4]=f e_2 + h e_3,
		\quad 
		[e_3,e_4]=-h e_2 -\left(f-\tfrac{2}{3}\right)e_3, 
		\]
		where $36(f^2 - h^2) - 24 f - 5=0$, and (\ref{eq: R3-t}) implies $t=-\frac{36h^2+5}{12h^2+11}\in(-3,-\frac{5}{11})$. A straightforward calculation shows that the Lie groups above are not algebraic Ricci solitons and, furthermore,   (\ref{eq: R3-energia}) implies that   $\energy_t = -\frac{16}{3}h^2$, so the energy is non-zero.
		Moreover, the isometry  $e_3\mapsto-e_3$ changes the sign of the parameter $h$, and $(e_1,e_2,e_3,e_4)\mapsto (e_1,e_3,-e_2,e_4)$ transforms the parameters $(f,h)$ into $(\frac{2}{3}-f,h)$. Hence we may consider $h>0$ and, once $h$ is fixed, we take $f>5/6$ given as the only positive solution of  $36(f^2 - h^2) - 24 f - 5=0$. Thus, Family~(6) is obtained.

	\subsubsection{Case $3f-1=0$, $b\neq 0$}
 The non-zero components of $\mathfrak{\overline F}$ are determined by
	\[
	\begin{array}{l}
	 \mathfrak{\overline F}(e_1,e_1)=- \mathfrak{\overline F}(e_2,e_2)= \mathfrak{\overline F}(e_3,e_3)= -\tfrac{2}{3}(3a-1)^2 (b^2-h^2),
	\\
	\noalign{\medskip}
	\mathfrak{\overline F}(e_1,e_2) = \tfrac{1}{81}(3a-1)^3 b (9 a^2 + 18 b^2 - 36 h^2 - 6 a - 8),
	\\
	\noalign{\medskip}
	\mathfrak{\overline F}(e_2,e_3) = \tfrac{1}{81} (3a-1)^3 h (9 a^2 - 36 b^2 + 18 h^2 - 6 a - 8).
	\end{array}
	\]
	Since $a\neq f$, we have $3a-1\neq 0$, so $b^2=h^2$ and $9 a^2 - 18 b^2 - 6 a - 8=0$. Therefore, in this case, a $\mathcal{F}_t$-critical metric is given by
	$$
	[e_1,e_4]=a e_1 + b e_2,
	\quad 
	[e_2,e_4]= -b e_1 +\tfrac{1}{3} e_2 + \varepsilon b e_3,
	\quad 
	[e_3,e_4]=-\varepsilon b e_2-\left(a-\tfrac{2}{3}\right) e_3, 
	$$
	where $\varepsilon^2=1$ and $9 a^2 - 18 b^2 - 6 a - 8=0$. Moreover, (\ref{eq: R3-t}) gives $t=-\frac{18b^2+7}{12b^2+10}\in(-\frac{3}{2},-\frac{7}{10})$ and  the energy, given by $\energy_t=-\frac{8}{3}b^2$ (see~(\ref{eq: R3-energia})), is non-zero.  Also, the isometry  $(e_1,e_2,e_3,e_4)\mapsto (-e_1,e_2,-e_3,e_4)$ changes the sign of $b$, whereas the isometry $e_1\mapsto -e_1$  transforms $(a,b,\varepsilon)$ into $(a,-b,-\varepsilon)$, so  we can take $\varepsilon=1$ and $b>0$.  Furthermore, the isometry $(e_1,e_2,e_3,e_4)\mapsto (e_3,-e_2,e_1,e_4)$ transforms the parameter $a$ into $\frac{2}{3}-a$, thus interchanging the two real roots for $a$ in the polynomial $9 a^2 - 18 b^2 - 6 a - 8=0$. Hence, we can take $a>4/3$, and Family~(7) is obtained.

	\subsection{\boldmath  Case  $a\neq f\neq p$, $a\neq p$, with $a+f+p\neq 0$ and $bch\neq 0$\unboldmath} 
	In this last case  we will show that no critical metrics different from those obtained in the previous sections  may exist under the conditions above. We will make use of Gröbner bases.
	
	As in the previous case, since $a+f+p\neq 0$,  we work with representatives of the homothetic classes satisfying $a+f+p=1$. Moreover, to explicitly express  that $a\neq f\neq p$, $a\neq p$, we introduce additional   variables $\varphi_1$,  $\varphi_2$, $\varphi_3$ and use the polynomials $(a-f) \varphi_1-1$, $(a-p) \varphi_2-1$ and $(f-p) \varphi_3-1$. Recall that, as in the previous sections \ref{case 5.2} and \ref{ref: R3 c=0}, $\mathfrak{\overline F}$ vanishes if and only if the components $\mathfrak{\overline F}(e_1,e_1)$, $\mathfrak{\overline F}(e_1,e_2)$, $\mathfrak{\overline F}(e_1,e_3)$ and $\mathfrak{\overline F}(e_2,e_3)$ are zero. Let   
	$\mathfrak{\overline F}'=
	\{\mathfrak{\overline F}(e_1,e_1),
	\mathfrak{\overline F}(e_1,e_2)$, $\mathfrak{\overline F}(e_1,e_3)$, $\mathfrak{\overline F}(e_2,e_3)\}$ be  the set of these components of $\mathfrak{\overline F}$. We consider the polynomial ring 
	$\mathbb{R}[\varphi_1,\varphi_2,\varphi_3,b,p,h,c,f,a]$ with the lexicographic order and denote by $\mathcal{I}$ the ideal generated by 
	\[
	\mathfrak{\overline F}' \cup\{a+f+p-1, (a-f) \varphi_1-1, (a-p) \varphi_2-1, (f-p) \varphi_3-1\}.
	\]
	Computing a Gröbner basis of    $\mathcal{I}$   we get $161$ polynomials, among which we find
	$$
	\begin{array}{l}
	\mathbf{g}= \frac{1}{3} c h (3a+3f-2) \mathbf{p} \mathbf{q} (3 a^2+3 f^2+3 a f-3 a-3 f+1),
	\end{array}
	$$
	where writing $a-f=\mu_1$ and $a-p=\mu_2$ (which implies $a=\frac{1}{3}(\mu_1+\mu_2+1)$) the factors $\mathbf{p}$ and $\mathbf{q}$ are given by 
	
	\smallskip
	
	\noindent
	$\small
	\begin{array}{l}	
	\mathbf{p} = 
	92 \mu_1^{12} + 92 \mu_2^{12} - 552 \mu_1^{11} \mu_2 + 
	1455 \mu_1^{10} \mu_2^2 - 2215 \mu_1^9 \mu_2^3 + 
	3807 \mu_1^8 \mu_2^4 - 8010 \mu_1^7 \mu_2^5
	\\
	\noalign{\medskip}
	\phantom{\mathbf{p} = }
	+ 
	10938 \mu_1^6 \mu_2^6 - 8010 \mu_1^5 \mu_2^7 + 
	3807 \mu_1^4 \mu_2^8 - 2215 \mu_1^3 \mu_2^9 + 
	1455 \mu_1^2 \mu_2^{10} - 552 \mu_1 \mu_2^{11}-359 \mu_1^{10}
	\\
	\noalign{\medskip}
	\phantom{\mathbf{p}=   }
	- 359 \mu_2^{10} + 1795 \mu_1^9 \mu_2 - 
	4683 \mu_1^8 \mu_2^2 + 7962 \mu_1^7 \mu_2^3 - 
	10539 \mu_1^6 \mu_2^4 + 11289 \mu_1^5 \mu_2^5
	\\
	\noalign{\medskip}
	\phantom{\mathbf{p}=  }
	- 
	10539 \mu_1^4 \mu_2^6 + 7962 \mu_1^3 \mu_2^7 - 
	4683 \mu_1^2 \mu_2^8 + 1795 \mu_1 \mu_2^9
	+402 \mu_1^8 + 402 \mu_2^8- 1608 \mu_1^7 \mu_2
	\\
	\noalign{\medskip}
	\phantom{\mathbf{p} = }
	+ 
	3615 \mu_1^6 \mu_2^2 - 5217 \mu_1^5 \mu_2^3 + 
	6018 \mu_1^4 \mu_2^4 - 5217 \mu_1^3 \mu_2^5
	+ 
	3615 \mu_1^2 \mu_2^6 - 1608 \mu_1 \mu_2^7-135 \mu_1^6 
	\\
	\noalign{\medskip}
	\phantom{\mathbf{p}=   }
	
	- 135 \mu_2^6 + 405 \mu_1^5 \mu_2 - 
	810 \mu_1^4 \mu_2^2+ 945 \mu_1^3 \mu_2^3
	- 
	810 \mu_1^2 \mu_2^4 + 405 \mu_1 \mu_2^5 ,
	
\end{array}
$

\medskip 

\noindent
$\small
\begin{array}{l}
\mathbf{q} =  
4 \mu_1^6 + 4 \mu_2^6 - 12 \mu_1^5 \mu_2 - 12 \mu_1 \mu_2^5 - 
3 \mu_1^4 \mu_2^2 - 3 \mu_1^2 \mu_2^4 + 26 \mu_1^3 \mu_2^3 - 
81 \mu_1^4 - 81 \mu_2^4 
\\
\noalign{\medskip}
\phantom{\mathbf{q} =  }
+ 162 \mu_1^3 \mu_2 + 
162 \mu_1 \mu_2^3 - 243 \mu_1^2 \mu_2^2 + 162 \mu_1^2 + 
162 \mu_2^2 - 162 \mu_1 \mu_2 .
\end{array}
$

\medskip

\noindent 
Note that the vanishing of the last factor in $\mathbf{g}$ implies  $a=f=1/3$, which contradicts the assumption $a\neq f$. Moreover, $ch\neq 0$, so we are led to the cases  $3a+3f-2=0$, $\mathbf{p}=0$ and $\mathbf{q}=0$.

\smallskip

\subsubsection{Case $3a+3f-2=0$} \label{R3: Case 6.3.2}
Since $a+f+p=1$, we get that $p=1/3$. Thus,  the polynomials in $\mathfrak{\overline F}'$ reduce to
\[
\begin{array}{l}
\mathfrak{\overline F}(e_1,e_1) = \tfrac{2}{3}(3a-1)^2 (c^2 - h^2 - 3 b c h),
\\
\noalign{\medskip}
\mathfrak{\overline F}(e_1,e_2) = \tfrac{1}{81} (3a-1)^3 b  (72 a^2 - 72 b^2 + 9 c^2 + 9 h^2 - 48 a - 10),
\\
\noalign{\medskip}
\mathfrak{\overline F}(e_1,e_3) = \tfrac{1}{81} (3a-1)^3 (c (9 a^2 - 9 b^2 + 18 c^2 - 36 h^2 - 6 a - 8) + 54 b h),
\\
\noalign{\medskip}
\mathfrak{\overline F}(e_2,e_3) = -\tfrac{1}{81} (3a-1)^3 ( h (9 a^2 - 9 b^2 - 36 c^2 + 18 h^2 - 6 a - 8) - 54 b c).
\end{array}
\]
A direct calculation shows that 
\[
c h (3a-1)\,  \mathfrak{\overline F}(e_1,e_1)
-h\, \mathfrak{\overline F}(e_1,e_3) 
-c\, \mathfrak{\overline F}(e_2,e_3)
= 
-\tfrac{2}{3} b (3a-1)^3 (c^2(3h^2+1)+h^2).
\]
Note that $b\neq 0$; $3a-1\neq 0$, since $a\neq \frac13=p$; and $c^2(3h^2+1)+h^2\neq 0$, since $c\neq 0$ and $h\neq 0$. So there are no critical metrics in this case.

\subsubsection{Case $\mathbf{p}=0$} 
In the polynomial ring   
$\mathbb{R}[\varphi_1,\varphi_2,\varphi_3,a,f,p,b,c,h]$ with the lexicographic order we  consider the ideal $\mathcal{I}_1$ generated by $\mathcal{I}\cup \{\mathbf{p}\}$ and compute a 
Gröbner basis of    $\mathcal{I}_1$, which contains $111$ polynomials. These polynomials  are quite complicated and it is necessary a more detailed analysis of the basis to solve this case.
Among the  elements of $\mathcal{I}_1$ there are just five which  depend only on $c$ and $h$. We are interested in two of them, which are
\[
\begin{array}{l}
\mathbf{g}_{1}^1 = 
c h P_1(c,h) (5h^2+1)^2(128 h^2+49) S_{1}(h),\,\,\text{and }

\\
\noalign{\medskip}

\mathbf{g}_{1}^2 =
h P_1(c,h) (\alpha_1 c^4+T_{1}^1(h)c^2+T_{1}^2(h)),
\end{array}
\]
where   $P_1(c,h)$ is a symmetric polynomial, 
$\alpha_1>0$,
and $S_1(h)$, $T_{1}^1(h)$ and $T_{1}^2(h)$ are polynomials with only even  powers of $h$. Recall that  $c\neq 0$ and $h\neq 0$. The polynomial  $P_1(c,h)$ vanishes if and only if $c=h=0$, so $P_1(c,h)\neq 0$.   Hence, $S_1(h)$ and $\alpha_1 c^4+T_{1}^1(h)c^2+T_{1}^2(h)$ must vanish. Finally, we consider the ideal $\mathcal{J}_1$ generated by 
$\{S_1(h),\alpha_1 c^4+T_{1}^1(h)c^2+T_{1}^2(h)\}$ in the polynomial ring $\mathbb{R}[c,h]$. We compute  a Gröbner basis with the graded reverse  lexicographic order and obtain $6$ polynomials, among which we have  
\small\[
\begin{array}{l}
\mathbf{g}_{_{\mathcal{J}_1}} = 
247878727200 c^8
+3 (12047491687221 h^4 + 13004016701628 h^2 + 4019608368416) c^6 

\\
\noalign{\medskip}
\phantom{\mathbf{g}}

+ (36142475061663 h^6 - 16246524285081 h^4 + 44177371142781 h^2 + 
27149520949987) c^4

\\
\noalign{\medskip}
\phantom{\mathbf{g}}

+ (39012050104884 h^6 + 44177371142781 h^4 + 36467803675675 h^2 + 
8034850546195) c^2

\\
\noalign{\medskip}
\phantom{\mathbf{g}}

+12058825105248 h^6 + 247878727200 h^8 + 27149520949987 h^4 + 
8034850546195 h^2 
\\
\noalign{\medskip}
\phantom{\mathbf{g}}

+ 638325404640 .
\end{array}
\]\normalsize
Clearly, this polynomial has no real roots, so  we conclude that no critical metric may exist in this case.

\subsubsection{Case $\mathbf{q}=0$} 
Let $\mathcal{I}_2$ be the ideal  generated by 
$\mathfrak{\overline F}'\cup\{a+f+p-1, \mathbf{q}\}
$ in the polynomial ring 
$\mathbb{R}[p,a,f,b,c,h]$.  
We compute  a Gröbner basis $\mathcal{B}_2$ of    $\mathcal{I}_2$ with respect to the lexicographic order, obtaining  $85$ polynomials, and see that
$$\small
\begin{array}{l}
\mathbf{g}_{2}^1=  c^2 h^2  (3f-1)^3 (12h^2+1)^2 
\left((b^2 - c^2)^2 + (b^2 - h^2)^2 - (b^2 - c^2) (b^2 - h^2)\right)
\mathbf{p}_{2}^1 \,\mathbf{q}_{2}^1
\end{array}
$$
belongs to the basis, where

\medskip

\noindent
$\small
\begin{array}{l}
\mathbf{p}_{2}^1=
144 (c^8 + h^8)
+ 84 (c^6 + h^6) (9 c^2 h^2 + 1) 
+ (c^4 + h^4) (1296 c^4 h^4 + 153 c^2 h^2 + 16)

\\
\noalign{\medskip}
\phantom{\mathbf{p}_{2}^1=}

+ (c^2 + h^2) (9 c^2 h^2 + 1) (81 c^4 h^4 - 102 c^2 h^2 + 1)
- c^2 h^2 (1377 c^4 h^4 + 864 c^2 h^2 + 17) ,

\end{array}
$

\medskip

\noindent
$\small
\begin{array}{l}
\mathbf{q}_{2}^1 =
13089195000 h^{18} - 150198512625 h^{16} + 575974443600 h^{14} - 
464946766740 h^{12}

\\
\noalign{\medskip}
\phantom{\mathbf{ q}_{2}^1 =}

- 1169010941202 h^{10} + 1352051494317 h^8 + 
531977212062 h^6 - 89013307725 h^4

\\
\noalign{\medskip}
\phantom{\mathbf{q}_{2}^1= }

+ 3173480100 h^2 - 203889500 .
\end{array}
$

\medskip

\noindent
Recall that  $c\neq 0$ and $h\neq 0$. By Remark~\ref{re:iso-R3}, the case  $3f-1=0$ can be reduced to the case in~$\S\ref{R3: Case 6.3.2}$ ($3a+3f-2=0$). Hence, we have three different possibilities corresponding to the last three factors in $\mathbf{g}_{2}^1$:
\begin{itemize}
\item[(i)] If $(b^2 - c^2)^2 + (b^2 - h^2)^2 - (b^2 - c^2) (b^2 - h^2)=0$ then $b^2=c^2=h^2$ and, under these conditions, we have 
\[
\mathfrak{\overline F}(e_1,e_1) =
-2 b c  h \left( (3 a - 1)^2 + (3 f - 1)^2 + (3 a - 1) (3 f - 1)\right)  .
\]
Since $ b c  h\neq 0$, it follows that
$a=f=1/3$, which is not possible since we are assuming $a\neq f$.

\smallskip

\item[(ii)]  If  $\mathbf{p}_{2}^1=0$, then the study of this  symmetric polynomial shows that the only real solutions are given by $c=h=0$ and $c^2=h^2=1/3$, but $ch\neq 0$.
Hence, in the polynomial ring  $\mathbb{R}[p,a,f,b,c,h]$ with the lexicographic order we consider the ideal $\mathcal{I}_{21}=\langle \mathcal{I}_2\cup \{3c^2-1, 3h^2-1 \}\rangle$ and compute a Gröbner basis. We get $9$ polynomials, being one of them  
\[
\mathbf{g}_{21}=(3b^2-1)(3b^2+1)(3f-1)^3.
\]
Thus, it follows that $b^2=1/3$ and this case reduces to the previous one.

\smallskip

\item[(iii)]  Finally, if $\mathbf{ q}_{2}^1=0$, a detailed analysis of the Gröbner basis $\mathcal{B}_2$ of $\mathcal{I}_2$ shows that there exists  just one polynomial with the form
\[
\mathbf{g}_{2}^2 = c h^2  (3f-1)^3  
\left((b^2 - c^2)^2 + (b^2 - h^2)^2 - (b^2 - c^2) (b^2 - h^2)\right)
\mathbf{p}_{2}^1 Q_2(c,h),
\]
where   $Q_2(c,h)=(\alpha_2 c^4+T_{2}^1(h)c^2+T_{2}^2(h))$,  $\alpha_2>0$ and with  $T_{2}^1(h)$ and $T_{2}^2(h)$  polynomials with only even  powers of $h$.
Excluding the factors previously considered, we have that     $Q_2(c,h)$  must vanish.
Now,  we compute a Gröbner basis  for the ideal $\mathcal{J}_{2}=\langle \{\mathbf{ q}_{2}^1, Q_2(c,h)\}\rangle$ in the polynomial ring $\mathbb{R}[c,h]$ with the graded reverse lexicographic order. As a result we get $6$ polynomials, which include the following
\[
\small
\begin{array}{l}

\mathbf{g}_{_{\mathcal{J}_2}} =
-162741604855409038884600 c^8

\\
\noalign{\medskip}
\phantom{\mathbf{g}_{_{\mathcal{J}_2}} =}

+15  (253026638662278691056852 h^4 -
383142569625180082065192 h^2 

\\
\noalign{\medskip}
\phantom{\mathbf{g}_{_{\mathcal{J}_2}} =}

+   57044648695561578739735) c^6

\\
\noalign{\medskip}
\phantom{\mathbf{g}_{_{\mathcal{J}_2}} =}

+ (3795399579934180365852780 h^6 - 
12444355698616995995236734 h^4 

\\
\noalign{\medskip}
\phantom{\mathbf{g}_{_{\mathcal{J}_2}} =}

+ 441593778374609756435607 h^2 + 
248544510738030291383903) c^4

\\
\noalign{\medskip}
\phantom{\mathbf{g}_{_{\mathcal{J}_2}} =}

- (5747138544377701230977880 h^6 - 441593778374609756435607 h^4 

\\
\noalign{\medskip}
\phantom{\mathbf{g}_{_{\mathcal{J}_2}} =}

- 16582434176665032560612153 h^2 + 2739047229386473655704675) c^2

\\
\noalign{\medskip}
\phantom{\mathbf{g}_{_{\mathcal{J}_2}} =}

-162741604855409038884600 h^8 + 855669730433423681096025 h^6 

\\
\noalign{\medskip}
\phantom{\mathbf{g}_{_{\mathcal{J}_2}} =}

+ 248544510738030291383903 h^4 - 
2739047229386473655704675 h^2 

\\
\noalign{\medskip}
\phantom{\mathbf{g}_{_{\mathcal{J}_2}} =}

- 60336068254303466865600  .

\end{array}
\]
Despite the length of this polynomial, it is straightforward to check that it never vanishes for the only two real  solutions, $h=\pm  0,33844287\dots$, of $\mathbf{q}_{2}^1=0$. Thus,  no critical metric may exist in this case. 
\end{itemize} 
\end{proof}

\begin{remark}\rm 
The metric in Theorem~\ref{th:criticas-R3}--(1) corresponds to the algebraic Ricci soliton (i) in Section~\ref{ss:1-3-2}. The left-invariant metric in Theorem~\ref{th:criticas-R3}--(2) is the algebraic Ricci soliton (ii), while   metrics  in Theorem~\ref{th:criticas-R3}--(3)   correspond to (iii) in Section~\ref{ss:1-3-2}.
\end{remark}

As a consequence of the previous analysis we have the following relation between Ricci solitons and critical metrics with zero energy.

\begin{corollary}\label{cor:hoy}
	Let $(M,g)$ be a connected and simply connected homogeneous four-dimensional manifold. Then $g$ is critical for a quadratic curvature functional with zero energy if and only if it is a Ricci soliton or, otherwise, it is homothetic to a metric in Families (4) or (5) in Theorem~\ref{th:criticas-R3}.
\end{corollary}

\begin{proof}
 In view of  Theorems~\ref{th:non-solvable}, \ref{th:E(1,1)},  \ref{th:heisenberg} and \ref{th:criticas-R3}, the critical metrics with zero energy which are not algebraic Ricci solitons are homothetic to a metric in Families~(4) or (5) in Theorem~\ref{th:criticas-R3} (see Figures~\ref{figura-3} and \ref{figura-R}).
It remains to show that metrics in these two families are not Ricci solitons, so we check that they do not belong to the homothetic class of an algebraic Ricci soliton. 
It follows from Section~\ref{ss:1-3-2} that non-Einstein Ricci solitons are $\mathcal{F}_t$-critical for $t=-1/2$ or $t=-1/3$ in the symmetric case. Moreover, they are  $\mathcal{F}_t$-critical for $t=-3$, $t=-3/2$, or for  $-1\leq t<-1/4$ otherwise  (see Figure~\ref{figura-2}).

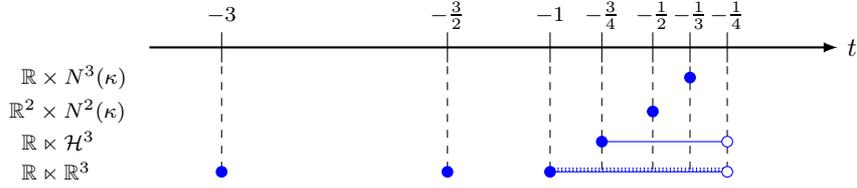
\begin{figure}[h]
	\centering
	\begin{tikzpicture}
	\draw[thick,-latex]   (-4.4-0.15,0) -- (4.50,0) node[right] {$t$}; 
			
	\foreach \i/\n in 
	{-3.6,  -0.63, 0.72, 1.40, 2.07, 2.56, 3.05
	}{\draw (\i,-.2)--(\i,.2);}
	
	\foreach \i/\n in 
	{-3.6/$-3$,    0.72/$-1$ 
	}{\draw  (\i,.2) node[above] {\footnotesize \n};}
	\foreach \i/\n in 
	{-0.63/$-\frac{3}{2}$, 1.40/$-\frac{3}{4}$, 2.07/$-\frac{1}{2}$, 2.56/$-\frac{1}{3}$, 3.05/$-\frac{1}{4}$
	}{\draw  (\i,.12) node[above] {\footnotesize \n};}
	
	\foreach \i/\n in 
	{-3.6,  -0.63, 0.72, 1.40, 2.07, 2.56, 3.05
	}{\draw[dashed] (\i,-.25)--(\i,-1.65);} 
	
	\node at (-5.75,-0.4) {\footnotesize \phantom{.....}$\mathbb{R}\times N^3(\kappa)$ }; 
	\filldraw [color=\ColorECeroARS,fill=\ColorECeroARS] (2.56,-0.4) circle (2pt);
	
	\node at (-5.75,-0.85) {\footnotesize \phantom{...}$\mathbb{R}^2\times N^2(\kappa)$ }; 
	\filldraw [color=\ColorECeroARS,fill=\ColorECeroARS] (2.07,-0.85) circle (2pt);
	
	\node at (-5.75,-1.25) {\footnotesize $\mathbb{R}\ltimes \mathcal{H}^3$ }; 
	\draw[color=\ColorECeroARS] (1.4,-1.25)--(3.05,-1.25);
	\filldraw [color=\ColorECeroARS,fill=\ColorECeroARS] (1.4,-1.25) circle (2pt);
	\filldraw [color=\ColorECeroARS,fill=white] (3.05,-1.25) circle (2pt);
	
	\node at (-5.75,-1.65) {\footnotesize $\mathbb{R}\ltimes \mathbb{R}^3$\,}; 
	
	\filldraw [color=\ColorECeroARS,fill=\ColorECeroARS] (-3.6,-1.65) circle (2pt);
	
	\filldraw [color=\ColorECeroARS,fill=\ColorECeroARS] (-0.63,-1.65) circle (2pt);
	
	\draw[color=\ColorECeroARS] (0.72,-1.65)--(3.05,-1.65);
	\draw[densely dotted,color=\ColorECeroARS] 	
	 (0.72,-1.60)--(3.05,-1.60);
	 \draw[densely dotted,color=\ColorECeroARS] 	
	 (0.72,-1.61)--(3.05,-1.61);
	\draw[densely dotted,color=\ColorECeroARS] 	
	 (0.72,-1.64)--(3.05,-1.64);
	\draw[densely dotted,color=\ColorECeroARS] 	
	 (0.72,-1.63)--(3.05,-1.63);
	\filldraw [color=\ColorECeroARS,fill=\ColorECeroARS] (0.72,-1.65) circle (2pt);
	\filldraw [color=\ColorECeroARS,fill=white] (3.05,-1.65) circle (2pt);
	\end{tikzpicture}
	\caption{Range of the critical parameter $t=-\tau^{-2}\|\rho\|^2$ for non-Einstein homogeneous Ricci solitons.}\label{figura-2}
\end{figure}

Metrics in Theorem~\ref{th:criticas-R3}--(4) are not  Ricci solitons since the corresponding values of the parameter $t$ do not lie in the range of critical values for algebraic Ricci solitons.
For metrics in Theorem~\ref{th:criticas-R3}--(5) we only have to analyze the case of $\mathcal{F}_{-3}$-criticality, which corresponds   to  $p=\frac{1}{16}(1-\sqrt{33})$. For this particular value of $p$, the third order homothetic invariant $\frac{\|\nabla\rho\|^2}{\tau^3}$ is given by $\frac{\|\nabla\rho\|^2}{\tau^3}=-\frac{776}{81}$.
On the other hand, algebraic Ricci solitons with $t=-3$ correspond to metrics in    $\mathbb{R}\ltimes\mathbb{R}^3$ given by Theorem~\ref{th:criticas-R3}--(1), in which case the same homothetic invariant is given by $\frac{\|\nabla\rho\|^2}{\tau^3}=-8$. This shows that metrics in Theorem~\ref{th:criticas-R3}--(5) with  $p=\frac{1}{16}(1-\sqrt{33})$ are not Ricci solitons either.
\end{proof}

\begin{remark}\rm
	Nilpotent Lie groups whose structure constants are rational admit compact quotients (nilmanifolds). This is, for example, the case of Families~(1) and (3) in Theorem~\ref{th:heisenberg}, and  Families~(1) and (2) in Theorem~\ref{th:criticas-R3}, where the metrics descend  to  critical metrics with zero energy in the nilmanifolds. 
	
	In the more general solvable case, the existence of compact lattices is a subtle question. It was shown in \cite{Milnor} that the Lie group must be unimodular in order to admit a lattice.
	Non-nilpotent solvable Lie groups admitting compact quotients correspond to the following Lie algebras (see \cite{Bock}, and also \cite{Ovando} for the notation)
	$\mathfrak{r}\mathfrak{r}_{3,-1}$, 
	$\mathfrak{r}\mathfrak{r}'_{3,0}$,
	$\mathfrak{r}_{4,\alpha,-(1+\alpha)}$,
	$\mathfrak{r}'_{4,-\frac{1}{2},\delta}$,
	$\mathfrak{d}_4$, and $\mathfrak{d}'_{4,0}$.
	Notice that the Lie groups associated to the Lie algebras above do not admit a lattice for all the values of the parameters, but they do for some of them.
	For instance, unimodular groups of type $\mathbb{R}\ltimes\mathbb{R}^3$ as in Theorem~\ref{th:criticas-R3}--(3) with $p=-(f+1)$ admit compact lattices if $f\neq 1$ and $e$, $e^f$, $e^p$ are solutions of $\lambda^3-m\lambda^2+n\lambda-1=0$ with $m,n\in\mathbb{N}$ \cite{Wall}. It was shown in \cite{LLSY} that this is the case for a countable number of $f$'s.  If $f=1$, $\mathbb{R}\ltimes\mathbb{R}^3$ does not admit a lattice \cite{LLSY}.
	
	It follows from Theorem~\ref{th:heisenberg} and Theorem~\ref{th:criticas-R3} that there are no $\mathcal{F}_t$-critical metrics on $\mathfrak{d}_4$ and $\mathcal{F}_t$-critical metrics on $\mathfrak{r}\mathfrak{r}'_{3,0}$ are flat. Furthermore $\mathfrak{d}'_{4,0}$ and $\mathfrak{r}'_{4,-\frac{1}{2},\delta}$ admit $\mathcal{F}_t$-critical metrics which are isometric to $\mathcal{F}_t$-critical metrics on $\mathfrak{rh}_3$ and $\mathfrak{r}_{4,-\frac{1}{2},-\frac{1}{2}}$, respectively. Therefore the possible solvmanifolds corresponding to the solvable Lie algebras
	$\mathfrak{rh}_3$, $\mathfrak{n}_4$,
	$\mathfrak{r}\mathfrak{r}_{3,-1}$, 
	$\mathfrak{r}_{4,\alpha,-(1+\alpha)}$,
	$\mathfrak{r}'_{4,-\frac{1}{2},\delta}$, and $\mathfrak{d}'_{4,0}$ admit non-symmetric $\mathcal{F}_t$-critical metrics with zero energy for $t=-3$, $t=-3/2$, and $t=-1$.
	
	We emphasize that although the left-invariant metrics descend to the quotient manifolds, the Ricci soliton vector fields determined by the algebraic Ricci soliton structure do not pass to the quotient. The relation between Ricci solitons and $\mathcal{F}_t$-critical metrics with zero energy is therefore not valid in the non simply connected case.
\end{remark}

\end{document}